\documentclass[11pt]{article}

\usepackage{amsfonts,exscale}

\usepackage{srcltx,hyperref}

\long\def\forget#1{}

\usepackage{color}
\newcounter{commentcounter}

\def\?{\ 
{\bf\color{red}???}\ 
\immediate\write16{}
\immediate\write16{Warning: There was still a question mark . . . }
\immediate\write16{}}

\usepackage{geometry}
\usepackage{amsmath}
\usepackage{amssymb}
\usepackage{amscd}
\usepackage{textcomp}
\usepackage{mathrsfs}

\usepackage{amsthm}
\theoremstyle{plain}
\newtheorem{Lemma}{Lemma}[section]
\newtheorem{Theorem}[Lemma]{Theorem}
\newtheorem{Proposition}[Lemma]{Proposition}
\newtheorem{Corollary}[Lemma]{Corollary}

\theoremstyle{definition}
\newtheorem{Definition}[Lemma]{Definition}
\newtheorem{Example}[Lemma]{Example}
\newtheorem{Remark}[Lemma]{Remark}
\newtheorem{Notation}[Lemma]{Notation}

\usepackage{xy}
\xyoption{all}

\newdir^{ (}{{}*!/-3pt/\dir^{(}}    
\newdir^{  }{{}*!/-3pt/\dir^{}}    
\newdir_{ (}{{}*!/-3pt/\dir_{(}}    
\newdir_{  }{{}*!/-3pt/\dir_{}}

\def\theenumi{(\alph{enumi})}

\def\p@enumii{\theenumi}

\newcommand{\DS}{\displaystyle}
\newcommand{\TS}{\textstyle}
\newcommand{\SC}{\scriptstyle}
\newcommand{\SSC}{\scriptscriptstyle}

\newcounter{zahl}

\DeclareMathOperator{\DGr}{DGr}

\DeclareMathOperator{\End}{End}
\DeclareMathOperator{\Ext}{Ext}
\DeclareMathOperator{\Frob}{Frob}

\DeclareMathOperator{\GL}{GL}
\DeclareMathOperator{\Koh}{H}
\DeclareMathOperator{\CKoh}{\check H}
\DeclareMathOperator{\Hom}{Hom}
\newcommand{\CHom}{{\cal H}om}

\DeclareMathOperator{\Inf}{Inf}

\DeclareMathOperator{\pr}{pr}

\DeclareMathOperator{\Quot}{Frac}

\DeclareMathOperator{\Spec}{Spec}
\DeclareMathOperator{\Spf}{Spf}

\DeclareMathOperator{\Sym}{Sym}
\DeclareMathOperator{\Tor}{Tor}

\DeclareMathOperator{\Var}{V}

\DeclareMathOperator{\coker}{coker}

\newcommand{\et}{{\rm\acute{e}t}}
\newcommand{\fppf}{{\it fppf\/}}

\DeclareMathOperator{\gr}{gr}

\DeclareMathOperator{\height}{height}
\DeclareMathOperator{\id}{\,id}
\DeclareMathOperator{\im}{im}

\renewcommand{\mod}{\;{\rm mod}\;}
\newcommand{\nil}{{\rm nil}}
\DeclareMathOperator{\ord}{ord}

\newcommand{\red}{{\rm red}}

\DeclareMathOperator{\rk}{rk}
\newcommand{\semi}{{\rm semi}}

\let\setminus\smallsetminus

\newcommand{\es}{\enspace}

\newcommand{\dual}{^{\SSC^\vee}}

\newcommand{\mal}{^{\SSC\times}}

\newcommand{\dbl}{{\mathchoice{\mbox{\rm [\hspace{-0.15em}[}}
                              {\mbox{\rm [\hspace{-0.15em}[}}
                              {\mbox{\scriptsize\rm [\hspace{-0.15em}[}}
                              {\mbox{\tiny\rm [\hspace{-0.15em}[}}}}
\newcommand{\dbr}{{\mathchoice{\mbox{\rm ]\hspace{-0.15em}]}}
                              {\mbox{\rm ]\hspace{-0.15em}]}}
                              {\mbox{\scriptsize\rm ]\hspace{-0.15em}]}}
                              {\mbox{\tiny\rm ]\hspace{-0.15em}]}}}}
\newcommand{\dpl}{{\mathchoice{\mbox{\rm (\hspace{-0.15em}(}}
                              {\mbox{\rm (\hspace{-0.15em}(}}
                              {\mbox{\scriptsize\rm (\hspace{-0.15em}(}}
                              {\mbox{\tiny\rm (\hspace{-0.15em}(}}}}
\newcommand{\dpr}{{\mathchoice{\mbox{\rm )\hspace{-0.15em})}}
                              {\mbox{\rm )\hspace{-0.15em})}}
                              {\mbox{\scriptsize\rm )\hspace{-0.15em})}}
                              {\mbox{\tiny\rm )\hspace{-0.15em})}}}}
\newcommand{\invlim}[1][]{\ifthenelse{\equal{#1}{}}
{\DS \lim_{\longleftarrow}}
{\DS \lim_{\underset{#1}{\longleftarrow}}}
}
\newcommand{\dirlim}[1][]{\ifthenelse{\equal{#1}{}}
{\DS \lim_{\longrightarrow}}
{\DS \lim_{\underset{#1}{\longrightarrow}}}
}
\newcommand{\ul}[1]{{\underline{#1}}}
\newcommand{\ol}[1]{{\overline{#1}}}

\newcommand{\wt}[1]{{\widetilde{#1}}}
\newcommand{\wh}[1]{{\widehat{#1}}}

\newcommand{\Balpha}{\mbox{$\hspace{0.12em}\shortmid\hspace{-0.62em}\alpha$}}
\usepackage{xcolor}
\usepackage{graphicx}
\newcommand{\Bmu}{\mbox{$\raisebox{-0.59ex}{$l$}\hspace{-0.16em}\mu\hspace{-0.91em}\raisebox{-0.95ex}{\scalebox{2}{$\color{white}.$}}\hspace{-0.59em}\raisebox{+0.78ex}{\scalebox{2}{$\color{white}.$}}\hspace{0.46em}$}{}} 

\newcommand{\BOne} {{\mathchoice{\hbox{\rm1\kern-2.7pt l\kern.9pt}}
                              {\hbox{\rm1\kern-2.7pt l\kern.9pt}}
                              {\hbox{\scriptsize\rm1\kern-2.3pt l\kern.4pt}}
                              {\hbox{\scriptsize\rm1\kern-2.4pt l\kern.5pt}}}}
\def\UOne{\underline{\BOne}}

\newcommand{\BF}{{\mathbb{F}}}
\newcommand{\BG}{{\mathbb{G}}}

\newcommand{\BN}{{\mathbb{N}}}

\newcommand{\BP}{{\mathbb{P}}}

\newcommand{\BZ}{{\mathbb{Z}}}

\newcommand{\CA}{{\cal{A}}}

\newcommand{\CalD}{{\cal{D}}}

\newcommand{\CF}{{\cal{F}}}
\newcommand{\CG}{{\cal{G}}}
\newcommand{\CH}{{\cal{H}}}

\newcommand{\CL}{{\cal{L}}}

\newcommand{\CN}{{\cal{N}}}
\newcommand{\CO}{{\cal{O}}}

\newcommand{\FG}{{\mathfrak{G}}}

\newcommand{\Fa}{{\mathfrak{a}}}

\newcommand{\Fp}{{\mathfrak{p}}}

\renewcommand{\epsilon}{\varepsilon}
\renewcommand{\phi}{\varphi}

\def\longto{\longrightarrow}
\def\into{\hookrightarrow}
\let\onto\twoheadrightarrow
\def\longonto{\mbox{$\kern2pt\longto\kern-8pt\to\kern2pt$}}
\def\isoto{\stackrel{}{\mbox{\hspace{1mm}\raisebox{+1.4mm}{$\SC\sim$}\hspace{-3.5mm}$\longrightarrow$}}}

\newbox\mybox
\def\arrover#1{\mathrel{
       \setbox\mybox=\hbox spread 1.4em{\hfil$\scriptstyle#1$\hfil}
       \vbox{\offinterlineskip\copy\mybox
             \hbox to\wd\mybox{\rightarrowfill}}}}

\newcommand{\CoCI}[1]{L^{^\bullet}_{#1}}
\newcommand{\CoCLS}[1]{L^{^\bullet}_{\rm LS}(#1)}
\newcommand{\CoCM}[1]{L^{^\bullet}_{\rm ME}(#1)}
\newcommand{\CoL}[1]{\ell^{\,^\bullet}_{#1}}
\newcommand{\CoLA}[1]{\ell^{\,^\bullet}_{#1}}
\newcommand{\CoLI}[1]{\ell^{\,^\bullet}_{#1}}

\DeclareMathOperator{\Nilp}{\CN \!{\it ilp}}

\DeclareMathOperator{\Gr}{{\rm Gr}}
\DeclareMathOperator{\Dr}{{\rm Dr}}
\DeclareMathOperator{\CDr}{\CalD{\it r}}
\DeclareMathOperator{\DSch}{DSch}
\DeclareMathOperator{\FqSht}{\mbox{$\BF_q$-Sht}}
\def\ulE{{\underline{E\!}\,}}
\def\ulM{{\underline{M\!}\,}}
\def\ulN{{\underline{N\!}\,}}
\def\ulCN{{\underline{\CN\!}\,}}
\def\ulX{{\underline{X\!}\,}}

\def\olG{{\,\overline{\!G}}}

\def\olR{{\,\overline{\!R}}}

\newcommand{\inclN}{n}
\newcommand{\charmorph}{c}

\begin{document}
\author{Urs Hartl, Rajneesh Kumar Singh}
\title{Local Shtukas and Divisible Local Anderson Modules}
\maketitle

\begin{abstract}
We develop the analog of crystalline Dieudonn\'e theory for $p$-divisible groups in the arithmetic of function fields. In our theory $p$-divisible groups are replaced by divisible local Anderson modules, and Dieudonn\'e modules are replaced by local shtukas. We show that the categories of divisible local Anderson modules and of effective local shtukas are anti-equivalent over arbitrary base schemes. We also clarify their relation with formal Lie groups and with global objects like Drinfeld modules, Anderson's abelian $t$-modules and $t$-motives, and Drinfeld shtukas. Moreover, we discuss the existence of a Verschiebung map and apply it to deformations of local shtukas and divisible local Anderson modules. As a tool we use Faltings's and Abrashkin's theory of strict modules, which we review to some extent.
 
\end{abstract}

\tableofcontents
\bigskip

This article was published as \cite{HartlSingh}. The present arXiv version of this article contains a few more details most notably Lemmas~\ref{LemmaNBLocFree}, \ref{LemmaExistsLift}, \ref{LemmaStrictOAct} and \ref{LemmaEtaleIsStrict}, Remark~\ref{RemarkFaltingsAbrashkin}, Example~\ref{ExampleNotBounded}, Corollary~\ref{CorCanonDecompZDiv}, and Appendix~\ref{AppCotCom}.

%
%

\section{Introduction}

In the arithmetic of number fields elliptic curves and abelian varieties are important objects. Their theory has been vastly developed in the last two centuries and their moduli spaces have played a major role in Faltings's proof of the Mordell conjecture \cite{Faltings,Cornell-Silverman}, the proof of Fermat's Last Theorem by Wiles and Taylor \cite{Wiles,Taylor-Wiles, FLT}, and the proof of the Langlands correspondence for $\GL_n$ over non-archimedean local fields of characteristic zero by Harris and Taylor~\cite{Harris-Taylor}. A useful tool to study abelian varieties and their moduli spaces are $p$-divisible groups. More precisely, for an elliptic curve or an abelian variety $E$ over a $\BZ_p$-algebra $R$ the \emph{$p$-divisible group} $E[p^\infty]=\dirlim E[p^n]$, also called \emph{Barsotti-Tate group}, captures the local $p$-adic information of $E$. One reason why $E[p^\infty]$ is a useful tool to study $E$ is that the complicated arithmetic data of a $p$-divisible group over a $\BZ_p$-algebra $R$ in which $p$ is nilpotent can be faithfully encoded by an object of semi-linear algebra, its Dieudonn\'e module.

Elliptic curves and abelian varieties have analogs in the arithmetic of function fields. Namely, Drinfeld~\cite{Drinfeld, Drinfeld87} invented the notions of \emph{elliptic modules} (today called \emph{Drinfeld modules}) and the dual notion of \emph{$F$-sheaves} (today called \emph{Drinfeld shtukas}). These structures are function field analogs of elliptic curves in the following sense. Their endomorphism rings are rings of integers in global function fields of positive characteristic or orders in central division algebras over the later. On the other hand, their moduli spaces are varieties over smooth curves over a finite field. Through these two aspects in which global function fields of positive characteristic come into play, Drinfeld shtukas and variants of them proved to be fruitful for establishing large parts of the Langlands program over local and global function fields of positive characteristic in works by Drinfeld~\cite{Drinfeld, Drinfeld77, Drinfeld87}, Laumon, Rapoport, and Stuhler~\cite{LRS}, L.~Lafforgue~\cite{Lafforgue02} and V.~Lafforgue~\cite{Lafforgue18}. Beyond this the analogy between Drinfeld modules and elliptic curves is abundant. 

In this spirit, Anderson~\cite{Anderson} introduced higher dimensional generalizations of Drinfeld modules, called \emph{abelian $t$-modules}. These are group schemes which carry an action of the polynomial ring $\BF_r[t]$ over a finite field $\BF_r$ with $r$ elements subject to certain conditions. Abelian $t$-modules are the function field analogs of abelian varieties; see for example \cite{BH2}. Although Anderson worked over a field, abelian $t$-modules also exist naturally over arbitrary $\BF_r[t]$-algebras $R$ as base rings; see Definition~\ref{DefAndModule}. They possess an (anti-)equivalent description by semi-linear algebra objects called \emph{$t$-motives}, which are $R[t]$-modules together with a Frobenius semi-linear endomorphism, see Definition~\ref{DefAMotive} and Theorem~\ref{ThmMotiveOfAModule}, and are a variant and generalization of Drinfeld shtukas. Through the work of Drinfeld and Anderson it was realized very early on that a Drinfeld module or abelian $t$-module over a field is completely described by its $t$-motive. The same is true over an arbitrary $\BF_r[t]$-algebra $R$, as is shown for example in \cite{HartlIsog}. So in a way the situation in function field arithmetic is much better than in the arithmetic of abelian varieties: the $t$-motive is a ``global'' Dieudonn\'e module which integrates the ``local'' Dieudonn\'e modules for every prime in a single object.

Correspondingly it is not difficult to come up with a definition of a ``Dieudonn\'e module'' at a prime $\Fp\subset\BF_r[t]$ of an abelian $t$-module: it should arise as the $\Fp$-adic completion of its $t$-motive; see Example~\ref{ExTorsionOfAModule}(b) for details. The object one ends up with is an effective local shtuka. To define these let $\Fp=(z)$ for a monic irreducible polynomial $z\in\BF_r[t]$ and let $\BF_q=\BF_r[t]/\Fp$ be the residue field. Then $\invlim\BF_r[t]/\Fp^n=\BF_q\dbl z\dbr$. Let $R$ be an $\BF_q\dbl z\dbr$-algebra in which the image $\zeta$ of $z$ is nilpotent. An \emph{effective local shtuka} over $R$ is a pair $\ulM=(M,F_M)$ consisting of a locally free $R\dbl z\dbr$-module $M$ of finite rank, and an isomorphism $F_M\colon\sigma_{\!q}^\ast M[\frac{1}{z-\zeta}] \isoto M[\frac{1}{z-\zeta}]$ with $F_M(\sigma_{\!q}^*M)\subset M$. Here $\sigma_{\!q}^*$ is the endomorphism of $R\dbl z\dbr$ which extends the $q$-Frobenius endomorphism $\sigma_{\!q}^*:=\Frob_{q,R}\colon b\mapsto b^q$ for $b\in R$ by $\sigma_{\!q}^*(z)=z$, and $\sigma_{\!q}^*M:=M\otimes_{R\dbl z\dbr,\sigma_{\!q}^*}R\dbl z\dbr$. Now the goal of crystalline Dieudonn\'e theory in the arithmetic of function fields is to describe the analogs of $p$-divisible groups which correspond to effective local shtukas. In the present article we call them $z$-divisible local Anderson modules as in the following definition, and we develop this theory under the technical assumption that $\zeta\in R$ is nilpotent. This theory was already announced in \cite{HartlAbSh,HartlDict,HartlPSp,HartlKim} and is used in \cite{HartlIsog}.

\medskip\noindent
{\bfseries Definition~\ref{DefZDivGp}.}
A \emph{$z$-divisible local Anderson module over $R$} is a sheaf of $\BF_q\dbl z\dbr$-modules $G$ on the big \fppf-site of $\Spec R$ such that
\begin{enumerate}
\item
$G$ is \emph{$z$-torsion}, that is $G = \dirlim G[z^n]$, where $G[z^n]:=\ker(z^n\colon G\to G)$,
\item
$G$ is \emph{$z$-divisible}, that is $z\colon G \to G$ is an epimorphism,
\item \label{DefInIntro_C}
For every $n$ the $\BF_q$-module $G[z^n]$ is representable by a finite locally free strict $\BF_q$-module scheme over $R$ in the sense of Faltings (Definition~\ref{DefStrictF_qMod}), and
\item
locally on $\Spec R$ there exists an integer $d \in \BZ_{\geq 0}$, such that
$(z-\zeta)^d=0$ on $\omega_ G$  where $\omega_G := \invlim\omega_{G[z^n]}$ and $\omega_{G[z^n]}:=\epsilon^*\Omega^1_{G[z^n]/\Spec R}$ for the unit section $\epsilon$ of $G[z^n]$ over $R$.
\end{enumerate}

\medskip

Such objects were studied in the special case with $d=1$ in work of Drinfeld~\cite{Drinfeld76}, Genestier~\cite{Genestier}, Laumon~\cite{Laumon}, Taguchi~\cite{Taguchi93} and Rosen~\cite{Rosen}. Generalizations for $d>1$ and their semi-linear algebra description by the analog of Dieudonn\'e theory were attempted by the first author in \cite[Definition~6.2]{HartlAbSh} and by W.~Kim \cite[Definition~7.3.1]{Kim}. But unfortunately both definitions and the statements about the analog of Dieudonn\'e theory \cite[Theorem~7.2]{HartlAbSh} and \cite[Theorem~7.3.2]{Kim} are wrong. The problem lies in the fact that the strictness assumption from \ref{DefInIntro_C} is missing. Our above definition corrects this error. It generalizes Anderson's~\cite[\S\,3.4]{Anderson93} definition of formal $t$-modules who considered the case where the $G[z^n]$ are radicial and $G$ is a formal $\BF_q\dbl z\dbr$-module in the following sense. 

\begin{Definition}\label{DefFormalModule}
In this article we define a \emph{formal $\BF_q\dbl z\dbr$-module} over an $\BF_q$-scheme $S$ to be a formal Lie group $G$ equipped with an action of $\BF_q\dbl z\dbr$. In particular, we do \emph{not} impose a condition for the $\BF_q\dbl z\dbr$-action on $\omega_G$.
\end{Definition}

The description of $z$-divisible local Anderson modules by effective local shtukas is deduced from Abrashkin's \cite{Abrashkin} anti-equivalence between finite locally free strict $\BF_q$-module schemes over $\Spec R$ and \emph{finite $\BF_q$-shtukas}. The latter are pairs $(M,F_M)$ consisting of a locally free $R$-module $M$ of finite rank and an $R$-module homomorphism $F_M\colon\sigma_{\!q}^\ast M\to M$. We define finite and local shtukas in Section~\ref{SectLocFinSht} and we recall Abrashkin's results in Section~\ref{Relshtgroupscheme}. His equivalence is given by Drinfeld's functor
\[
(M,F_M) \es\longmapsto\es \Dr_q(M,F_M)\;:=\;\Spec\;\bigl(\bigoplus_{n\ge0}\Sym^n_R M\bigr)\big/\bigl(m^{\otimes q}-F_M(\sigma_{\!q}^*m)\colon m\in M\bigr)\,,
\]
and its quasi-inverse defined on a finite locally free strict $\BF_q$-module scheme $G$ as
\[
G\es\longmapsto\es \ulM_q(G) \es :=\es \Hom_{R\text{\rm-groups},\BF_q\text{\rm-lin}}(G,\BG_{a,R}),
\]
by which we mean the $R$-module of $\BF_q$-equivariant morphisms of group schemes over $R$ on which the Frobenius $F_{M_q(G)}$ is provided by the relative $q$-Frobenius of the additive group scheme $\BG_{a,R}$ over $R$. Various properties of $\ulM$ are reflected in properties of $\Dr_q(\ulM)$; see Theorem~\ref{ThmEqAModSch} for details. The functors $\Dr_q$ and $\ulM_q$ are extended to effective local shtukas $\ulM$ and $z$-divisible local Anderson modules $G$ by
\begin{eqnarray*}
\ulM & \longmapsto & \Dr_q(\ulM)\;:=\;\dirlim[n] \Dr_q(\ulM/z^n\ulM) \qquad\text{and}\\
G & \longmapsto & \ulM_q(G)\;:=\;\invlim[n] \ulM_q\bigl(G[z^n]\bigr)\,.
\end{eqnarray*}
Generalizing \cite[\S\,3.4]{Anderson93}, who treated the case of formal $\BF_q\dbl z\dbr$-modules, we prove the following

\medskip\noindent
{\bfseries Theorem~\ref{ThmEqZDivGps}.}
{\itshape 
\begin{enumerate}
\item 
The two contravariant functors $\Dr_q$ and $\ulM_q$ are mutually quasi-inverse anti-equivalences between the category of effective local shtukas over $R$ and the category of $z$-divisible local Anderson modules over $R$.
\item 
Both functors are $\BF_q\dbl z\dbr$-linear, map short exact sequences to short exact sequences, and preserve (ind-) \'etale objects.
\item
$G$ is a formal $\BF_q\dbl z\dbr$-module if and only if $F_M$ is topologically nilpotent, that is $\im(F_M^n)\subset z M$ for an integer $n$.
\setcounter{enumi}{4}
\item 
the $R\dbl z\dbr$-modules $\omega_{\Dr_q(M,F_M)}$ and $\coker F_M$ are canonically isomorphic.
\end{enumerate}
}

In Section~\ref{SectGlobalObj} we explain the relation of $z$-divisible local Anderson modules and local shtukas to global objects like Drinfeld modules \cite{Drinfeld}, Anderson's \cite{Anderson} abelian $t$-modules and $t$-motives, and Drinfeld shtukas \cite{Drinfeld87}. In particular, if $E$ is a Drinfeld-$\BF_r[t]$-module or an abelian $t$-module over $R$, then the $z^n$-torsion points $E[z^n]$ of $E$ form a finite locally free $\BF_r[t]/(z^n)$-module scheme over $R$. By Example~\ref{ExTorsionOfAModule}(b), the limit $G:=E[z^\infty]:=\dirlim E[z^n]$ in the category of \fppf-sheaves of $\BF_q\dbl z\dbr$-modules on $\Spec R$ satisfies $G[z^n]:=\ker(z^n\colon G\to G)=E[z^n]$ and is a $z$-divisible local Anderson module over $R$. Moreover, the associated effective local shtuka $\ulM_q(G)$ from Theorem~\ref{ThmEqZDivGps} arises as the $z$-adic completion of the $t$-motive associated with $E$; see Example~\ref{ExTorsionOfAModule}(b).

In Section~\ref{DivlAnmodules} we present the above definition of $z$-divisible local Anderson modules $G$ and give equivalent definitions. We also introduce \emph{truncated $z$-divisible local Anderson modules}, like for example $G[z^n]$; see Proposition~\ref{PropTruncZDAM}. In Section~\ref{SectVersch} we investigate, for $\zeta=0$ in $R$, the existence of a \emph{$z^d$-Verschiebung} $V_{z^d,G}$ for (truncated) $z$-divisible local Anderson modules $G$, respectively for local shtukas, with $V_{z^d,G}\circ F_{q,G}=z^d\cdot\id_G$ and $F_{q,G}\circ V_{z^d,G}=z^d\cdot\id_{\sigma_{\!q}^*G}$, where $F_{q,G}$ is the relative $q$-Frobenius of $G$ over $R$. We use the $z^d$-Verschiebung in Theorem~\ref{ThmDefo} to prove that lifting a $z$-divisible local Anderson module from $R/I$ to $R$ when $I^q=(0)$ is equivalent to lifting the Hodge filtration on its de Rham cohomology. In Section~\ref{SectFormalLieGps} we use the $z^d$-Verschiebung to clarify the relation between $z$-divisible local Anderson modules $G$ and formal $\BF_q\dbl z\dbr$-modules. Following the approach of Messing~\cite{Messing} who treated the analogous situation of $p$-divisible groups and formal Lie groups, we show that a $z$-divisible local Anderson module is formally smooth (Theorem~\ref{DLAMFsmooth}), and how to associate a formal $\BF_q\dbl z\dbr$-module with it (Theorem~\ref{LoNilForLie}). We also discuss conditions under which it is an extension of an (ind-)\'etale $z$-divisible local Anderson module by a $z$-divisible formal $\BF_q\dbl z\dbr$-module (Proposition~\ref{Propextension}) and we prove the following

\medskip\noindent
{\bfseries Corollary~\ref{CorEqvcat1}.}
{\itshape There is an equivalence of categories between that of $z$-divisible local Anderson modules over $R$ with $G[z]$ radicial, and the category of $z$-divisible formal $\BF_q\dbl z\dbr$-modules $G$ with $G[z]$ representable by a finite locally free group scheme, such that locally on $\Spec R$ there is an integer $d$ with $(z-\zeta)^d = 0$ on $\omega_G$.
}

\bigskip

In Section~\ref{SectStrictF_qAct} we explain Faltings's notion of \emph{strict} $\BF_q$-module schemes and give details additional to the treatments of Faltings~\cite{Faltings02} and Abrashkin~\cite{Abrashkin}. This notion is based on certain deformations of finite locally free group schemes and the associated cotangent complex, which we review in Section~\ref{CotCom}, respectively in Appendix~\ref{AppCotCom}. There is an equivalent description of finite locally free strict $\BF_q$-module schemes by Poguntke~\cite{Poguntke17}; see Remark~\ref{RemPoguntke}.

\medskip\noindent
{\bfseries Acknowledgements.}
We are grateful to Alain Genestier and Vincent Lafforgue for many valuable discussions on the topics of the present article. We also thank the referees for allowing us to include formulations from their reports to improve our introduction. Both authors were supported by the Deutsche Forschungsgemeinschaft (DFG) in form of the research grant HA3002/2-1 and the SFB's 478 and 878.

\subsection*{Notation}

Let $\BF_q$ be a finite field with $q$ elements and characteristic $p$. For a scheme $S$ over $\Spec \BF_q$ and a positive integer $n\in\BN_{>0}$ we denote by $\sigma_{\!q^n}:=\Frob_{q^n,S}\colon S \to S$ its absolute $q^n$-Frobenius endomorphism which acts as the identity on points and as the $q^n$-power map $b\mapsto b^{q^n}$ on the structure sheaf. For an $S$-scheme $X$, respectively an $\CO_S$-module $M$ we write $\sigma_{\!q^n}^*X:=X\times_{S,\sigma_{\!q^n}}S$, respectively $\sigma_{\!q^n}^*M:=M\otimes_{\CO_S,\sigma_{\!q^n}^*}\CO_S$ for the pullback under $\sigma_{\!q^n}$. For $m\in M$ we also write $\sigma_{\!q^n}^*(m):=m\otimes1\in\sigma_{\!q^n}^*M$ and note that $\sigma_{\!q^n}^*(bm)=bm\otimes1=m\otimes b^{q^n}=b^{q^n}\cdot\sigma_{\!q}^*m$ for $b\in\CO_S$ and $m\in M$.

Let $z$ be an indeterminant over $\BF_q$. Let $\CO_S\dbl z\dbr$ be the sheaf on $S$ of formal power series in $z$. That is $\Gamma\bigl(U,\CO_S\dbl z\dbr\bigr)=\Gamma(U,\CO_S)\dbl z\dbr$ for open $U\subset S$ with the obvious restriction maps. This is indeed a sheaf being the countable direct product of $\CO_S$. Let $\zeta$ be an indeterminant over $\BF_q$ and let $\BF_q\dbl \zeta\dbr$ be the ring of formal power series in $\zeta$ over $\BF_q$. Let $\Nilp_{\BF_q\dbl \zeta\dbr}$ be the category of $\BF_q\dbl \zeta\dbr$-schemes on which $\zeta$ is locally nilpotent.
For $S\in\Nilp_{\BF_q\dbl\zeta\dbr}$ let $\CO_S\dpl z\dpr$ be the sheaf of $\CO_S$-algebras on $S$ associated with the presheaf $U\mapsto\Gamma(U,\CO_U)\dbl z\dbr[\frac{1}{z}]$. If $U$ is quasi-compact then $\CO_S\dpl z\dpr(U)=\Gamma\bigl(U,\CO_S\dbl z\dbr\bigr)[\frac{1}{z}]$. Since $\zeta$ is locally nilpotent on $S$, the sheaf $\CO_S\dpl z\dpr$ equals the sheaf associated with the presheaf $U\mapsto\Gamma\bigl(U,\CO_S\dbl z\dbr\bigr)[\frac{1}{z-\zeta}]$. We denote by $\sigma_{\!q}^\ast$ the endomorphism of $\CO_S\dbl z\dbr$ and $\CO_S\dpl z\dpr$ that acts as the identity on $z$ and as $b\mapsto b^q$ on local sections $b\in \CO_S$. For a sheaf $M$ of $\CO_S\dbl z\dbr$-modules on $S$ we let $\sigma_{\!q}^\ast M:=M\otimes_{\CO_S\dbl z\dbr,\sigma_{\!q}^\ast}\CO_S\dbl z\dbr$ and $M[\frac{1}{z-\zeta}]:=M\otimes_{\CO_S\dbl z\dbr}\CO_S\dbl z\dbr[\frac{1}{z-\zeta}]=M\otimes_{\CO_S\dbl z\dbr}\CO_S\dpl z\dpr$ be the tensor product sheaves. Also for a section $m\in M$ we write $\sigma_{\!q}^*m:=m\otimes 1\in\sigma_{\!q}^\ast M$. Note that a sheaf $M$ of $\CO_S\dbl z\dbr$-modules which fpqc-locally on $S$ is isomorphic to $\CO_S\dbl z\dbr^{\oplus r}$ is already Zariski-locally on $S$ isomorphic to $\CO_S\dbl z\dbr^{\oplus r}$ by \cite[Proposition~2.3]{HV1}. We therefore call such a sheaf simply a \emph{locally free sheaf of $\CO_S\dbl z\dbr$-modules of rank $r$}.

\section{Local and finite shtukas}\label{SectLocFinSht}

Let $S$ be a scheme in $\Nilp_{\BF_q\dbl\zeta\dbr}$.

\begin{Definition} \label{Def1.1}
A \emph{local shtuka of rank} (\emph{or height}) $r$ over $S$ is a pair $\ulM=(M,F_M)$ consisting of a locally free sheaf $M$ of $\CO_S\dbl z\dbr$-modules of rank $r$, and an isomorphism $F_M\colon\sigma_{\!q}^\ast M[\frac{1}{z-\zeta}] \isoto M[\frac{1}{z-\zeta}]$.

A \emph{morphism} of local shtukas $f\colon(M,F_M)\to(M',F_{M'})$ over $S$ is a morphism of the underlying sheaves $f\colon M\to M'$ which satisfies $F_{M'}\circ \sigma_{\!q}^\ast f = f\circ F_M$.

A \emph{quasi-isogeny} between local shtukas $f\colon(M,F_M)\to(M',F_{M'})$ over $S$ is an isomorphism of $\CO_S\dpl z\dpr$-modules $f\colon M\otimes_{\CO_S\dbl z\dbr}\CO_S\dpl z\dpr\isoto M'\otimes_{\CO_S\dbl z\dbr}\CO_S\dpl z\dpr$ with $F_{M'}\circ \sigma_{\!q}^*(f)=f\circ F_M$. A morphism which is a quasi-isogeny is called an \emph{isogeny}.
\end{Definition}

For any local shtuka $(M,F_M)$ over $S\in\Nilp_{\BF_q\dbl\zeta\dbr}$ the homomorphism $M\to M[\tfrac{1}{z-\zeta}]$ is injective by the flatness of $M$ and the following

\begin{Lemma}\label{LemOSZ}
Let $R$ be an $\BF_q\dbl\zeta\dbr$-algebra in which $\zeta$ is nilpotent. Then the sequence of $R\dbl z\dbr$-modules
\[
\xymatrix @R=0.2pc {
0 \ar[r] & R\dbl z\dbr \ar[r] & R\dbl z\dbr \ar[r] & R \ar[r] & 0 \\
& 1 \ar@{|->}[r] & z-\zeta\es,\es z \ar@{|->}[r] & \zeta
}
\]
is exact. In particular $R\dbl z\dbr\subset R\dbl z\dbr[\frac{1}{z-\zeta}]$.
\end{Lemma}

\begin{proof}
If $\sum_i b_iz^i$ lies in the kernel of the first map, that is, $0=(z-\zeta)(\sum_ib_iz^i)=\sum_i(b_{i-1}-\zeta b_i)z^i$, then $b_i=\zeta b_{i+1}=\zeta^n b_{i+n}$ for all $n$. Since $\zeta$ is nilpotent, all $b_i$ are zero. 
Also due to the nilpotency of $\zeta$ the second map is well defined and surjective. For exactness in the middle note that $\sum_i b_i\zeta^i=0$ implies $\sum_i b_iz^i=\sum_i b_i(z^i-\zeta^i)$ which is a multiple of $z-\zeta$.
\end{proof}

For a morphism $f\colon S'\to S$ in $\Nilp_{\BF_q\dbl\zeta\dbr}$ we can pull back a local shtuka $(M,F_M)$ over $S$ to the local shtuka $\bigl( M\otimes_{\CO_S\dbl z\dbr}\CO_{S'}\dbl z\dbr,\,F_M\otimes\id\bigr)$ over $S'$.

We define the \emph{tensor product} of two local shtukas $(M,F_M)$ and $(N,F_N)$ over $S$ as the local shtuka $\bigl(M\otimes_{\CO_S\dbl z\dbr}N\,,\,F_M\otimes F_N\bigr)$. The local shtuka $\UOne(0):=\bigl(\CO_S\dbl z\dbr,\,F_{\BOne(0)}=\id_{\CO_S\dbl z\dbr}\colon\sigma_{\!q}^*\CO_S\dbl z\dbr=\CO_S\dbl z\dbr\isoto\CO_S\dbl z\dbr\bigr)$ is a \emph{unit object} for the tensor product. The \emph{dual} $(M\dual, F_{M\dual})$ of a local shtuka $(M,F_M)$ over $S$ is defined as the sheaf $M\dual=\CHom_{\CO_S\dbl z\dbr}\bigl(M,\CO_S\dbl z\dbr\bigr)$ together with
\[
\TS F_{M\dual}\colon\;\sigma_{\!q}^\ast M\dual[\frac{1}{z-\zeta}]\isoto M\dual[\frac{1}{z-\zeta}]\,,\es f\mapsto f\circ F_M^{-1}.
\]
Also there is a natural definition of \emph{internal Hom's} with $\CHom(\ulM,\ulN)=\ulM\dual\otimes\ulN$. This makes the category of local shtukas over $S$ into an $\BF_q\dbl z\dbr$-linear, additive, rigid tensor category. It is an exact category in the sense of Quillen~\cite[\S2]{Quillen} if one calls a short sequence of local shtukas \emph{exact} when the underlying sequence of sheaves of $\CO_S\dbl z\dbr$-modules is exact.

\begin{Lemma}\label{ExistenceOfe}
Let $(M,F_M)$ be a local shtuka over $S$. Then locally on $S$ there are $e,e', N\in \BZ$ such that $(z-\zeta)^{e'} M \subset F_M(\sigma_{\!q}^\ast M) \subset (z-\zeta)^{-e} M$ and $z^N M\subset F_M(\sigma_{\!q}^\ast M)$. For any such e the map $F_M\colon\sigma_{\!q}^*M\to(z-\zeta)^{-e}M$ is injective, and the quotient $(z-\zeta)^{-e} M/F_M(\sigma_{\!q}^\ast M)$ is a locally free $\CO_S$-module of finite rank.
\end{Lemma}
\begin{proof}
We work locally on $\Spec R\subset S$ and assume that $\sigma_{\!q}^\ast M$ and $M$ are free $\CO_S\dbl z\dbr$-modules. Applying $F_M$ to a basis of $\sigma_{\!q}^*M$, respectively $F_M^{-1}$ to a basis of $M$, proves the existence of $e$, respectively $e'$. If $N\ge e'$ is an integer which is a power of $p$ such that $\zeta^N=0$ in $R$, then $z^N M=(z^N-\zeta^N)M=(z-\zeta)^N M\subset F_M(\sigma_{\!q}^\ast M)$.

We prove that the quotient $K:=(z-\zeta)^{-e} M/F_M(\sigma_{\!q}^\ast M)$ is a locally free $R$-module of finite rank. This was already proved in \cite[Lemma~4.3]{HV1}, but the argument given there only works if $R$ is noetherian, because it uses that $R\dbl z\dbr$ is flat over $R$. We now give a proof also in the non-noetherian case. Since $K= \coker\bigl(F_M\mod(z-\zeta)^{e+e'}\colon \sigma_{\!q}^\ast M/(z-\zeta)^{e+e'}\sigma_{\!q}^\ast M\to (z-\zeta)^{-e}M/(z-\zeta)^{e'}M\bigr)$, it is of finite presentation over $R$. Since $R\dbl z\dbr\subset R\dbl z\dbr[\tfrac{1}{z-\zeta}]$ is a subring by Lemma~\ref{LemOSZ} and $M$ is locally free, the map $F_M\colon\sigma_{\!q}^*M\to(z-\zeta)^{-e}M$ is injective. Let $\mathfrak{m} \subset R$ be a maximal ideal and set $k = R/\mathfrak{m}$. In the exact sequence 
\[
0 \;\to\; \Tor_1^{R\dbl z\dbr }(K,k\dbl z\dbr)  \;\to\; \sigma_{\!q}^\ast M\otimes _{R\dbl z\dbr }k\dbl z\dbr  \;\to\; (z-\zeta)^{-e}M\otimes _{R\dbl z\dbr }k\dbl z\dbr \;\to\; K\otimes _{R\dbl z\dbr }k\dbl z\dbr \;\to\; 0,
\]
we have isomorphisms $\sigma_{\!q}^\ast M\otimes _{R\dbl z\dbr }k\dbl z\dbr \cong  k\dbl z\dbr^{\oplus \rk M} \cong (z-\zeta)^{-e}M\otimes _{R\dbl z\dbr }k\dbl z\dbr$. Moreover, $\zeta=0$ in $k$ and hence $z^{e+e'}K\otimes _{R\dbl z\dbr }k\dbl z\dbr=0$. Since $k\dbl z\dbr$ is a PID, the map $\sigma_{\!q}^\ast M\otimes _{R\dbl z\dbr }k\dbl z\dbr  \to (z-\zeta)^{-e}M\otimes _{R\dbl z\dbr }k\dbl z\dbr$ is  injective by the elementary divisor theorem, and hence $0=\Tor_1^{R\dbl z\dbr }(K,k\dbl z\dbr)$. To relate this to $\Tor_1^{R}(K,k)=\Tor_1^{R\dbl z\dbr/(z-\zeta)^{e+e'}}\bigl(K,k\dbl z\dbr/(z^{e+e'})\bigr)$ we use the change of rings spectral sequence \cite[Theorem 10.71]{Rotman} and the induced epimorphism (from its associated 5-term sequence of low degrees, see \cite[Theorem 10.31]{Rotman})
\[
\ldots\;\to\;\Tor_1^{R\dbl z\dbr }(K,k\dbl z\dbr)\;\to\;\Tor_1^{R\dbl z\dbr/(z-\zeta)^{e+e'}}\bigl(K,k\dbl z\dbr/(z^{e+e'})\bigr)\;\to\;0\,.
\]
It follows that $\Tor_1^{R}(K,k)=0$ and from Nakayama's lemma we conclude that $K$ is locally free over $R$ of finite rank; compare \cite[Exercise~6.2]{Eisenbud}. 
\end{proof}

\begin{Definition}
A local shtuka $\ulM=(M,F_M)$ over $S$ is called \emph{effective} if $F_M$ is actually a morphism $F_M\colon\sigma_{\!q}^\ast M\into M$. Let $(M,F_M)$ be effective of rank $r=\rk\ulM$. We say that
\begin{enumerate}
\item 
$(M,F_M)$ has \emph{dimension $d$} if $\coker F_M$ is locally free of rank $d$ as an $\CO_S$-module.
\item 
$(M,F_M)$ is \emph{\'etale} if $F_M\colon\sigma_{\!q}^\ast M\isoto M$ is an isomorphism.
\item 
$F_M$ is \emph{topologically nilpotent} if locally on $S$ there is an integer $n$ such that $\im(F_M^n)\subset z M$, where $F_M^n:=F_M\circ\sigma_{\!q}^*F_M\circ\ldots\circ\sigma_{\!q^{n-1}}^*F_M\colon \sigma_{\!q^n}^*M\to M$.
\item 
$\ulM$ is bounded by $(d,0,\ldots,0)\in\BZ^r$ for an integer $d\ge0$, if $\ulM$ satisfies 
\[
\bigwedge\limits_{\CO_S\dbl z\dbr}^r F_M(\sigma_{\!q}^*M) \;=\; (z-\zeta)^d\cdot\bigwedge\limits_{\CO_S\dbl z\dbr}^r M\,.
\]
\end{enumerate}
\end{Definition}

\begin{Example}
We define the \emph{Tate objects} in the category of local shtukas over $S$ as
\[
\UOne(n)\es:=\es\bigl(\CO_S\dbl z\dbr\,,\,F_M:1\mapsto(z-\zeta)^n\bigr)\,.
\]
By Lemma~\ref{ExistenceOfe} every local shtuka over a quasi-compact scheme $S$ becomes effective after tensoring with a suitable Tate object.
\end{Example}

More generally, let now $S$ be an arbitrary $\BF_q$-scheme.

\begin{Definition}
A \emph{finite $\BF_q$-shtuka} over $S$ is a pair $\ulM=(M,F_M)$ consisting of a locally free $\CO_S$-module $M$ on $S$ of finite rank denoted $\rk\ulM$, and an $\CO_S$-module homomorphism $F_M\colon\sigma_{\!q}^\ast M\to M$. A \emph{morphism} $f\colon(M,F_M)\to(M',F_{M'})$ of finite $\BF_q$-shtukas is an $\CO_S$-module homomorphism $f\colon M\to M'$ 
which makes the following diagram commutative
\[
\xymatrix @C+1pc {
\sigma_{\!q}^\ast M \ar[r]^{\sigma_{\!q}^\ast f}\ar[d]^{F_M} & \sigma_{\!q}^\ast M'\ar[d]^{F_{M'}}\\
M \ar[r]^f & M'{\,.\!\!\!\!}
}
\]
We denote the category of finite $\BF_q$-shtukas over $S$ by $\FqSht_S$.

A finite $\BF_q$-shtuka over $S$ is called \emph{\'etale} if $F_M$ is an isomorphism. We say that $F_M$ is \emph{nilpotent} if there is an integer $n$ such that $F_M^n:=F_M\circ\sigma_{\!q}^*F_M\circ\ldots\circ\sigma_{\!q^{n-1}}^*F_M=0$.
\end{Definition}

Finite $\BF_q$-shtukas were studied at various  places in the literature. They were called ``(finite) $\phi$-sheaves'' by Drinfeld~\cite[\S\,2]{Drinfeld87}, Taguchi and Wan~\cite{Taguchi95,TW} and ``Dieudonn\'e $\BF_q$-modules'' by Laumon~\cite{Laumon}. Finite $\BF_q$-shtukas over a field admit a canonical decomposition.

\begin{Proposition} (\cite[Lemma B.3.10]{Laumon}) \label{PropCanonDecompDMod}
If $S$ is the spectrum of a field $L$ every finite $\BF_q$-shtuka $\ulM=(M,F_M)$ is canonically an extension of finite $\BF_q$-shtukas
\[
0\es\longto\es(M_\et,F_\et) \es\longto\es (M,F_M)\es\longto\es (M_\nil,F_\nil)\es\longto\es 0
\]
where $F_\et$ is an isomorphism and $F_\nil$ is nilpotent. $\ulM_\et=(M_\et,F_\et)$ is the largest \'etale finite $\BF_q$-sub-shtuka of $\ulM$ and equals $\im(F_M^{\rk\ulM})$. If $L$ is perfect this extension splits canonically.
\end{Proposition}

\begin{proof}
This was proved by Laumon~\cite[Lemma B.3.10]{Laumon} for perfect $L$. In general one considers the descending sequence of $L$-subspaces $\ldots\supset\im(F_M^n)\supset\im(F_M^{n+1})\supset\ldots$ of $M$ which stabilizes at some finite $n$. If $\im(F_M^{n+1})=\im(F_M^n)$ then $F_M\colon\sigma_{\!q}^*(\im F_M^n)\onto\im F_M^{n+1}=\im F_M^n$ is surjective, hence bijective, and therefore $\im(F_M^{n'})=\im(F_M^n)$ for all $n'\ge n$. So the sequence stabilizes already for some $n\le\rk\ulM$ and $M_\et=\im(F_M^{\rk\ulM})$. If $L$ is perfect, $M_\nil$ is isomorphic to the submodule $\bigcup_{n\ge0}\ker(F_M^n\circ\sigma_{\!q^n}^*\colon M\to M)$ of $M$; see \cite[Lemma B.3.10]{Laumon}.
\end{proof}

\begin{Example}
Every effective local shtuka $(M,F_M)$ of rank $r$ over $S$ yields for every $n\in\BN$ a finite $\BF_q$-shtuka $\bigl(M/z^nM,F_M\mod z^n\bigr)$ of rank $rn$, and $(M,F_M)$ equals the projective limit of these finite $\BF_q$-shtukas.
\end{Example}

Thus from Proposition~\ref{PropCanonDecompDMod} we obtain

\begin{Proposition} \label{PropCanonDecompZCryst}
If $S$ is the spectrum of a field $L$ in $\Nilp_{\BF_q\dbl\zeta\dbr}$ every effective local shtuka $(M,F_M)$ is canonically an extension of effective local shtukas
\[
0\es\longto\es (M_\et,F_\et) \es\longto\es (M,F_M)\es\longto\es (M_\nil,F_\nil)\es\longto\es 0
\]
where $F_\et$ is an isomorphism and $F_\nil$ is topologically nilpotent. 
$(M_\et,F_\et)$ is the largest \'etale effective local sub-shtuka of $(M,F_M)$. If $L$ is perfect this extension splits canonically. \qed
\end{Proposition}

\section{Review of deformations of finite locally free group schemes} \label{CotCom}

For a commutative group scheme $G$ over $S$ we denote by $\epsilon_G\colon S\to G$ its unit section and by $\omega_G:=\epsilon_G^*\Omega^1_{G/S}$ its \emph{co-Lie module}. It is a sheaf of $\CO_S$-modules. In order to describe which group objects are classified by finite $\BF_q$-shtukas we need to review the definition of strict $\BF_q$-module schemes in the next two sections. We follow Faltings~\cite{Faltings02} and Abrashkin~\cite{Abrashkin}. We begin in this section with a review of deformations of finite locally free group schemes. Recall that a group scheme $G$ over $S$ is called \emph{finite locally free} over $S$ if on every open affine $\Spec R\subset S$ the scheme $G$ is of the form $\Spec A$ for a finite locally free $R$-module $A$. By \cite[I$_{\rm new}$, Proposition~6.2.10]{EGA} this is equivalent to $G$ being finite flat and of finite presentation over $S$. The rank of the $R$-module $A$ is called the \emph{order of $G$} and is denoted $\ord G$. It is a locally constant function on $S$. The following facts will be used throughout.

\begin{Remark}\label{RemFactsOnG}
(a) A morphism $G'\to G$ of finite locally free group schemes is a monomorphism (of schemes, or equivalently of \fppf-sheaves on $S$) if and only if it is a closed immersion by \cite[IV$_4$, Corollaire~18.12.6]{EGA}, because it is proper.

\medskip\noindent
(b) Let $G$ and $G''$ be group schemes over $S$ which are finite and of finite presentation and assume that $G$ is flat over $S$. Then a morphism $G\to G''$ of is an epimorphism of \fppf-sheaves on $S$ if and only if it is faithfully flat; compare the proof of \cite[Chapter~I, Lemma~1.5(b)]{Messing}.

\medskip\noindent
(c) A sequence $0\to G'\to G\to G''\to0$ of finite locally free group schemes over $S$ is called \emph{exact} if it is exact when viewed as a sequence of \fppf-sheaves on $S$. By the above this is equivalent to the conditions that $G\to G''$ is faithfully flat, and that $G'\to G$ is a closed immersion which equals the kernel of $G\to G''$. In this case $\ord(G)=\ord(G')\cdot\ord(G'')$ as can be seen from the isomorphism $\CO_{G'}\cong \CO_G\otimes_{\CO_{G''},\epsilon_{G''}^*}\CO_S$ where $\epsilon_{G''}\colon S\to G''$ is the unit section, and from the multiplicativity of ranks $\rk_{\CO_S}\CO_G=(\rk_{\CO_{G''}}\CO_G)\cdot(\rk_{\CO_S}\CO_{G''})=(\rk_{\CO_S}\CO_{G'})\cdot(\rk_{\CO_S}\CO_{G''})$.

\medskip\noindent
(d) If $G'\into G$ is a closed immersion of finite locally free group schemes over $S$, then the quotient $G/G'$ exists as a finitely presented group scheme over $S$ by \cite[Th\'eor\`eme~V.4.1 and Proposition V.9.1]{SGA3}, which is flat by \cite[IV$_3$, Corollaire~11.3.11]{EGA}. It is integral over $S$ and hence finite, because $\CO_{G/G'}\subset\CO_G$. In particular, $G/G'$ is finite locally free over $S$.
\end{Remark}

In the following we will work locally on $S$ and assume that $S=\Spec R$ is affine. Let $G=\Spec A$ be a finite locally free group scheme over $S$. Then $G$ is a relative complete intersection by \cite[Proposition~III.4.15]{SGA3}. This means that locally on $S$ we can take $A = R[X_1,\ldots , X_n]/I$ where the ideal $I$ is generated by a regular sequence $(f_1,\ldots,f_n)$ of length $n$; compare \cite[IV$_4$, Proposition~19.3.7]{EGA}. The unit section $\epsilon_G\colon S\to G$ defines an augmentation $\epsilon_A:=\epsilon_G^*\colon A\onto R$ of the $R$-algebra $A$, that is, $\epsilon_A$ is a section of the structure morphism $\iota_A\colon R\hookrightarrow A$. Faltings~\cite{Faltings02} and Abrashkin~\cite{Abrashkin} define deformations of augmented $R$-algebras as follows. For every augmented $R$-algebra $(A,\epsilon_A\colon A\onto R)$ set $I_A:=\ker\epsilon_A$. For the polynomial ring $R[\ulX]=R[X_1,\ldots,X_n]$ set $I_{R[\ulX]}=(X_1,\ldots,X_n)$ and $\epsilon_{R[\ulX]}\colon R[\ulX]\onto R,\,X_\nu\mapsto0$. Abrashkin~\cite[\S\S\,1.1 and 1.2]{Abrashkin} makes the following

\begin{Definition}\label{DefDAug}
The category $\DSch_S$ has as objects all triples $\CH=(H,H^\flat,i_\CH)$, where $H=\Spec A$ for an augmented $R$-algebra $A$ which is finite locally free as an $R$-module, where $H^\flat=\Spec A^\flat$ for an augmented $R$-algebra $A^\flat$, and where $i_\CH\colon H\into H^\flat$ is a closed immersion given by an epimorphism $i_\CA\colon A^\flat\onto A$ of augmented $R$-algebras, such that locally on $\Spec R$ there is a polynomial ring $R[\ulX]=R[X_1,\ldots,X_n]$ and an epimorphism of augmented $R$-algebras $j\colon R[\ulX]\onto A^\flat$ satisfying the properties that
\begin{itemize}
\item the ideal $I:=\ker(i_\CA\circ j)$ is generated by elements of a regular sequence of length $n$ in $R[\ulX]$,
\item $\ker j=I\cdot I_{R[\ulX]}$, and hence $A=R[\ulX]/I$ and $A^\flat=R[\ulX]/(I\cdot I_{R[\ulX]})$.
\end{itemize}
In particular, $H$ is a relative complete intersection. We write $\CA=(A,A^\flat,i_\CA)$ and $\CH=\Spec\CA$. A morphism $\Spec(\wt A,\wt A^\flat,i_{\wt\CA})\to\Spec(A,A^\flat,i_\CA)$ in $\DSch_S$ is given by morphisms $f\colon A\to \wt A$ and $f^\flat\colon A^\flat\to \wt A^\flat$ of augmented $R$-algebras such that $f\circ i_\CA=i_{\wt\CA}\circ f^\flat$. Sometimes $i_\CH$ and $i_\CA$ are ommited.
\end{Definition}

For an object $\CH=\Spec(A, A^{\flat}, i_\CA)$ of $\DSch_S$ define the two $R$-modules $N_\CH = \ker i_\CA$ and $t_\CH^\ast = I_{A^\flat}/I_{A^\flat}^2$, where $I_{A^\flat}$ is the kernel of the augmentation $\epsilon_{A^\flat}\colon A^\flat\to R$. After choosing locally on $\Spec R$ an epimorphism $j\colon R[\ulX]\onto A^\flat$ we have $I_{A^\flat}=I_{R[\ulX]}/(I\cdot I_{R[\ulX]})$, which implies $N_\CH = I/{(I\cdot I_{R[\ulX]})}$ and $t_\CH^\ast = I_{R[\ulX]}/{I_{R[\ulX]}^2}$. Both are finite locally free $R$-modules of the same rank. This is obvious for $t^*_\CH$, and for $N_\CH$ we give a proof in Lemma~\ref{LemmaNBLocFree} below. Also note that $I_{A^\flat}\cdot\ker i_\CA=0$, because $\ker i_\CA=I/(I\cdot I_{R[\ulX]})$. We write $\inclN=\inclN_\CH\colon N_\CH\into A^\flat$ for the natural inclusion and $\pi=\pi_\CH:=(\id-\iota_{A^\flat}\epsilon_{A^\flat})\mod I_{A^\flat}^2\colon A^\flat\onto t^*_\CH$. If $\CH=(A,A^\flat)$ and $\wt\CH=\Spec(\wt A,\wt A^\flat)$ every morphism $(f,f^\flat)\colon(A,A^\flat)\to(\wt A,\wt A^\flat)$ in $\Hom_{\DSch}(\wt\CH,\CH)$ induces morphisms of $R$-modules $N_f\colon N_\CH\to N_{\wt\CH}$ and $t^*_f\colon t^*_\CH\to t^*_{\wt\CH}$ with $f^\flat\circ\inclN_\CH=\inclN_{\wt\CH}\circ N_f$ and $\pi_{\wt\CH}\circ f^\flat=t^*_f\circ\pi_\CH$.

\begin{Lemma}\label{LemmaNBLocFree}
If $\CH=\Spec(A,A^\flat,i_\CA)\in\DSch_S$ then $A^\flat$ and $N_\CH$ are finite locally free $R$-modules with $\rk_R N_\CH=\rk_R t^*_\CH$.  
\end{Lemma}

\begin{proof}
Considering the exact sequence $0\to \ker i_\CA\to A^\flat\to A\to 0$ of $R$-modules it suffices to prove that $N_\CH=\ker i_\CA$ is finite locally free. Working locally on $\Spec R$ we assume that there is an epimorphism $j\colon R[\ulX]\onto A^\flat$ as in Definition~\ref{DefDAug} such that $I:=\ker(i_\CA\circ j)$ is generated by a regular sequence $(f_1,\ldots,f_n)$. Then $\ker i_\CA=I/(I\cdot I_{R[\ulX]})=I\otimes_{R[\ulX]}R$. From Lemma~\ref{Lemmaforfinitepresentation} we get an exact sequence of $R[\ulX]$-modules
\[
\xymatrix @R=0pc {
\bigoplus_{1\le \mu<\nu\le n}\; R[\ulX]\cdot h_{\mu\nu} \ar[r] & \bigoplus_{\nu=1}^n \; R[\ulX]\cdot g_\nu \ar[r] & I \ar[r] & 0\\
h_{\mu\nu}\ar@{|->}[r] & f_\nu g_\mu-f_\mu g_\nu\,,\quad g_\nu \ar@{|->}[r] & f_\nu
}
\]
Applying $.\,\otimes_{R[\ulX]}R$, the first homomorphism becomes zero because $f_\nu\in I_{R[\ulX]}$, whence $f_\nu=0$ in $R$. So $\ker i_\CA\cong R^{\oplus n}$. The equality of ranks follows from $t^*_\CH=I_{R[\ulX]}/I_{R[\ulX]}^2=\bigoplus_{\nu=1}^n R\cdot X_\nu$.
\end{proof}

Faltings~\cite[\S\,2]{Faltings02} notes the following

\begin{Lemma}\label{LemmaExistsLift}
Let $\CH=\Spec(A, A^{\flat},i_\CA)$ and  $\wt\CH=\Spec(\wt A, \wt A^{\flat},i_{\wt\CA})$ be objects in $\DSch_S$ and let $f\colon A \to \wt A$ be a morphism of augmented $R$-algebras. Then the set 
\[
\CL\,:=\,\bigl\{\,f^\flat\colon A^\flat\to \wt A^\flat\text{ morphisms of augmented $R$-algebras for which }(f, f^\flat) \in \Hom_{\DSch_S} (\wt\CH,\CH)\,\bigr\}
\]
is non-empty and is a principal homogeneous space under $\Hom_R(t^*_{\CH},N_{\wt\CH})$. That is, for any $f^\flat\in\CL$ the map $\Hom_R(t^*_{\CH},N_{\wt\CH})\to\CL,\;h\mapsto f^\flat+\inclN_{\wt\CH}\circ h\circ \pi_\CH$ is a bijection.
\end{Lemma}

For the convenience of the reader we include a

\begin{proof}
We first show that for every $f^\flat\in\CL$ the map $\tilde f^\flat:=f^\flat+\inclN_{\wt\CH} h\pi_\CH\colon A^\flat\to\wt A^\flat$ is a morphism of augmented $R$-algebras. Clearly it is a map of $R$-modules with $\tilde f^\flat(I_{A^\flat})\subset I_{\wt A^\flat}$. We must show that $\tilde f^\flat(xy)=\tilde f^\flat(x)\tilde f^\flat(y)$. We write $x=x'+x''$ and $y=y'+y''$ with $x',y'\in\iota_{A^\flat}(R)$ and $x'',y''\in I_{A^\flat}$. Since $\inclN_{\wt\CH} h\pi_\CH(x)\cdot\inclN_{\wt\CH} h\pi_\CH(y)\in\ker(i_{\wt\CA})^2=0$ and $f^\flat(x'')\cdot\inclN_{\wt\CH} h\pi_\CH(y)\in I_{\wt A^\flat}\cdot\ker(i_{\wt\CA})=0$, as well as $\inclN_{\wt\CH} h\pi_\CH(x'y')=\inclN_{\wt\CH} h\pi_\CH(x''y'')=0$ we compute
\begin{eqnarray*}
\tilde f^\flat(x)\tilde f^\flat(y) & = & f^\flat(x)f^\flat(y)+(f^\flat(x')+f^\flat(x''))\cdot\inclN_{\wt\CH} h\pi_\CH(y)+\inclN_{\wt\CH} h\pi_\CH(x)\cdot(f^\flat(y')+f^\flat(y''))\\
& = & f^\flat(xy)+x'\cdot\inclN_{\wt\CH} h\pi_\CH(y)+y'\cdot\inclN_{\wt\CH} h\pi_\CH(x)\\
& = & f^\flat(xy)+\inclN_{\wt\CH} h\pi_\CH(x'y+y'x-x'y'+x''y'')\\
& = & \tilde f^\flat(xy)\,.
\end{eqnarray*}
Since $\im(\inclN_{\wt\CH} h\pi_\CH)\subset\ker(i_{\wt A^\flat})$ we have $i_{\wt A^\flat}\circ\tilde f^\flat=i_{\wt A^\flat}\circ f^\flat=f\circ i_{A^\flat}$ and so $\tilde f^\flat\in\CL$.

Next if $f^\flat,\tilde f^\flat\in\CL$ then $\im(f^\flat-\tilde f^\flat)\subset\ker(i_{\wt A^\flat})$. If further $x=x'+x''\in A^\flat$ with $x'\in\iota_{A^\flat}(R)$ and $x''\in I_{A^\flat}$, then $f^\flat(x)-\tilde f^\flat(x)=x'+f^\flat(x'')-x'-\tilde f^\flat(x'')$ is independent of $x'$. And for $x'',y''\in I_{A^\flat}$ the equation $f^\flat(x''y'')-\tilde f^\flat(x''y'')=f^\flat(x'')\bigl(f^\flat(y'')-\tilde f^\flat(y'')\bigr)+\bigl(f^\flat(x'')-\tilde f^\flat(x'')\bigr)\tilde f^\flat(y'')\in I_{\wt A^\flat}\cdot\ker(i_{\wt\CA})=0$ shows that $f^\flat-\tilde f^\flat$ factors through a unique $R$-homomorphism $h\in\Hom_R(t^*_{\CH},N_{\wt\CH})$ as $f^\flat-\tilde f^\flat=\inclN_{\wt\CH} h\pi_\CH$. 

It remains to show that $\CL\ne\emptyset$. Locally on open affine subsets $U_i:=\Spec R_i\subset S$ we consider presentations $A=R_i[\ulX]/I$, $A^\flat=R_i[\ulX]/(I\,I_{R_i[\ulX]})$ and $\wt A=R_i[\wt\ulX]/\tilde I$, $\wt A^\flat=R_i[\wt\ulX]/(\tilde I\,I_{R_i[\wt\ulX]})$. We choose an $R_i$-homomorphism $F_i\colon R_i[\ulX]\to R_i[\wt\ulX]$ which lifts $f$. In particular, $F_i(I)\subset\tilde I$ and $F_i(I_{R_i[\ulX]})\subset I_{R_i[\wt\ulX]}$. Therefore $F_i$ induces a homomorphism $f_i^\flat\colon A^\flat \to \wt A^ \flat$ which lifts $f$. In order to glue the $f_i^\flat$ we consider the quasi-coherent sheaf $H:=\CHom_S(t^*_\CH,N_{\wt\CH})$ on $S$. Over $U_{ij}:=U_i\cap U_j$ both $f_i^\flat$ and $f_j^\flat$ lift $f$. By the above there is a section $h_{ij}\in\Gamma(U_{ij},H)$ with $f_i^\flat-f_j^\flat=\inclN_{\wt\CH} h_{ij}\pi_\CH$. The $h_{ij}$ form a \v{C}ech cocycle. Since $\CKoh^1(\{U_i\},H)=0$ we find elements $h_i\in\Gamma(U_i,H)$ with $h_{ij}=h_i-h_j$. This means that the $\tilde f_i^\flat:=f_i^\flat-\inclN_{\wt\CH} h_i\pi_\CH$ coincide over $U_{ij}$ and glue to a morphism $\tilde f^\flat\colon A^\flat\to\wt A^\flat$ which lies in $\CL$.
\end{proof}

The category $\DSch_S$ possesses direct products. If $\CH=\Spec(A, A^{\flat}, i_\CA)$ and $\wt\CH=\Spec(\wt A,\wt A^{\flat}, i_{\wt\CA})$, then the product $\CH\times_S\wt\CH$ is given by $\Spec(A\otimes_R\wt A,(A\otimes_R\wt A)^\flat,\kappa)$, where 
\[
(A\otimes_R\wt A)^\flat\,:=\,(A^\flat\otimes_R\wt A^\flat)/(\ker i_\CA\otimes\wt A^\flat+A^\flat\otimes\ker i_{\wt\CA})\cdot(I_{A^\flat}\otimes\wt A^\flat+A^\flat\otimes I_{\wt A^\flat})
\]
and $\kappa$ is the natural epimorphism $(A\otimes_R A)^\flat\onto A\otimes_R A$. After choosing locally on $\Spec R$ presentations $A = R[\ulX]/I,\, A^\flat = R[\ulX]/(I \cdot I_{R[\ulX]})$ and $\wt A = R[\wt\ulX]/\tilde I,\,\wt A^\flat = R[\wt\ulX]/(\tilde I \cdot I_{R[\wt\ulX]})$ we can write
\[
(A\otimes_R\wt A)^\flat\,=\,R[\ulX\otimes 1, 1\otimes\wt\ulX]\,\big/\,(I\otimes R[\wt\ulX] + R[\ulX] \otimes\tilde I)\cdot\bigl(I_{R[\ulX]}\otimes R[\wt\ulX] + R[\ulX] \otimes I_{R[\wt\ulX]}\bigr).
\]
The projections $\pr_1\colon\CH\times_S\wt\CH\to\CH$ and $\pr_2\colon\CH\times_S\wt\CH\to\wt\CH$ come from the natural embeddings of $R[\ulX]$ and $R[\wt\ulX]$ into $R[\ulX\otimes 1,1 \otimes\wt\ulX]$.

\begin{Definition}\label{DefDGr}
Let $\DGr_S$ be the category of group objects in $\DSch_S$. If $\CG = \Spec \CA \in \DGr_S$, then its group structure is given via the comultiplication $\Delta\colon A\to A\otimes_R A$ and $\Delta^\flat\colon A^\flat\to(A\otimes_R A)^\flat$, the counit $\epsilon\colon A\to R$ and $\epsilon^\flat\colon A^\flat\to R$, and the coinversion $[-1]\colon A\to A$ and $[-1]^\flat\colon A^\flat\to A^\flat$, which satisfy the usual axioms. In particular, we require the counit axiom $(\id_{A^\flat}\otimes\epsilon^\flat) \circ \Delta^\flat = \id_{A^\flat} = (\epsilon^\flat\otimes\id_{A^\flat}) \circ \Delta^\flat $, and that $\epsilon$ and $\epsilon^\flat$ are the augmentation maps. The morphisms in $\DGr_S$ are morphisms of group objects.
\end{Definition}

If $\CG=(G,G^\flat)\in\DGr_S$, note that $G = \Spec A$ is a finite locally free group scheme over $R$ with the comultiplication $\Delta$, the counit $\epsilon$ and the coinversion $[-1]$. But in general $G^\flat$ is not a group scheme over $S$ when the comultiplication $\Delta^\flat\colon A^\flat\to(A\otimes_R A)^\flat$ does not lift to $A^\flat\otimes_R A^\flat$. Faltings and Abrashkin~\cite[\S\,1.2]{Abrashkin} make the following

\begin{Remark}\label{RemarkFaltingsAbrashkin}
\begin{enumerate}
\item \label{Abra_D}
If $\CG = \Spec(A, A^{\flat}, i_\CA) \in \DSch_S$ and $G = \Spec A$ is a finite locally free group scheme over $R$, then there exists a unique structure of a group object on $\CG$, which is compatible with that of $G$. It satisfies $\Delta^\flat(x)-x\otimes 1-1\otimes x\,\in\, I_{A^\flat}\otimes I_{A^\flat}$ for all $x\in I_{A^\flat}$.
\item \label{Abra_E}
If $\CG,\CH\in\DGr_S$ are group objects and $(f, f^\flat) \in \Hom_{\DSch_S}(\CG, \CH)$ such that $f\colon G \to H$ is a morphism of group schemes, then $(f, f^\flat) \in \Hom_{\DGr_S}(\CG, \CH)$.
\end{enumerate}
\end{Remark}

For the convenience of the reader we give a
\begin{proof}
\ref{Abra_D} By Lemma~\ref{LemmaExistsLift} we may choose a homomorphism $\wt\Delta^\flat\colon A^\flat\to(A\otimes_R A)^\flat$ which lifts the comultiplication map $\Delta\colon A \to A \otimes_R A$. We want to modify $\wt\Delta^\flat$ to $\Delta^\flat:=\wt\Delta^\flat+\inclN_{\CG\times\CG}\circ h\circ\pi_\CG$ for an $R$-homomorphism $h\in\Hom_R(t^*_\CG,N_{\CG\times\CG})$ such that $(\id_{A^\flat}\otimes\epsilon^\flat) \circ\Delta^\flat = \id_{A^\flat} = (\epsilon^\flat\otimes\id_{A^\flat}) \circ \Delta^\flat $ holds. Thus we can take $\inclN_{\CG\times\CG}\circ h\circ\pi_\CG(x)=\bigl(x-(\id_{A^\flat}\otimes\epsilon^\flat)\circ\wt\Delta^\flat(x)\bigr)\otimes 1 + 1\otimes\bigl(x-(\epsilon^\flat\otimes\id_{A^\flat})\circ\wt\Delta^\flat(x)\bigr)$. Note that this lies in $N_{\CG\times\CG}$, because $(\id_A\otimes\epsilon) \circ\Delta=\id_A=(\epsilon\otimes\id_A) \circ\Delta$ implies $x-(\id_{A^\flat}\otimes\epsilon^\flat)\circ\wt\Delta^\flat(x)\in\ker i_\CA$ and $x-(\epsilon^\flat\otimes\id_{A^\flat})\circ\wt\Delta^\flat(x)\in\ker i_\CA$. This also shows that $\inclN_{\CG\times\CG}\circ h\circ\pi_\CG\colon A^\flat\to(A\otimes_R A)^\flat$ factors through $t_\CG^*$ and therefore $h$ exists.

To prove uniqueness of $\Delta^\flat$ we work locally on $S$ and choose a presentation $A=R[\ulX]/I$ and $A^\flat=R[\ulX]/(I\cdot I_{R[\ulX]})$. Then we have $A\otimes_RA=R[\ulX\otimes1,1\otimes\ulX]/(I\otimes R[\ulX]+R[\ulX]\otimes I)$ and 
\[
(A\otimes_RA)^\flat\;=\;R[\ulX\otimes1,1\otimes\ulX]\,\big/\,(I\otimes R[\ulX]+R[\ulX]\otimes I)\cdot\bigl(I_{R[\ulX]}\otimes R[\ulX]+R[\ulX]\otimes I_{R[\ulX]}\bigr)\,. 
\]
Note that every element $u\otimes \tilde u\in I\otimes R[\ulX]$ with $u\in I$ and $\tilde u=\tilde u'+\tilde u''\in R[\ulX]$, where $\tilde u'\in R$ and $\tilde u''\in I_{R[\ulX]}$, satisfies $u\otimes \tilde u=u\otimes \tilde u'=(\tilde u'u)\otimes1$ in 
\[
\ker(i_{\CA\otimes\CA})\,=\,(I\otimes R[\ulX]+R[\ulX]\otimes I)\,\big/\,(I\otimes R[\ulX]+R[\ulX]\otimes I)\cdot\bigl(I_{R[\ulX]}\otimes R[\ulX]+R[\ulX]\otimes I_{R[\ulX]}\bigr)\,. 
\]
Now assume that $\wt\Delta^\flat$ and $\Delta^\flat$ both satisfy the counit axiom and lift $\Delta$. Then for every $x\in A^\flat$ there are $u,v\in I$ such that $\Delta^\flat(x)-\wt\Delta^\flat(x)=u\otimes1 + 1\otimes v$ in $(A\otimes_R A)^\flat$. We obtain $u=(\id_{A^\flat}\otimes\epsilon^\flat)(\Delta^\flat(x)-\wt\Delta^\flat(x))=x-x=0$ and $v=(\epsilon^\flat\otimes\id_{A^\flat})(\Delta^\flat(x)-\wt\Delta^\flat(x))=x-x=0$. This proves that $\Delta^\flat=\wt\Delta^\flat$.

The last assertion is standard. Namely, write $\Delta^\flat(x)-x\otimes 1-1\otimes x=\sum_i u_i\otimes v_i$ for $u_i=u_i'+u_i''$, $v_i=v_i'+v_i''$ with $u_i',v_i'\in\iota_{A^\flat}(R)$ and $u_i'',v_i''\in I_{A^\flat}$. Then $\sum_i u_i'v_i=(\epsilon^\flat\otimes\id_{A^\flat})(\sum_i u_i\otimes v_i)=x-\epsilon^\flat(x)-x=0$ implies $\sum_i u_i\otimes v_i=\sum_i u''_i\otimes v_i$. And $\sum_i u_i''v_i'=(\id_{A^\flat}\otimes\epsilon^\flat)(\sum_i u_i''\otimes v_i)=(\id_{A^\flat}\otimes\epsilon^\flat)(\Delta^\flat(x)-x\otimes 1-1\otimes x)=x-x-\epsilon^\flat(x)=0$ implies $\sum_i u''_i\otimes v_i=\sum_i u''_i\otimes v_i''$. 

\smallskip\noindent
\ref{Abra_E} We write $\CG=\Spec(A,A^\flat)$ with comultiplication $(\Delta,\Delta^\flat)$ and counit $(\epsilon,\epsilon^\flat)$, and $\CH=\Spec(\wt A,\wt A^\flat)$ with comultiplication $(\wt\Delta,\wt\Delta^\flat)$ and counit $(\tilde\epsilon,\tilde\epsilon^\flat)$. We also write $A=R[\ulX]/I$ locally on $S$. We have to show that $F:=(\Delta^\flat\circ f^\flat-(f^\flat\otimes f^\flat)\circ\wt\Delta^\flat)\colon\wt A^\flat\to(A\otimes_R A)^\flat$ is zero. From $\Delta\circ f=(f\otimes f)\circ\wt\Delta$ we see as in \ref{Abra_D} that for every $x\in\wt A^\flat$ there are $u,v\in I$ with $F(x)=u\otimes 1 + 1\otimes v$ in $(A\otimes_R A)^\flat$. Now 
\begin{eqnarray*}
u & = & (\id_{A^\flat}\otimes\epsilon^\flat)\circ F(x) \\
& = & (\id_{A^\flat}\otimes\epsilon^\flat)\circ\Delta^\flat\circ f^\flat(x)-(\id_{A^\flat}\otimes\epsilon^\flat)\circ(f^\flat\otimes f^\flat)\circ\wt\Delta^\flat(x)\\
& = & \id_{A^\flat}\circ f^\flat(x) - f^\flat\circ(\id_{A^\flat}\otimes\tilde\epsilon^\flat)\circ\wt\Delta^\flat(x)\\
& = & f^\flat(x)-f^\flat(x)\\
& = & 0\,,
\end{eqnarray*}
and likewise $v=0$. This shows that $F=0$ and $(f,f^\flat)\in\Hom_{\DGr_S}(\CG, \CH)$ as claimed.
\end{proof}
 
\bigskip

Let $\CG=(G,G^\flat,i_\CG)\in\DGr_S$. Faltings defines the \emph{co-Lie complex of $\CG$ over $S=\Spec R$} (that is, the fiber at the unit section of $G$ of the cotangent complex) as the complex of finite locally free $R$-modules
\begin{eqnarray}
\CoLA{\CG/S}\colon\qquad 0 \longto N_\CG \xrightarrow{\es d\;} t_\CG^\ast\longto 0
\end{eqnarray}
concentrated in degrees $-1$ and $0$ with differential $d:=\pi_\CG\circ\inclN_\CG$. Recall that the co-Lie complex of $G/S$ and more generally the cotangent complex of a morphism was defined by Illusie~\cite{Illusie71,Illusie72} generalizing earlier work of Lichtenbaum and Schlessinger~\cite{LS}; cf.\ Appendix~\ref{AppCotCom}. If $G=\Spec A$ for $A=R[\ulX]/I$ where $I$ is generated by a regular sequence then the cotangent complex of Illusie~\cite[II.1.2.3]{Illusie71} is quasi-isomorphic to the complex of finite locally free $A$-modules
\[
\CoCI{G/S}\colon\qquad 0\longto I/I^2 \xrightarrow{\es d\;} \Omega^1_{R[\ulX]/R}\otimes_{R[\ulX]}A \longto 0
\]
concentrated in degrees $-1$ and $0$ with $d$ being the differential map; see \cite[Corollaire~III.3.2.7]{Illusie71}. The \emph{co-Lie complex of $G$ over $S$} is defined by Illusie~\cite[\S\,VII.3.1]{Illusie72} as $\CoLI{G/S}:=\epsilon_G^*\CoCI{G/S}$ where $\epsilon_G\colon S\to G$ is the unit section. To see that this is equal to Faltings's definition note that
\begin{eqnarray*}
 \epsilon_G^\ast(I/{I^2}) &=&  I/{I^2}\otimes_A R \es = \es I \otimes_{R[\ulX]}R \es = \es  I / (I\cdot I_{R[\ulX]}) \es = \es N_\CG\qquad\text{and}\\[2mm]
 \epsilon_G^\ast(\Omega^1_ {R[\ulX]/R} \otimes_{R[\ulX]} A)  & = & \Omega^1_ {R[\ulX]/R} \otimes_{R[\ulX]} R \es = \es  \oplus_{\nu=1}^{n}R\cdot dX_\nu \es = \es I_{R[\ulX]}/I_{R[\ulX]}^2 \es = \es t_\CG^\ast,
\end{eqnarray*}
and that the differential of both co-Lie complexes sends an element $x\in I$ to the linear term in its expansion as a polynomial in $\ulX$, because all terms of higher degree are sent to zero under $\epsilon_G^\ast$.

Up to homotopy equivalence both $\CoCI{G/S}$ and $\CoL{G/S}$ only depend on $G$, and not on the presentation $A=R[\ulX]/I$ nor on the deformation $\CG$ of $G$. Note that $\CoCI{G/S}$ and $\iota^*\CoLI{G/S}$ are quasi-isomorphic by \cite[Chapter~II, Proposition~3.2.9]{Messing} where $\iota\colon G\to S$ is the structure map. 

\begin{Definition}\label{DefOmega}
We (re-)define the \emph{co-Lie module of $G$ over $S$} as $\omega_G:=\Koh^0(\CoL{\CG/S}):=\coker d$ and set $n_G:=\Koh^{-1}(\CoL{\CG/S}):=\ker d$. These $R$-modules only depend on $G$ and not on $\CG$. Since $\Koh^0(\CoCI{G/S})=\Omega^1_{G/S}$ we have $\omega_G=\epsilon_G^*\Omega^1_{G/S}$ which is also canonically isomorphic to the $R$-module of invariant differentials on $G$.
\end{Definition}

We record the following lemmas.

\begin{Lemma} \label{LemmaGEtale}
If $\CG\in\DGr_S$ the following are equivalent:
\begin{enumerate}
\setlength{\itemsep}{0ex}
\item \label{LemmaGEtale_A}
$G$ is \'etale over $S$,
\item \label{LemmaGEtale_B}
$\omega_G=0$,
\item \label{LemmaGEtale_C}
the differential of $\CoL{\CG/S}$ is an isomorphism.
\end{enumerate}
\end{Lemma}

\begin{proof}
If $G$ is \'etale then $\Omega^1_{G/S}=0$. Conversely, since $\Omega^1_{G/S}$ is a finitely generated $\CO_G$-module, $\omega_G=0$ implies by Nakayama that $G$ is \'etale along the zero section. Being a group scheme it is \'etale everywhere.

Clearly \ref{LemmaGEtale_C} implies \ref{LemmaGEtale_B}. Conversely if $\omega_G=0$, that is, if $d$ is surjective, then $d$ is also injective, because both $t^*_\CG$ and $N_\CG$ are finite locally free of the same rank by Lemma~\ref{LemmaNBLocFree}.
\end{proof}

\begin{Lemma}[{\cite[Chapter~II, Proposition~3.3.4]{Messing}}] \label{LemmaLongExSeq}
Let $0\to G'\to G\to G''\to0$ be an exact sequence of finite locally free group schemes over $S$. Then there is an exact sequence of $R$-modules
\[
0\,\longto\, n_{G''}\,\longto\, n_G\,\longto\, n_{G'}\,\longto\,\omega_{G''}\,\longto\,\omega_G\,\longto\,\omega_{G'}\,\longto\,0\,.
\]
In particular, if $G'\into G$ is a closed immersion then $\omega_G\onto\omega_{G'}$ is surjective.
\end{Lemma}

\section{Strict $\BF_q$-module schemes} \label{SectStrictF_qAct}

We keep the notation of the previous section. Let $\CO$ be a commutative unitary ring. 

\begin{Definition}
In this article an \emph{$\CO$-module scheme} over $S$ is a finite locally free commutative group scheme $G$ over $S$ together with a ring homomorphism $\CO\to\End_S(G)$. We denote the category of $\CO$-module schemes over $S$ by $\Gr(\CO)_S$.
\end{Definition}

\begin{Proposition} \label{PropCanonDecompAMod}
If $S$ is the spectrum of a field $L$ every $\CO$-module scheme $G$ over $S$ is canonically an extension $0\to G^0\to G\to G^\et\to 0$ of an \'etale $\CO$-module scheme $G^\et$ by a connected $\CO$-module scheme $G^0$. The $\CO$-module scheme $G^\et$ is the largest \'etale quotient of $G$. If $L$ is perfect, $G^\et$ is canonically isomorphic to the reduced closed $\CO$-module subscheme $G^\red$ of $G$ and the extension splits canonically, $G=G^0\times_S G^\red$.
\end{Proposition}

\begin{proof}
The constituents of the canonical decomposition of the finite $S$-group scheme $G$ are $\CO$-invariant.
\end{proof}

\begin{Definition}\label{DefStrictOAct}
Let $S = \Spec R$ be a scheme over $\CO$ and let $\CG\in\DGr_S$. A \emph{strict $\CO$-action} on $\CG$ is a homomorphism $\CO \to \End_{\DGr_S}(\CG)$  such that the induced action on $\CoLA{\CG/S}$ is \emph{equal} to the scalar multiplication via $\CO \to R$; compare Remark~\ref{RemStrictWithHomotopicAction}.

We let $\DGr(\CO)_S$ be the category whose objects are pairs $(\CG,[\,.\,])$ where $\CG\in\DGr_S$ and $[\,.\,]\colon\CO\to\End_{\DGr_S}(\CG)$, $a\mapsto[a]$ is a strict $\CO$-action, and whose morphisms $f\colon(\CG,[\,.\,])\to(\CG',[\,.\,]')$ are those morphisms $f\colon\CG\to\CG'$ in $\DGr_S$ which are compatible with the $\CO$-actions, that is, which satisfy $f\circ[a]=[a]'\circ f$ for all $a\in\CO$.

We let $\DGr^*(\CO)_S$ be the quotient category of $\DGr(\CO)_S$ having the same objects, whose morphisms are the equivalence classes of morphisms $(G,G^\flat)\to(H,H^\flat)$ in $\DGr(\CO)_S$ which induce the same morphism $G\to H$.
\end{Definition}

So by definition the forgetful functor $\DGr^*(\CO)_S\to\Gr(\CO)_S$, which sends $(G,G^\flat)$ to $G$ and morphisms $(G,G^\flat)\to(H,H^\flat)$ to their restriction to $G\to H$, is faithful.

Faltings~\cite[Remark~b) after Definition~1]{Faltings02} notes the following property of strict $\CO$-actions for which we include a proof.

\begin{Lemma}\label{LemmaStrictOAct}
A strict $\CO$-action $[\,.\,]$ on $\CG$ induces on every deformation $\wt\CG$ of $G$ a unique strict $\CO$-action $\wt{[\,.\,]}$ which is compatible with all lifts $\wt\CG\to\CG$ and $\CG\to\wt\CG$ of the identity on $G$. In particular, the pairs $(\CG,[\,.\,])$ and $(\wt\CG,\wt{[\,.\,]})$ are isomorphic in $\DGr^*(\CO)_S$.
\end{Lemma}

\begin{proof}
Let $\CG=\Spec(A,A^\flat)$ and $\wt\CG=\Spec(A,\wt A^\flat)$. By Lemma~\ref{LemmaExistsLift} we may choose lifts $f\colon A^\flat\to\wt A^\flat$ and $g\colon\wt A^\flat\to A^\flat$ of the identity on $A$, and there are homotopies $h\colon t^*_\CG\to N_\CG$ and $\tilde h\colon t^*_{\wt\CG}\to N_{\wt\CG}$ satisfying $gf=\id+\inclN h\pi$ and $fg=\id+\tilde \inclN \tilde h\tilde\pi$ where $\inclN:=\inclN_\CG$, $\pi:=\pi_\CG$ and $\tilde\inclN:=\inclN_{\wt\CG}$, $\tilde\pi:=\pi_{\wt\CG}$. If we have a strict $\CO$-action $\CO\to\End_{\DGr_S}(\wt\CG),a\mapsto\wt{[a]}$ on $\wt\CG$ satisfying $f\circ[a]=\wt{[a]}\circ f$ and $g\circ\wt{[a]}=[a]\circ g$ then we necessarily must have $\wt{[a]}=\wt{[a]}(fg-\tilde \inclN \tilde h\tilde\pi)=f\circ[a]\circ g-\tilde\inclN N_{\wt{[a]}}\tilde h\tilde\pi=f[a]g-a\,\tilde \inclN \tilde h\tilde\pi$. This proves uniqueness.

So to define a strict $\CO$-action on $\wt\CG$ we choose $f$ and $g$ as in the previous paragraph and we set $\wt{[a]}:=f[a]g-a\,\tilde \inclN \tilde h\tilde\pi\colon\wt A^\flat\to\wt A^\flat$. To show that this is a ring homomorphism we first note that it is $R$-linear.  To prove compatibility with multiplication let $b,c\in\wt A^\flat$. We write $b=b'+b''$ and $c=c'+c''$ with $b',c'\in\iota_{\wt A^\flat}(R)$ and $b'',c''\in I_{\wt A^\flat}$. Then $\tilde\pi(b)=(b''\mod I_{\wt A^\flat}^2)=:\ol{b''}$ and $\wt{[a]}(b)=f[a]g(b)-a\,\tilde n\tilde h(\ol{b''})=b'+f[a]g(b'')-a\,\tilde n\tilde h(\ol{b''})$ with $f[a]g(b'')\in I_{\wt A^\flat}$, because $f$, $[a]$ and $g$ are homomorphisms of augmented $R$-algebras. Since $\tilde n(N_{\wt\CG})^2=\tilde n(N_{\wt\CG})\cdot I_{\wt A^\flat}=(0)$ and $\ol{b''c''}=0$ we can compute
\begin{eqnarray*}
\wt{[a]}(b)\cdot\wt{[a]}(c) & = & f[a]g(bc)-b'\cdot a\,\tilde n\tilde h(\ol{c''})-c'\cdot a\,\tilde n\tilde h(\ol{b''})\\[2mm]
 & = & f[a]g(bc)-a\,\tilde n\tilde h(b'\,\ol{c''}+c'\,\ol{b''}+\ol{b''c''})\\[2mm]
 & = & f[a]g(bc)-a\,\tilde n\tilde h\tilde\pi(bc)\\[2mm]
 & = & \wt{[a]}(bc)\,.
\end{eqnarray*}

We next claim that the map $\CO\to\End_{\DGr_S}(\wt\CG),a\mapsto\wt{[a]}$ is a ring homomorphism. First of all $\wt{\,[1]\,}=fg-\tilde\inclN \tilde h\tilde\pi=\id$. Next note that 
\begin{eqnarray*}
\wt{[a]}+\wt{\,[b]\,}\! & \!\!\!\!\!:=\!\!\!\! & \wt m\circ\bigl((f[a]g-a\,\tilde \inclN \tilde h\tilde\pi)\otimes(f[b]g-b\,\tilde \inclN \tilde h\tilde\pi)\bigr)\circ\wt\Delta^\flat\\
& \!\!\!\!\!=\!\!\!\!\! & \wt m\!\circ\!\Bigl((f\otimes f)([a]\otimes[b])(g\otimes g) - (a\,\tilde \inclN \tilde h\tilde\pi)\otimes f[b]g - f[a]g\otimes(b\,\tilde \inclN \tilde h\tilde\pi) + (a\,\tilde \inclN \tilde h\tilde\pi)\otimes(b\,\tilde \inclN \tilde h\tilde\pi)\Bigr)\!\circ\!\wt\Delta^\flat\!,
\end{eqnarray*}
where $\wt m\colon(\wt A\otimes_R\wt A)^\flat\to\wt A^\flat$ is induced from the multiplication in the ring $\wt A^\flat$ and the homomorphism $(f[a]g-a\,\tilde \inclN \tilde h\tilde\pi)\otimes(f[b]g-b\,\tilde \inclN \tilde h\tilde\pi)\colon\wt A^\flat\otimes_R\wt A^\flat\to\wt A^\flat\otimes_R\wt A^\flat$ induces a homomorphism $(\wt A\otimes_R\wt A)^\flat\to(\wt A\otimes_R\wt A)^\flat$ denoted by the same symbol.
To prove $\wt{[a]}+\wt{\,[b]\,}\stackrel{!}{=}\wt{[a+b]}:=f[a+b]g - (a+b)\tilde \inclN \tilde h\tilde\pi$ we observe
\[
\wt m\circ(f\otimes f)([a]\otimes[b])(g\otimes g)\circ\wt\Delta^\flat\;=\;f\circ m\circ([a]\otimes[b])\circ\Delta^\flat\circ g\;=:\;f\circ[a+b]\circ g\,, 
\]
and we evaluate the claimed equality on $\wt X_\nu$ where $\wt A^\flat=R[\wt\ulX]/\tilde I\cdot I_{R[\wt\ulX]}$. For every $\nu$ there are $u_i,v_i\in I_{R[\wt\ulX]}$ such that $\wt\Delta^\flat(\wt X_\nu)=\wt X_\nu\otimes1+1\otimes \wt X_\nu+\sum_i u_i\otimes v_i$; see \ref{Abra_D} after Definition~\ref{DefDGr}. Now $\tilde\pi(1)=0$, as well as $(\tilde \inclN \tilde h\tilde\pi)(I_{R[\wt\ulX]})\subset\tilde I/\tilde I\cdot I_{R[\wt\ulX]}\subset\wt A^\flat$ and $f[b]g(I_{R[\wt\ulX]})\subset I_{R[\wt\ulX]}$ imply
\begin{eqnarray*}
\wt m\circ\bigl((a\,\tilde \inclN \tilde h\tilde\pi)\otimes f[b]g\bigr)\circ\Delta^\flat(\wt X_\nu) & = & a\,\tilde \inclN \tilde h\tilde\pi(\wt X_\nu)\qquad\text{and}\\[2mm]
\wt m\circ\bigl(f[a]g \otimes (b\,\tilde \inclN \tilde h\tilde\pi)\bigr)\circ\Delta^\flat(\wt X_\nu) & = & b\,\tilde \inclN \tilde h\tilde\pi(\wt X_\nu)\qquad\text{and}\qquad\bigl((a\,\tilde \inclN \tilde h\tilde\pi)\otimes(b\tilde \inclN \tilde h\tilde\pi)\bigr)\circ\Delta^\flat(\wt X_\nu)\es=\es 0\,.
\end{eqnarray*}
From this $\wt{[a]}+\wt{\,[b]\,}=\wt{[a+b]}$ follows. To prove that $\wt{[a]}\circ\wt{\,[b]\,}=\wt{[ab]}$, we use that $N_{[a]}=a$ and $t^*_{[b]}=b$ implies $[a]\inclN=\inclN N_{[a]}=a\cdot\inclN$ and $\pi[b]=t^*_{[b]}\pi=b\cdot\pi$, as well as $[a]g\tilde\inclN=[a]\inclN N_g=a\cdot g\tilde\inclN$ and $\tilde\pi f[b]=t^*_f\pi[b]=b\cdot\tilde\pi f$. We compute 
\begin{eqnarray*}
\wt{[a]}\circ\wt{\,[b]\,} & = & (f[a]g-a\,\tilde \inclN \tilde h\tilde\pi)\circ(f[b]g-b\,\tilde \inclN \tilde h\tilde\pi)\\
& = & f[a](\id+\inclN h\pi)[b]g - a\,\tilde \inclN \tilde h\tilde\pi f[b]g - b\cdot f[a]g\,\tilde \inclN \tilde h\tilde\pi + ab\, \tilde \inclN \tilde h\tilde\pi\tilde \inclN \tilde h\tilde\pi\\
& = & f[a][b]g+ab\,f\inclN h\pi g - ab\,\tilde\inclN\tilde h\tilde\pi fg - ab\,fg\tilde\inclN\tilde h\tilde\pi + ab\,\tilde\inclN \tilde h\tilde\pi\tilde\inclN\tilde h\tilde\pi\\
& = & f[ab]g + ab\,f(gf-\id)g - ab\,(fg-\id)fg-ab\,(fg-\tilde\inclN\tilde h\tilde\pi)\tilde\inclN\tilde h\tilde\pi \\
& = & f[ab]g-ab\,\tilde\inclN\tilde h\tilde\pi\\
& = & \wt{[ab]}\,.
\end{eqnarray*}
This proves that $\CO\to\End_{\DGr_S}(\wt\CG),a\mapsto\wt{[a]}$ is a ring homomorphism. From 
\[
\begin{array}{rcccccccl}
t^*_{\wt{[a]}}\tilde\pi & = &  \tilde\pi\wt{[a]} & = & \tilde\pi f[a]g-a\,\tilde\pi\tilde\inclN\tilde h\tilde\pi & = &  a\,\tilde\pi(fg-\tilde\inclN\tilde h\tilde\pi) & = & a\,\tilde\pi\qquad\text{and}\\
\tilde\inclN N_{\wt{[a]}} & = & \wt{[a]}\tilde\inclN & = & f[a]g\tilde\inclN -a\,\tilde\inclN\tilde h\tilde\pi\tilde\inclN & = & a(fg-\tilde\inclN\tilde h\tilde\pi)\tilde\inclN & = & a\,\tilde\inclN
\end{array}
\]
we conclude that $\CO\to\End_{\DGr_S}(\wt\CG),a\mapsto\wt{[a]}$ is a strict $\CO$-action on $\wt\CG$.

To prove that $\wt{[a]}$ is compatible with $f$ and $g$, we compute $\wt{[a]}f=f[a]gf-a\,\tilde\inclN\tilde h\tilde\pi f=f[a]+f[a]\inclN h\pi-a(fgf-f)=f[a]$ and $g\wt{[a]}=gf[a]g-a\,g\tilde\inclN\tilde h\tilde\pi=[a]g+\inclN h\pi[a]g-a(gfg-g)=[a]g$. If $f'\colon A^\flat\to\wt A^\flat$ and $g'\colon\wt A^\flat\to A^\flat$ are other lifts of the identity on $A$ then $f'=f+\tilde\inclN\tilde\ell\pi$ and $g'=g+\inclN\ell\tilde\pi$ for $R$-homomorphisms $\tilde\ell\colon t^*_\CG\to N_{\wt\CG}$ and $\ell\colon t^*_{\wt\CG}\to N_\CG$. Then $f'[a]=f[a]+\tilde\inclN\tilde\ell\pi[a]=\wt{[a]} f+a\,\tilde\inclN\tilde\ell\pi=\wt{[a]} f+\wt{[a]}\tilde\inclN\tilde\ell\pi=\wt{[a]} f'$ and $[a] g'=[a] g+[a]\inclN\ell\tilde\pi=g\wt{[a]}+a\,\inclN\ell\tilde\pi=g\wt{[a]}+\inclN\ell\tilde\pi\wt{[a]}=g'\wt{[a]}$. This proves the first part of the lemma.

Finally $f$ and $g$ are mutually inverse isomorphisms between $(\CG,[\,.\,])$ and $(\wt\CG,\wt{[\,.\,]})$ in $\DGr^*(\CO)_S$.
\end{proof}

\begin{Remark} \label{RemStrictWithHomotopicAction}
The co-Lie complex $\CoL{\CG/S}$ depends on the deformation $\CG$ of $G$. For another deformation $\wt\CG$ the complex $\CoL{\wt\CG/S}$ is homotopically equivalent to $\CoL{\CG/S}$. Therefore one might try to weaken Definition~\ref{DefStrictOAct} and only require that the action of $a\in\CO$ on $\CoL{\CG/S}$ is \emph{homotopic} to the scalar multiplication with $a$. We do not know whether this is equivalent to Definition~\ref{DefStrictOAct} and whether Lemma~\ref{LemmaStrictOAct} remains valid for general $\CO$. Both is true for the polynomial ring $\CO=\BF_p[a]$.
\end{Remark}

\begin{Remark} Note that there can be different non-isomorphic strict $\CO$-actions on a deformation $\CG$. For example let $G=\Balpha_p=\Spec R[X]/(X^p)$ and $A^\flat=R[X]/(X^{p+1})$. Let $\CO=\BF_p[a]$ be the polynomial ring in the variable $a$, and let $R$ be an $\CO$-algebra by sending $a$ to $0$ in $R$. For every $u\in R$ the endomorphism $[a]=0\colon\Balpha_p\to\Balpha_p,\,X\mapsto 0$ lifts to $[a]\colon A^\flat\to A^\flat,\,X\mapsto u X^p$. All these lifts define strict $\CO$-actions on $(G,\Spec A^\flat)$ which are non-isomorphic in $\DGr^*(\BF_p[a])_S$. In particular, the forgetful functor $\DGr^*(\BF_p[a])_S\to\Gr(\BF_p[a])_S$ is \emph{not} fully faithful.
\end{Remark}

In contrast, for $\CO=\BF_q$ we have the following

\begin{Lemma}\label{LemmaStrictF_qAction}
The forgetful functor $\DGr^*(\BF_q)_S\to\Gr(\BF_q)_S$ is fully faithful. In particular, if $G\in\Gr(\BF_q)_S$ and $\CG=(G,G^\flat)\in\DGr_S$ is a deformation of $G$, then there is at most one strict $\BF_q$-action on $\CG$ which lifts the action on $G$.
\end{Lemma}

\begin{proof}
Let $(\CG,[\,.\,])$ and $(\wt\CG,\wt{[\,.\,]})$ be in $\DGr^*(\BF_q)_S$ with $\CG=\Spec(A,A^\flat)$ and $\wt\CG=\Spec(\wt A,\wt A^\flat)$. Let $f\colon A\to\wt A$ be a morphism in $\Gr(\BF_q)_S$, that is $\wt{[a]}f=f[a]$. Take any lift $f^\flat\colon A^\flat\to\wt A^\flat$ of $f$. Then for each $a\in\BF_q$ there is an $R$-homomorphism $h_a\colon t^*_\CG\to N_{\wt\CG}$ with $\wt{[a]}f^\flat-f^\flat[a]=\tilde\inclN h_a\pi$. It satisfies $h_{ab}=a h_b+b h_a$ because $\pi[b]=t_{[b]}^*\pi=b\,\pi$ and $\wt{[a]}\tilde\inclN=\tilde\inclN N_{\wt{[a]}}=a\,\tilde\inclN$, and hence
$$\tilde\inclN h_{ab}\pi\,=\,\wt{[ab]}f^\flat-f^\flat[ab]\,=\,\wt{[a]}(\wt{\,[b]\,}f^\flat-f^\flat[b])+(\wt{[a]}f^\flat-f^\flat[a])[b]\,=\,\wt{[a]}\tilde\inclN h_b\pi+\tilde\inclN h_a\pi[b]\,=\,\tilde\inclN(ah_b+bh_a)\pi\,.$$
We claim that it also satisfies $h_{a+b}=h_a+h_b$. Namely 
\begin{eqnarray*}
\tilde\inclN h_{a+b}\pi & = & \wt{[a+b]}f^\flat-f^\flat[a+b] \\
& = & \wt m\circ(\wt{[a]}\otimes\wt{\,[b]\,})\circ\wt\Delta^\flat\circ f^\flat-f^\flat\circ m\circ([a]\otimes[b])\circ\Delta^\flat \\
& = & \wt m\circ(\wt{[a]}f^\flat\otimes\wt{\,[b]\,}f^\flat-f^\flat[a]\otimes f^\flat[b])\circ\Delta^\flat \\
& = & \wt m\circ\bigl((\wt{[a]}f^\flat-f^\flat[a])\otimes\wt{\,[b]\,}f^\flat + f^\flat[a]\otimes(\wt{\,[b]\,}f^\flat-f^\flat[b])\bigr)\\
& = & \wt m\circ(\tilde\inclN h_a\pi\otimes\wt{\,[b]\,}f^\flat+f^\flat[a]\otimes\tilde\inclN h_b\pi)\circ\Delta^\flat
\end{eqnarray*}
where $\wt m\colon(\wt A\otimes_R\wt A)^\flat\to\wt A^\flat$ is induced from the multiplication in the ring $\wt A^\flat$ and $(\wt{[a]}f^\flat\otimes\wt{\,[b]\,}f^\flat)\colon A^\flat\otimes_R A^\flat\to\wt A^\flat\otimes_R\wt A^\flat$ induces a homomorphism $(A\otimes_R A)^\flat\to(\wt A\otimes_R\wt A)^\flat$ denoted by the same symbol. We evaluate this expression on $X_\nu$ where $A^\flat=R[\ulX]/I\cdot I_{R[\ulX]}$. For every $\nu$ there are $u_i,v_i\in I_{R[\ulX]}$ such that $\Delta^\flat(X_\nu)=X_\nu\otimes1+1\otimes X_\nu+\sum_i u_i\otimes v_i$; see \ref{Abra_D} after Definition~\ref{DefDGr}. Now $\pi(1)=0$, as well as $(\tilde \inclN \tilde h_a\pi)(I_{R[\ulX]})\subset\tilde I/\tilde I\cdot I_{R[\wt\ulX]}\subset\wt A^\flat$ and $\wt{\,[b]\,}f^\flat(I_{R[\ulX]})\subset I_{R[\wt\ulX]}$ imply $\tilde\inclN h_{a+b}\pi(X_\nu)=\tilde\inclN h_a\pi(X_\nu)+\tilde\inclN h_b\pi(X_\nu)$ as desired. This proves $h_{a+b}=h_a+h_b$. If $a$ lies in the image $\BF_p$ of $\BZ$ in $\BF_q$ then $h_a=a\cdot h_1=0$. In other words $a\mapsto h_a,\,\BF_q\to\Hom_R(t^*_\CG,N_{\wt\CG})$ is an $\BF_p$-derivation. Since $\Omega^1_{\BF_q/\BF_p}=(0)$ we must have $h_a=0$ and $\wt{[a]}f^\flat=f^\flat[a]$ for all $a\in\BF_q$. This means that $(f,f^\flat)$ defines a morphism in $\DGr^*(\BF_q)_S$ which maps to $f$ under the forgetful functor. So this functor is fully faithful. The remaining assertion follows by taking $\wt A^\flat=A^\flat$, $\wt A=A$ and $f^\flat=\id$.
\end{proof}

\begin{Definition} \label{DefStrictF_qMod}
A finite locally free $\BF_q$-module scheme $G$ over $R$ is called a \emph{strict $\BF_q$-module scheme} if it lies in the essential image of the forgetful functor $\DGr^*(\BF_q)_S\to\Gr(\BF_q)_S$, that is, if it has a deformation $\CG$ carrying a strict $\BF_q$-action which lifts the $\BF_q$-action on $G$. We identify $\DGr^*(\BF_q)_S$ with the category of finite locally free strict $\BF_q$-module schemes over $S$.
\end{Definition}

\begin{Lemma}\label{LemmaStrictIsLocal}
For a finite locally free $\BF_q$-module scheme $G$ over $R$ the property of being a strict $\BF_q$-module scheme is local on $\Spec R$.
\end{Lemma}

\begin{proof}
Let $\wt\CG$ be a deformation of $G$ over $\Spec R$. Let $\Spec R_i\subset\Spec R$ be an open covering and let $\CG_i$ be deformations of $G\times_R\Spec R_i$ carrying a strict $\BF_q$-action which lifts the $\BF_q$-action on $G$. This action induces by Lemma~\ref{LemmaStrictOAct} a strict $\BF_q$-action on $\wt\CG\times_R\Spec R_i$ for all $i$. Above $\Spec R_i\cap\Spec R_j$ these actions coincide by Lemma~\ref{LemmaStrictF_qAction}, and hence they glue to a strict $\BF_q$-action on $\wt\CG$ as desired.
\end{proof}

\begin{Example} \label{ExStrictF_qAct}
We give some examples for finite locally free strict $\BF_q$-module schemes. Let $R$ be an $\BF_q$-algebra.

\smallskip\noindent
(a) Let $\Balpha_q = \Spec R[X]/{(X^q)}$ and $\Balpha_q^\flat = \Spec R[X]/{(X^{q+1})}$. Then $[a](X) = aX$ for $a\in\BF_q$ defines a strict $\BF_q$-action on $\CG=(\Balpha_q,\Balpha_q^\flat)$. Indeed, the co-Lie complex is 
\[
\CoL{\CG/S}\colon\qquad 0 \longto X^q \cdot R \longto X\cdot R \longto 0  
\]
with $d=0$ and $a\in\BF_q$ acts on it as scalar multiplication by $a$ because $N_{[a]}(X^q)=(aX)^q=aX^q$ and $t^*_{[a]}(X)=aX$. Therefore $\Balpha_q$ is a finite locally free strict $\BF_q$-module scheme.

\medskip\noindent
(b) On $\Balpha_p = \Spec R[X]/{(X^p)}$ there is an $\BF_q$-action given by $[a](X) = aX$. If $q\ne p$ it does \emph{not} lift to a strict $\BF_q$-action on $\Balpha_p^\flat = \Spec R[X]/{(X^{p+1})}$. Although we may lift the action to $\CG=(\Balpha_p,\Balpha_p^\flat)$ via $[a](X) = aX$, the co-Lie complex is 
\[
\CoL{\CG/S}\colon\qquad 0 \longto X^p \cdot R \longto X\cdot R \longto 0  
\]
and so $a\in\BF_q$ acts on $N_\CG$ by $a^p$ which is not scalar multiplication by $a$ when $a^p\ne a$. Any other lift $\wt{[a]}$ of the $\BF_q$-action on $\Balpha_p$ to $\CG$ satisfies $\wt{[a]}=[a]+\inclN h_a\pi$ for an $R$-homomorphism $h_a\colon t^*_\CG\to N_\CG$ and yields $\inclN N_{\wt{[a]}}=\wt{[a]}\inclN=[a]\inclN+\inclN h_a\pi\inclN=[a]\inclN=\inclN N_{[a]}$ because $\pi\inclN=d=0$ on $\CoL{\CG/S}$. So no such action is strict and $\Balpha_p$ is \emph{not} a strict $\BF_q$-module scheme.

\medskip\noindent
(c) The constant \'etale group scheme $\ul{\BF_q} = \Spec R[X]/{(X^q - X)}$ over $\Spec R$ and its deformation $\ul{\BF_q}^\flat=\Spec R[X]/(X^{q+1}-X^2)$ carry a strict $\BF_q$-action via $[a](X) = aX$. Indeed, the co-Lie complex is 
\[
\CoL{\CG/S}\colon\qquad 0 \longto (X-X^q) \cdot R \longto X\cdot R \longto 0  
\]
with $d\colon X-X^q\mapsto X$ and $a\in\BF_q$ acts on it by $N_{[a]}(X-X^q)=aX-(aX)^q=a(X-X^q)$ and $t^*_{[a]}(X)=aX$. Therefore $\ul{\BF_q}$ is a finite locally free strict $\BF_q$-module scheme.

\medskip\noindent
(d) The multiplicative group $\Bmu_p= \Spec R[X]/{(X^p-1)}$ has an $\BF_p$-action via $[a](X) = X^a$. This action does not lift to $\Bmu_p^\flat=\Spec R[X]/(X-1)^{p+1}$, because on $\Bmu_p^\flat$ we have $\Delta(X)=X\otimes X$ and hence $[a](X)=X^a$, which satisfies $[p](X)=X^p\ne1$. Therefore no deformation of $\Bmu_p$ can carry a strict $\BF_p$-action and $\Bmu_p$ is \emph{not} a strict $\BF_p$-module scheme. Note that nevertheless $\BF_p$ acts through scalar multiplication on the co-Lie complex $\CoL{\Bmu_p/S}$.
\end{Example}

Part (c) generalizes to the following

\begin{Lemma}\label{LemmaEtaleIsStrict}
Any finite \'etale $\BF_q$-module scheme is a finite locally free strict $\BF_q$-module. In particular, if $0\to G'\to G\to G''\to 0$ is an exact sequence of finite locally free $\BF_q$-module schemes with $G$ a strict $\BF_q$-module and $G''$ \'etale, then both $G'$ and $G''$ are strict $\BF_q$-modules.
\end{Lemma}

\begin{proof}
The first assertion was remarked by Faltings \cite[\S\,3, p.~252]{Faltings02} more generally for finite \'etale $\CO$-module schemes, and also follows from \cite[Proposition~2.1(6)]{Drinfeld87} and Theorem~\ref{ThmEqAModSch} below. 
For the convenience of the reader we include a direct proof which also works for general $\CO$. Let $G$ be a finite \'etale $\BF_q$-module scheme. Since it is clearly locally free we must prove its strictness. Locally we choose a presentation $G=\Spec A$ with $A=R[\ulX]/I$ and $A^\flat=R[\ulX]/I\cdot I_{R[\ulX]}$. Since $G$ is \'etale, its zero section is open and $A=R\times A_1$ is a product of rings, where $R=R[\ulX]/I_{R[\ulX]}$ corresponds to the zero section and $A_1$ to its complement. If $I_1:=\ker(R[\ulX]\onto A_1)$ then $I=I_{R[\ulX]}\cdot I_1$ and $I_{R[\ulX]}+I_1=(1)$. We fix elements $u_0\in I_{R[\ulX]}$ and $u_1\in I_1$ with $u_0+u_1=1$. For $a\in\BF_q$ let $[a]\colon A\to A$ denote its action on $A$ and lift it arbitrarily to an $R$-homomorphism $[a]\colon R[\ulX]\to R[\ulX]$. This lift satisfies $[a](I_{R[\ulX]})\subset I_{R[\ulX]}$ and $[a](I_1)\subset I_1$, because $[a]$ fixes the zero section of $G$ and stabilizes its complement. We define the $R$-homomorphism $[a]^\flat\colon R[\ulX]\to A^\flat$ by
\[
[a]^\flat(X_i)\;:=\;aX_iu_1+[a](X_i)(1-u_1) \;=\; aX_i(1-u_0)+[a](X_i)u_0 \;\in\; I_{R[\ulX]}\,. 
\]
Since $u_1\in I_1$ and $X_i,[a](X_i),u_0\in I_{R[\ulX]}$, it follows that $[a]^\flat(X_i)\equiv[a](X_i)\mod I=I_{R[\ulX]}I_1$ and $[a]^\flat(X_i)\equiv aX_i\mod I_{R[\ulX]}^2$. Therefore $[a]^\flat(I)\subset I$, whence $[a]^\flat(I\cdot I_{R[\ulX]})\subset I\cdot I_{R[\ulX]}$, and so this defines an $R$-homomorphism $[a]^\flat\colon A^\flat\to A^\flat$ which lifts the action of $[a]$ on $A$. Since for every $x\in I\subset I_{R[\ulX]}$ we have $[a]^\flat(x)-ax\in I_{R[\ulX]}^2$ and $x-xu_1=xu_0\in II_{R[\ulX]}$, we compute for $[a]^\flat(x)$ in $I/II_{R[\ulX]}=I_{R[\ulX]}I_1/I_{R[\ulX]}^2I_1$ that $[a]^\flat(x)=[a]^\flat(x)\cdot u_1=axu_1=ax$. In particular, $[a]$ acts as scalar multiplication by $a$ on $t_\CG^*=I_{R[\ulX]}/I_{R[\ulX]}^2$ and $N_\CG=I/II_{R[\ulX]}$. Moreover, this shows that the map $\BF_q\to\End_{R\text{-Alg}}(R[\ulX]/I_{R[\ulX]}^2), a\mapsto[a]^\flat$ is a ring homomorphism. Likewise $\BF_q\to\End_{R\text{-Alg}}(A_1), a\mapsto[a]^\flat=[a]$ is a ring homomorphism. Since $I\cdot I_{R[\ulX]}=I_{R[\ulX]}^2I_1$ and $I_{R[\ulX]}^2+I_1=(1)$ imply $A^\flat=R[\ulX]/I_{R[\ulX]}^2\times A_1$, it follows that the map $\BF_q\to\End_{\DGr_S}(\CG), a\mapsto([a],[a]^\flat)$ is a ring homomorphism. This defines a strict $\BF_q$-action on $\CG=(\Spec A,\Spec A^\flat)$.

The last assertion on the strictness of $G'$ can be proved on affine open subsets of $S$. There Lemma~\ref{LemmaStrictF_qAction} implies that the morphism $G\to G''$ is $\BF_q$-strict in the sense of Faltings \cite[Definition~1]{Faltings02}, and by \cite[Proposition~2]{Faltings02} its kernel $G'$ is a strict $\BF_q$-module.
\end{proof}

\section{Equivalence between finite $\BF_q$-shtukas and strict $\BF_q$-modules} \label{Relshtgroupscheme}

Let $S$ be a scheme over $\Spec \BF_q$. Recall that a finite locally free commutative group scheme $G$ over $S$ is equipped with a relative $p$-Frobenius $F_{p,G}\colon G\to\sigma_p^*G$ and a $p$-Verschiebung morphism $V_{p,G}\colon \sigma_p^*G\to G$ which satisfy $F_{p,G}\circ V_{p,G} = p\id_{\sigma_p^*G}$ and $V_{p,G} \circ F_{p,G} = p\id_G$. For more details see  \cite[Expos\'e~VII$_{\rm A}$, \S\,4.3]{SGA3}. Example~\ref{ExStrictF_qAct} is generalized by the following results of Abrashkin. The first is concerned with finite locally free strict $\BF_p$-module schemes.

\begin{Theorem}[{\cite[Theorem~1]{Abrashkin}}] \label{Theofequivalence1}
Let $G$ be a finite locally free group scheme equipped with an $\BF_p$-action over an $\BF_p$-scheme $S$. Then this action lifts (uniquely) to a strict $\BF_p$-action on some (any) deformation of $G$ if and only if the $p$-Verschiebung of $G$ is zero. In particular, the forgetful functor induces an equivalence between $\DGr^*(\BF_p)_S$ and the category of those group schemes in $\Gr(\BF_p)_S$ which have $p$-Verschiebung zero.
\end{Theorem}

To explain Abrashkin's classification of finite locally free strict $\BF_q$-module schemes we recall that Drinfeld \cite[\S\,2]{Drinfeld87} defined a functor from finite $\BF_q$-shtukas over $S$ to finite locally free $\BF_q$-module schemes over $S$. Abrashkin~\cite{Abrashkin} proved that the essential image of Drinfeld's functor consists of finite locally free strict $\BF_q$-module schemes. Other descriptions of the essential image were given by Taguchi~\cite[\S\,1]{Taguchi95} and Laumon~\cite[\S\,B.3]{Laumon}. (But note that \cite[Propositions~2.4.11, B.3.13 and Lemma B.3.16]{Laumon} are incorrect as the $\BF_q$-module scheme $G=\Balpha_p=\Spec R[x]/(x^p)$ shows when $p\ne q$.) Drinfeld's functor is defined as follows. Let $\ulM=(M,F_M)$ be a finite $\BF_q$-shtuka over $S$. Let 
\[
E\es=\es\ul{\Spec}_S\,\bigoplus_{n\ge0}\,\Sym^n_{\CO_S} M
\]
be the geometric vector bundle corresponding to $M$, and let $F_{q,E}\colon E\to\sigma_{\!q}^\ast E$ be its relative $q$-Frobenius morphism over $S$. On the other hand, the map $F_M$ induces another $S$-morphism $\Spec(\Sym^\bullet F_M)\colon E\to\sigma_{\!q}^\ast E$. Drinfeld defines 
\[
\Dr_q(\ulM)\es :=\es \ker \bigl(\Spec(\Sym^\bullet F_M)-F_{q,E}\colon E\to\sigma_{\!q}^\ast E\bigr)\es =\es \ul{\Spec}_S\;\bigl(\bigoplus_{n\ge0}\Sym^n_{\CO_S} M\bigr)/I
\]
where the ideal $I$ is generated by the elements $m^{\otimes q}-F_M(\sigma_{\!q}^*m)$ for all local sections $m$ of $M$. (Here $m^{\otimes q}$ lives in $\Sym^q_{\CO_S} M$ and $F_M(\sigma_{\!q}^*m)$ in $\Sym^1_{\CO_S} M$.) 

There is an equivalent description of $\Dr_q(\ulM)$ as follows. Let $S=\Spec R$ be affine and denote the $R$-module $\Gamma(S,M)$ again by $M$. Let $\Frob_{q,R}\colon R\to R$ be the $q$-Frobenius on $R$ with $x\mapsto x^q$. We equip $M$ with the $\Frob_{q,R}$-semi-linear endomorphism $F_M^\semi\colon M\to M$, $m\mapsto F_M(\sigma_{\!q}^*m)$, which satisfies $F_M^\semi(bm)=F_M\bigl(\sigma_{\!q}^*(bm)\bigr)=F_M(b^q\sigma_{\!q}^*m)=b^q F_M^\semi(m)$. Also we equip every $R$-algebra $T$ with the $\Frob_{q,R}$-semi-linear $R$-module endomorphism $F_T^\semi:=\Frob_{q,T}\colon T\to T$. Then $\Dr_q(\ulM)$ is the group scheme over $S$ which is given on $R$-algebras $T$ as
\[
\Dr_q(\ulM)(T) \;=\; \Hom_{F^\semi}(\ulM,T) \;:=\; \bigl\{\,h\in\Hom_{R\text{-Mod}}(M,T) \colon h(m)^q=h\bigl(F_M(\sigma_{\!q}^*m)\bigr)\;\forall\,m\in M\,\bigr\}\,,
\]
because $\Hom_{R\text{-Mod}}(M,T)\,=\,\Hom_{R\text{-Alg}}(\Sym^\bullet_R M,T)\,=\,E(\Spec T)$. We thank L.~Taelman for mentioning this to us.

$\Dr_q(\ulM)$ is an $\BF_q$-module scheme over $S$ via the comultiplication $\Delta\colon m\mapsto m\otimes1+1\otimes m$ and the $\BF_q$-action $[a]\colon m\mapsto am$ which it inherits from $E$. It has a canonical deformation
\[
\Dr_q(\ulM)^\flat\es :=\es \ul{\Spec}_S\;\bigl(\bigoplus_{n\ge0}\Sym^n_{\CO_S} M\bigr)/(I\cdot I_0),
\]
where $I_0=\bigoplus_{n\ge1}\,\Sym^n_{\CO_S} M$ is the ideal generated by all $m\in M$. This deformation is equipped with the comultiplication $\Delta^\flat\colon m\mapsto m\otimes1+1\otimes m$ and the $\BF_q$-action $[a]^\flat\colon m\mapsto am$. We set $\CDr_q(\ulM):=(\Dr_q(\ulM),\Dr_q(\ulM)^\flat)$. Its co-Lie complex is 
\begin{equation}\label{EqCoLieDr}
0\longto I/(I\cdot I_0) \longto I_0/I_0^2\longto 0
\end{equation}
with differential $d\colon m^{\otimes q}-F_M(\sigma_{\!q}^* m)\mapsto -F_M(\sigma_{\!q}^*m)$. On it $[a]$ acts by scalar multiplication with $a$ because $(am)^q-F_M(\sigma_{\!q}^*(am))=a^q(m^{\otimes q}-F_M(\sigma_{\!q}^*m))$. This defines the functor $\CDr_q\colon\FqSht_S\to\DGr(\BF_q)_S$. We also compose $\CDr_q$ with the projection to $\DGr^*(\BF_q)_S$. 

\medskip

Conversely, let $\CG=(G,G^\flat)=\Spec(A,A^\flat)\in\DGr(\BF_q)_S$ in the affine situation $S=\Spec R$. Note that on the additive group scheme $\BG_{a,S}=\Spec R[x]$ the elements $b\in R$ act via endomorphisms $\psi_b\colon\BG_{a,S}\to\BG_{a,S}$ given by $\psi_b^*\colon R[x]\to R[x],\,x\mapsto bx$. This makes $\BG_{a,S}$ into an $R$-module scheme, and in particular, into an $\BF_q$-module scheme via $\BF_q\subset R$. We associate with $\CG$ the $R$-module of $\BF_q$-equivariant homomorphisms on $S$
\[
M_q(\CG) \es :=\es \Hom_{R\text{\rm-groups},\BF_q\text{\rm-lin}}(G,\BG_{a,S}) \es = \es \bigl\{ x \in A\colon\Delta(x) = x \otimes 1 + 1 \otimes x,\ [a](x) = a x,\ \forall a \in \BF_q \bigr\}\,,
\]
with its action of $R$ via $R\to\End_{R\text{-groups},\BF_q\text{\rm-lin}}(\BG_{a,S})$. It is a finite locally free $R$-module by \cite[Proposition~3.6 and Remark~5.5]{Poguntke17}; see also \cite[VII$_{\rm A}$, 7.4.3]{SGA3} in the reedited version of SGA~3 by P.~Gille and P.~Polo. The composition on the left with the relative $q$-Frobenius endomorphism $F_{q,\BG_{a,S}}$ of $\BG_{a,S}=\Spec R[x]$ given by $x\mapsto x^q$ defines a map $M_q(\CG)\to M_q(\CG), m\mapsto F_{q,\BG_{a,S}}\circ m$ which is not $R$-linear, but $\sigma_{\!q}^*$-linear, because $F_{q,\BG_{a,S}}\circ \psi_b=\psi_{b^q}\circ F_{q,\BG_{a,S}}$. Therefore, $F_{q,\BG_{a,S}}$ induces an $R$-homomorphism $F_{M_q(\CG)}\colon \sigma_{\!q}^\ast M_q(\CG)\to M_q(\CG)$. Then $\ulM_q(\CG):=\bigl(M_q(\CG),F_{M_q(\CG)}\bigr)$ is a finite shtuka over $S$. Note that for $m\in M_q(\CG)$ the commutative diagram
\begin{equation}\label{EqDiagFrobenius}
\xymatrix {G\ar[rr]^{F_{q,G}}\ar[d]_m && \sigma_{\!q}^\ast G\ar[d]^{\sigma_{\!q}^\ast m}\\
\BG_{a,S}\ar[rr]^{F_{q,\BG_{a,S}}} && \BG_{a,S}  }
\end{equation}
implies that $F_{M_q(\CG)}(\sigma_{\!q}^*m) \,:=\, F_{q,\BG_{a,S}}\circ m\,=\,\sigma_{\!q}^\ast m \circ F_{q,G}$\,. If $\CH \in \DGr(\BF_q)_S$ and $(f,f^\flat)\colon \CG \to \CH$ is a morphism in the category $\DGr(\BF_q)_S$, then $\ulM_q(f)\colon\ulM_q(\CH)\to\ulM_q(\CG),\,m\mapsto m\circ f$. This defines the functor $\ulM_q\colon\DGr(\BF_q)_S\to\FqSht_S$. It factors through the category $\DGr^*(\BF_q)_S$ and further over the forgetful functor through the category of finite locally free strict $\BF_q$-module schemes over $S$.

There is a natural morphism $\ulM\to\ulM_q(\Dr_q(\ulM)),\,m\mapsto f_m$, where $f_m\colon\Dr_q(\ulM)\to\BG_{a,S}=\Spec R[x]$ is given by $f_m^*(x)=m$. There is also a natural morphism of group schemes $G\to\Dr_q(\ulM_q(G))$ given on the structure sheaves by $\bigoplus\limits_{n\ge0}\Sym^n_{\CO_S} M_q(G)/I\to\CO_G,\,m\mapsto m^*(x)$, which is well defined because 
\[
F_{M_q(G)}(\sigma_{\!q}^\ast m)^*(x)\;=\;(F_{q,\BG_{a,S}}\!\circ\, m)^*(x)\;=\;m^*(x^q)=(m^*(x))^q. 
\]

A large part of the following theorem was already proved by Drinfeld \cite[Proposition~2.1]{Drinfeld87} without using the notion of strict $\BF_q$-modules.

\begin{Theorem} \label{ThmEqAModSch}
\begin{enumerate}
\item \label{ThmEqAModSchItem1}
The contravariant functors $\Dr_q$ and $\ulM_q$ are mutually quasi-inverse anti-equi\-va\-len\-ces between the category of finite $\BF_q$-shtukas over $S$ and the category of finite locally free strict $\BF_q$-module schemes over $S$. 
\item \label{ThmEqAModSchItem2}
Both functors are $\BF_q$-linear and map short exact sequences to short exact sequences. They preserve \'etale objects and map the canonical decompositions from Propositions~\ref{PropCanonDecompAMod} and \ref{PropCanonDecompDMod} to each other.\\[-3mm]
\end{enumerate}
\noindent Let $\ulM=(M,F_M)$ be a finite $\BF_q$-shtuka over $S$ and let $G=\Dr_q(\ulM)$. Then
\begin{enumerate}
\setcounter{enumi}{2}
\item \label{ThmEqAModSchItem4}
the natural morphisms $\ulM\to\ulM_q(\Dr_q(\ulM)),\,m\mapsto f_m$ and $G\to\Dr_q(\ulM_q(G))$ are isomorphisms.
\item \label{ThmEqAModSchItem3}
the $\BF_q$-module scheme $\Dr_q(\ulM)$ is radicial over $S$ if and only if $F_M$ is nilpotent locally on $S$.
\item \label{ThmEqAModSchItem5}
the order of the $S$-group scheme $\Dr_q(\ulM)$ is $q^{\rk M}$.
\item \label{ThmEqAModSchItem6}
the co-Lie complex $\CoL{\CDr_q(\ulM)/S}$ is canonically isomorphic to the complex \ $0\to\sigma_{\!q}^*M\xrightarrow{\;F_M}M\to0$. In particular, $\omega_{\Dr_q(\ulM)}=\coker F_M$ and $n_{\Dr_q(\ulM)}=\ker F_M$.
\end{enumerate}
\end{Theorem}

\begin{proof} 
Assertions \ref{ThmEqAModSchItem1} and \ref{ThmEqAModSchItem4} were proved by Abrashkin~\cite[Theorem~2]{Abrashkin} in terms of the category $\DGr^*(\BF_q)_S$.

\smallskip\noindent
\ref{ThmEqAModSchItem2} The $\BF_q$-linearity is clear from the definitions and the compatibility with \'etale objects follows from \ref{ThmEqAModSchItem6} and Lemma~\ref{LemmaGEtale}. Let $0\to\ulM''\to\ulM\to\ulM'\to0$ be a short exact sequence of finite $\BF_q$-shtukas. Then by construction $\Dr_q(\ulM')\to\Dr_q(\ulM)$ is a closed immersion. Using \ref{ThmEqAModSchItem1} we consider the local sections of $M''=\CHom_{S\text{\rm-groups},\BF_q\text{\rm-lin}}(\Dr_q(\ulM''),\BG_{a,S})$ which are obtained by the closed immersion $\Dr_q(\ulM'')\into\ul{\Spec}_S(\Sym^\bullet_{\CO_S} M'')$ composed with local coordinate functions on $\ul{\Spec}_S(\Sym^\bullet_{\CO_S} M'')$. These local sections go to zero in $M'$ and this yields a morphism $\Dr_q(\ulM)/\Dr_q(\ulM')\to\Dr_q(\ulM'')$. The latter must be an isomorphism by \ref{ThmEqAModSchItem1} due to the identification
\[
\ulM_q\bigl(\Dr_q(\ulM)/\Dr_q(\ulM')\bigr)\es = \es\ker\bigl(\ulM_q(\Dr_q(\ulM))\to\ulM_q(\Dr_q(\ulM'))\bigr) \es = \es \ulM'' \es = \es \ulM_q(\Dr_q(\ulM'')).
\]
Conversely let $0\to G'\to G\to G''\to0$ be a short exact sequence of finite locally free strict $\BF_q$-module schemes.  Then the exactness of $0\to\ulM_q(G'')\to\ulM_q(G)\to\ulM_q(G')$ is obvious. Applying $\Dr_q$, whose exactness we just established, to the injection $\ulM_q(G)/\ulM_q(G'')\to\ulM(G')$ yields an isomorphism $\Dr_q\bigl(\ulM_q(G)/\ulM_q(G'')\bigr)=\ker(G\to G'')=G'$. From \ref{ThmEqAModSchItem1} it follows that $\ulM_q(G)/\ulM_q(G'')\to\ulM(G')$ is an isomorphism.

Finally consider the exact sequences from Propositions~\ref{PropCanonDecompAMod} and \ref{PropCanonDecompDMod}. Then $\Dr_q(\ulM_\et)$ is an \'etale quotient of $\Dr_q(\ulM)$. This yields a morphism $\Dr_q(\ulM)^\et\to\Dr_q(\ulM_\et)$. Conversely $\ulM_q(G^\et)$ is an \'etale $\BF_q$-subshtuka of $\ulM_q(G)$. This yields a morphism $\ulM_q(G^\et)\to\ulM_q(G)_\et$. The equivalence of \ref{ThmEqAModSchItem1} shows that both morphisms are isomorphisms. This proves the compatibility of $\Dr_q$ and $\ulM_q$ with the canonical decompositions.

\smallskip\noindent
\ref{ThmEqAModSchItem3} By definition $G:=\Dr_q(\ulM)$ is radicial over $S$ if $G(K)\to S(K)$ is injective for all fields $K$. This can be tested by applying the base change $\Spec K\to S$. By \ref{ThmEqAModSchItem2} and Propositions~\ref{PropCanonDecompAMod} and \ref{PropCanonDecompDMod} the base change $G\times_S\Spec K$ is connected if and only if $F_M\otimes\id_K$ is nilpotent. This implies \ref{ThmEqAModSchItem3} over $\Spec K$. It remains to show that $F_M$ is nilpotent locally on $S$ if $G$ is radicial. Locally on an affine open $\Spec R\subset S$ we may choose an $R$-basis of $M$ and write $F_M$ as an $r\times r$-matrix where $r=\rk M$. For every point $s\in S$, Proposition~\ref{PropCanonDecompDMod} implies that $F_M^r=0$ in $\kappa(s)^{r\times r}$. Therefore the entries of the matrix $F_M^r$ lie in the nil-radical of $R$. If $n$ is an integer such that their $q^n$-th powers are zero then $F_M^{r(n+1)}=F_M^r\cdot\ldots\cdot\sigma_{\!q}^{n*}(F_M^r)=0$. This establishes \ref{ThmEqAModSchItem3}.

\smallskip\noindent
\ref{ThmEqAModSchItem5} If locally on $S$ we choose an isomorphism $M\cong\bigoplus_{\nu=1}^n\CO_S\cdot X_\nu$ and let $(t_{ij})$ be the matrix of the morphism $F_M\colon\sigma_{\!q}^\ast M\to M$ with respect to the basis $(X_1,\ldots,X_n)$, then $\Dr_q(\ulM)$ is the subscheme of $\BG_{a,S}^n$, given by the system of equations 
\[
X_j^q \;=\; \DS \sum_{i=1}^{n}t_{ij}X_i\quad\text{for } j = 1,\ldots,n.
\]
Therefore $\CO_{\Dr_q(\ulM)}$ is a free $\CO_S$-module with basis $X_1^{m_1}\cdot\ldots\cdot X_n^{m_n}, \ 0 \leq m_i<q$. Thus $\ord\Dr_q(\ulM):=\rk_{\CO_S}\CO_{\Dr_q(\ulM)}=q^{\rk M}$.

\smallskip\noindent
\ref{ThmEqAModSchItem6} In the presentation of $\CoL{\CDr_q(\ulM)/S}$ given in \eqref{EqCoLieDr} with $I=(m^{\otimes q}-F_M(\sigma_{\!q}^*m)\colon m\in M)$ and $I_0=\bigoplus_{n\ge1}\,\Sym^n_{\CO_S} M$ we use the isomorphisms of $\CO_S$-modules $M\isoto I_0/I_0^2,\,m\mapsto m$ and $\sigma_{\!q}^*M\isoto I/(I\,I_0),\,b\,\sigma_{\!q}^*m=m\otimes b\mapsto b\,F_M(\sigma_{\!q}^*m)-b\,m^{\otimes q}$. Note that the latter is surjective by definition and injective because both $\sigma_{\!q}^*M$ and $I/(I\,I_0)$ are locally free $\CO_S$-modules of the same rank.
\end{proof}

\begin{Remark}\label{RemPoguntke}
Finite locally free strict $\BF_q$-module schemes over $S=\Spec R$ were equivalently described by Poguntke~\cite{Poguntke17} as follows. He defines the category $\BF_q\text{-gr}^{+,\text{b}}_S$ of finite locally free $\BF_q$-module schemes $G=\Spec A$ which locally on $S$ can be embedded into $\BG_{a,S}^N$ for some set $N$ and are \emph{balanced} in the following sense. The $R$-module 
\[
M_p(G)\;:=\;\Hom_{S\text{-groups}}(G,\BG_{a,S})\;=\;\bigl\{\,x \in A\colon\Delta(x) = x \otimes 1 + 1 \otimes x\, \bigr\}
\]
of morphisms of group schemes over $R$ decomposes under the action of $\BF_q$ on $G$ into eigenspaces 
\[
M_p(G)_{p^i}\;:=\;\{\,m\in M_p(G)\colon [\alpha](m)=\alpha^{p^i}\cdot m\text{ for all }\alpha\in\BF_q\,\}
\]
for $i\in\BZ/e\BZ$ where $q=p^e$. Now $G$ is \emph{balanced} if the composition on the right with the relative $p$-Frobenius $F_{p,\BG_{a,S}}$ of the additive group scheme $\BG_{a,S}$ induces isomorphisms $M_p(G)_{p^i}\isoto M_p(G)_{p^{i+1}}$ for all $i=0,\ldots,e-2$. Note that it is neither required nor implied that also $M_p(G)_{p^{e-1}}\to M_p(G)_1=:M_q(G)$ is an isomorphism. The latter holds if and only if $G$ is \'etale by Theorem~\ref{ThmEqAModSch}\ref{ThmEqAModSchItem2}.

Abrashkin~\cite[2.3.2]{Abrashkin} already showed that every finite locally free strict $\BF_q$-module scheme over $S$ belongs to $\BF_q\text{-gr}^{+,\text{b}}_S$. And Poguntke~\cite[Theorem~1.4]{Poguntke17} conversely shows that $\Dr_q$ and $\ulM_q$ provide an anti-equivalence between the category of finite $\BF_q$-shtukas over $S$ and the category $\BF_q\text{-gr}^{+,\text{b}}_S$.
\end{Remark}

\section{Relation to global objects}\label{SectGlobalObj}

Without giving proofs, we want in this section to relate local shtukas and divisible local Anderson modules (defined in the next section), as well as finite $\BF_q$-shtukas and finite locally free strict $\BF_q$-module schemes to global objects like $A$-motives, global shtukas, Drinfeld modules, Anderson $A$-modules, etc. which are defined as follows. Let $C$ be a smooth, projective, geometrically irreducible curve over $\BF_q$. For an $\BF_q$-scheme $S$ we set $C_S:=C\times_{\BF_q}S$ and we consider the endomorphism $\sigma_{\!q}:=\id_C\otimes\Frob_{q,S}\colon C_S\to C_S$.

\begin{Definition}\label{DefGlobalShtuka}
\begin{enumerate}
\item
Let $n$ and $r$ be positive integers. A \emph{global shtuka of rank $r$ with $n$ legs} over an $\BF_q$-scheme $S$ is a tuple $\ulCN=(\CN,\charmorph_1,\ldots,\charmorph_n,\tau_\CN)$ consisting of
\begin{itemize}
\item 
a locally free sheaf $\CN$ of rank $r$ on $C_S$,
\item 
$\BF_q$-morphisms $\charmorph_i\colon S\to C$ called the \emph{legs of $\ulCN$} and
\item 
an isomorphism $\tau_\CN\colon\sigma_{\!q}^*\CN|_{C_S\setminus\bigcup_i\Gamma_{\charmorph_i}}\isoto\CN|_{C_S\setminus\bigcup_i\Gamma_{\charmorph_i}}$ outside the graphs $\Gamma_{\charmorph_i}$ of the $\charmorph_i$.
\end{itemize}
In this article we will only consider the case where $\Gamma_{\charmorph_i}\cap\Gamma_{\charmorph_j}=\emptyset$ for $i\ne j$.
\item
A global shtuka over $S$ is a \emph{Drinfeld shtuka} if $n=2$, $\Gamma_{\charmorph_1}\cap\Gamma_{\charmorph_2}=\emptyset$, and $\tau_\CN$ satisfies $\tau_\CN(\sigma_{\!q}^*\CN)\subset\CN$ on $C_S\setminus\Gamma_{\charmorph_2}$ with cokernel locally free of rank $1$ as $\CO_S$-module, and  $\tau_\CN^{-1}(\CN)\subset\sigma_{\!q}^*\CN$ on $C_S\setminus\Gamma_{\charmorph_1}$ with cokernel locally free of rank $1$ as $\CO_S$-module.
\end{enumerate}
Drinfeld shtukas were introduced by Drinfeld~\cite{Drinfeld87} under the name $F$-sheaves.
\end{Definition}

An important class of special examples is defined as follows. Let $\infty\in C$ be a closed point and put $A:=\Gamma(C\setminus\{\infty\},\CO_C)$. Then $\Spec A=C\setminus\{\infty\}$. We will consider affine $A$-schemes $\charmorph\colon S=\Spec R\to\Spec A$ and the ideal $J:=(a\otimes 1-1\otimes\charmorph^*(a)\colon a\in A)\subset A_R:=A\otimes_{\BF_q}R$ whose vanishing locus $\Var(J)$ is the graph $\Gamma_\charmorph$ of the morphism $\charmorph$. The endomorphism $\sigma_{\!q}:=\id_C\otimes\Frob_{q,S}\colon C_S\to C_S$ induces the ring endomorphism $\sigma_{\!q}^*:=\id_A\otimes\Frob_{q,R}\colon A_R\to A_R,\,a\otimes b\mapsto a\otimes b^q$ of $A_R$ for $a\in A$ and $b\in R$. The following definition generalizes Anderson's~\cite{Anderson} notion of \emph{$t$-motives}, which is obtained as the special case, where $C=\BP^1$, $A=\BF_q[t]$ and $R$ is a field.

\begin{Definition}\label{DefAMotive}
Let $d$ and $r$ be positive integers and let $S=\Spec R$ be an affine $A$-scheme. An \emph{effective $A$-motive of rank $r$ and dimension $d$} over $S$ is a pair $\ulN=(N,\tau_N)$ consisting of a locally free $A_R$-module $N$ of rank $r$ and a morphism $\tau_N\colon\sigma_{\!q}^*N\to N$ of $A_R$-modules, such that $\coker\tau_N$ is a locally free $R$-module of rank $d$ and $J^d\cdot\coker\tau_N=0$. More generally, an \emph{$A$-motive of rank $r$} over $S$ is a pair $\ulN=(N,\tau_N)$ consisting of a locally free $A_R$-module $N$ of rank $r$ and an isomorphism $\tau_N\colon\sigma_{\!q}^*N|_{\Spec A_R\setminus\Var(J)}\isoto N|_{\Spec A_R\setminus\Var(J)}$ outside the vanishing locus $\Var(J)=\Gamma_\charmorph$ of $J$.
\end{Definition}

\begin{Example}\label{ExampleShtukaAMotive}
(a) If $\ulCN=(\CN,\charmorph_1,\charmorph_2,\tau_\CN)$ is a global shtuka of rank $r$ over $S=\Spec R$ with two legs such that $\charmorph_1=\charmorph$ and $\charmorph_2\colon S\to\{\infty\}\subset C$, then $\ulN(\ulCN):=(N,\tau_N):=\bigl(\Gamma(\Spec A_R,\CN),\tau_\CN\bigr)$ is an $A$-motive of rank $r$ over $S$.

\medskip\noindent
(b) Conversely, if $\infty\in C(\BF_q)$, every $A$-motive $\ulN=(N,\tau_N)$ over an affine $A$-scheme $\charmorph\colon S=\Spec R\to\Spec A$ can be obtained from a global shtuka $\ulCN=(\CN,\charmorph_1,\charmorph_2,\tau_\CN)$ by taking $\charmorph_1=\charmorph$ and $\charmorph_2\colon S\to\{\infty\}\subset C$, and taking $\CN$ as an extension to $C_S$ of the sheaf associated with $N$ on $\Spec A_R$, and $\tau_\CN=\tau_N$.
\end{Example}

These global objects give rise to finite and local shtukas, and that motivates the names for the latter. 

\begin{Example}\label{ExFiniteShtuka}
(a) Let $i\colon D\hookrightarrow C$ be a finite closed subscheme and let $\ulCN=(\CN,\charmorph_1,\ldots,\charmorph_n,\tau_\CN)$ be a global shtuka of rank $r$ over $S$ such that $\tau_\CN(\sigma_{\!q}^*\CN)\subset\CN$ in a neighborhood of $D_S:=D\times_{\BF_q}S$. (For example this is satisfied if $\ulCN$ is a Drinfeld-shtuka and $D_S\cap\Gamma_{\charmorph_2}=\emptyset$ or if $\ulCN$ is as in Example~\ref{ExampleShtukaAMotive} with $\ulN(\ulCN)$ an effective $A$-motive and $D\subset\Spec A$.) Then
\[
(M,F_M)\es:=\es(i^\ast\CN\,,\,i^\ast\tau_\CN)
\]
is a finite $\BF_q$-shtuka over $S$, because $M$ is locally free over $S$ of rank $r\cdot\dim_{\BF_q}\CO_D$. The sense in which $\ulCN$ is \emph{global} and $(M,F_M)$ is \emph{finite}, is with respect to the \emph{coefficients}: $\ulCN$ lives over all of $C$ and $\ulM$ lives over the finite scheme $D$. This example gave rise to the name ``finite $\BF_q$-shtuka''.

\bigskip\noindent 
(b) Let $v\in C$ be a closed point defined by a sheaf of ideals $\Fp\subset\CO_C$, let $\hat q$ be the cardinality of the residue field $\BF_v$ of $v$, let $f:=[\BF_v:\BF_q]$, and let $z\in\BF_q(C)$ be a uniformizing parameter at $v$. Let $\ulCN=(\CN,\charmorph_1,\ldots,\charmorph_n,\tau_\CN)$ be a global shtuka of rank $r$ over $S=\Spec R$ such that for some $i$ the elements of $\charmorph_i^*(\Fp)$ are nilpotent in $R$. Set $\zeta:=\charmorph_i^*(z)\in R$. Then the formal completion of $C_S$ along the graph $\Gamma_{\charmorph_i}$ of $\charmorph_i$ is canonically isomorphic to $\Spf R\dbl z\dbr$ by \cite[Lemma~5.3]{AH_Local}. The formal completion $M$ of $(\CN,\tau_\CN)$ along $\Gamma_{\charmorph_i}$ together with $\tau_M:=\tau_\CN^f\colon\sigma_{\!q}^{f*}M[\tfrac{1}{z-\zeta}]\isoto M[\tfrac{1}{z-\zeta}]$ is a local shtuka over $S$ of rank $r$ (as in Definition~\ref{Def1.1} with $q$ and $\BF_q\dbl z\dbr$ and $\sigma_{\!q}^*$ replaced by $\hat q$ and $\BF_v\dbl z\dbr$ and $\sigma_{\!q}^{f*}$). See \cite[\S\,6]{HartlIsog} for more details. Again $\ulM$ is \emph{local} with respect to the \emph{coefficients} as it lives over the complete local ring $\wh\CO_{C,v}=\BF_v\dbl z\dbr$ of $C$ at $v$. This gave rise to the name ``local shtuka''.
\end{Example}

\medskip

So far we discussed the semi-linear algebra side given by shtukas. On the side of group schemes, an important source from which the corresponding strict $\BF_q$-module schemes arise are Drinfeld $A$-modules, or more generally abelian Anderson $A$-modules. To define them, let $\charmorph\colon S=\Spec R\to\Spec A$ be an affine $A$-scheme. Recall that for a smooth commutative group scheme $E$ over $\Spec R$ the co-Lie module $\omega_E:=\epsilon_E^*\Omega^1_{E/R}$ is a locally free $R$-module of rank equal to the relative dimension of $E$ over $R$. Moreover, on the additive group scheme $\BG_{a,R}=\Spec R[x]$ the elements $b\in R$, and in particular $\charmorph^*(a)\in R$ for $a\in\BF_q\subset A$, act via endomorphisms $\psi_b\colon\BG_{a,R}\to\BG_{a,R}$ given by $\psi_b^*\colon R[x]\to R[x],\,x\mapsto bx$. This makes $\BG_{a,R}$ into an $\BF_q$-module scheme. In addition, let $\tau:=F_{q,\BG_{a,R}}$ be the relative $q$-Frobenius endomorphism of $\BG_{a,R}=\Spec R[x]$ given by $x\mapsto x^q$. It satisfies $\tau\circ \psi_b=\psi_{b^q}\circ\tau$. We let $R\{\tau\}:=\bigl\{\sum_{i=0}^nb_i\tau^i\colon n\in\BN_0, b_i\in R\bigr\}$ with $\tau b=b^q\tau$ be the non-commutative polynomial ring in the variable $\tau$ over $R$. There is an isomorphism of rings $R\{\tau\}\isoto\End_{R\text{-groups},\BF_q\text{-lin}}(\BG_{a,R})$ sending an element $f=\sum_ib_i\tau^i\in R\{\tau\}$ to the $\BF_q$-equivariant endomorphism $f\colon\BG_{a,R}\to\BG_{a,R}$ given by $f^*(x):=\sum_ib_i x^{q^i}$.

\begin{Definition}\label{DefAndModule}
Let $d$ and $r$ be positive integers. An \emph{abelian Anderson $A$-module of rank $r$ and dimension $d$} over an affine $A$-scheme $\charmorph\colon\Spec R\to\Spec A$ is a pair $\ulE=(E,\phi)$ consisting of a smooth affine group scheme $E$ over $\Spec R$ of relative dimension $d$, and a ring homomorphism $\phi\colon A\to\End_{R\text{-groups}}(E),\,a\mapsto\phi_a$ such that
\begin{enumerate}
\item \label{DefAndModule_A}
there is a faithfully flat ring homomorphism $R\to R'$ for which $E\times_R\Spec R'\cong\BG_{a,R'}^d$ as $\BF_q$-module schemes, where $\BF_q$ acts on $E$ via $\phi$ and $\BF_q\subset A$,
\item \label{DefAndModule_B}
$(a\otimes1-1\otimes\charmorph^*a)^d\cdot\omega_E=0$ for all $a\in A$ under the action of $a\otimes1$ induced from $\phi_a$ and the natural action of $1\otimes b$ for $b\in R$ on the $R$-module $\omega_E$,
\item \label{DefAndModule_C}
the set $N:=M_q(\ulE):=\Hom_{R\text{-groups},\BF_q\text{\rm-lin}}(E,\BG_{a,R})$ of $\BF_q$-equivariant homomorphisms of $R$-group schemes is a locally free $A_R$-module of rank $r$ under the action given on $m\in N$ by
\[
\begin{array}{rll}
A\ni a\colon & N\longto N, & m\mapsto m\circ \phi_a\\[2mm]
R\ni b\colon & N\longto N, & m\mapsto \psi_b\circ m
\end{array}
\]
\end{enumerate}
If $d=1$ this is called a \emph{Drinfeld $A$-module} over $S$; compare \cite[Theorem~2.13]{HartlIsog}.
\end{Definition}

The case in which $C=\BP^1$, $A=\BF_q[t]$, and $R$ is a field was considered by Anderson~\cite{Anderson} under the name \emph{abelian $t$-module}. In \cite[Theorem~2.10]{HartlIsog} we give a proof the following relative version of Anderson's theorem~\cite[Theorem~1]{Anderson}.

\begin{Theorem}\label{ThmMotiveOfAModule}
If $\ulE=(E,\phi)$ is an abelian Anderson $A$-module of rank $r$ and dimension $d$, we consider in addition on $N:=M_q(\ulE)$ the map $\tau_N^\semi\colon m\mapsto F_{q,\BG_{a,R}}\!\circ\, m$. Since $\tau_N^\semi(bm)=b^q\tau_N^\semi(m)$ for $b\in R$, the map $\tau_N^\semi$ is $\sigma_{\!q}$-semilinear and induces an $A_R$-linear map $\tau_N\colon\sigma_{\!q}^*N\to N$ with $\tau_N^\semi(m)=\tau_N(\sigma_{\!q}^*m)$. Then $\ulM_q(\ulE):=(N,\tau_N)$ is an effective $A$-motive of rank $r$ and dimension $d$. There is a canonical isomorphism of $R$-modules
\begin{equation}\label{EqT0}
\coker\tau_N\;\isoto\;\omega_E\,,\quad m\mod\tau_N(\sigma_{\!q}^*N)\;\longmapsto\; m^*(1)\,,
\end{equation}
where $m^*(1)$ means the image of $1\in\omega_{\BG_{a,R}}=R$ under the induced $R$-homomorphism $m^*\colon\omega_{\BG_{a,R}}\to\omega_E$. 

The contravariant functor $\ulE\mapsto\ulM_q(\ulE)$ is fully faithful. Its essential image consists of all effective $A$-motives $\ulN=(N,\tau_N)$ over $R$ for which there exists a faithfully flat ring homomorphism $R\to R'$ such that $N\otimes_RR'$ is a finite free left $R'\{\tau\}$-module under the map $\tau\colon N\to N,\,m\mapsto\tau_N(\sigma_{\!q}^*m)$.
\end{Theorem}

\begin{Example}\label{ExTorsionOfAModule}
Let $\ulE=(E,\phi)$ be an abelian Anderson $A$-module over an affine $A$-scheme $\charmorph\colon\Spec R\to\Spec A$, and let $\ulN:=\ulM_q(\ulE)$ be its associated effective $A$-motive. 

\medskip\noindent
(a) Let $\Fa\subset A$ be a non-zero ideal. By \cite[Theorem~5.4]{HartlIsog} the \emph{$\Fa$-torsion submodule of $E$}, defined as the scheme-theoretic intersection 
\[
\ulE[\Fa]\es:=\es\bigcap_{a\in \Fa} \ker(\phi_a\colon E\to E)\,,
\]
is a finite locally free $A/\Fa$-module scheme and a strict $\BF_q$-module scheme over $S$, which satisfies $\ulM_q(\ulE[\Fa])=\ulN/\Fa\ulN$ and $\ulE[\Fa]=\Dr_q(\ulN/\Fa\ulN)$. 

\medskip\noindent
(b) Let $\Fp\subset A$ be a maximal ideal and assume that the elements of $\charmorph^*(\Fp)\subset R$ are nilpotent. Let $\hat q$ be the cardinality of the residue field $\BF_\Fp:=A/\Fp$ and let $f:=[\BF_\Fp:\BF_q]$. We fix a uniformizing parameter $z\in\BF_q(C)=\Quot(A)$ at $\Fp$ and set $\zeta:=\charmorph^*(z)\in R$. We obtain an isomorphism $\BF_\Fp\dbl z\dbr\isoto A_\Fp:=\invlim A/\Fp^n$. As in Example~\ref{ExFiniteShtuka} the $J$-adic completion $M$ of $\ulN$ together with $\tau_M:=\tau_N^f\colon\sigma_{\!q}^{f*}M\to M$ is an effective local shtuka $\ulM=(M,\tau_M)$ over $R$ of rank $r$ (as in Definition~\ref{Def1.1} with $q$ and $\BF_q\dbl z\dbr$ and $\sigma_{\!q}^*$ replaced by $\hat q$ and $\BF_\Fp\dbl z\dbr$ and $\sigma_{\!q}^{f*}$). By \cite[Theorem~6.6]{HartlIsog} the torsion module $\ulE[\Fp^n]$ is a finite locally free strict $\BF_\Fp$-module scheme which satisfies $\Dr_{\hat q}(\ulM/\Fp^n\ulM)=\ulE[\Fp^n]$ and $\ulM/\Fp^n\ulM=\ulM_{\hat q}(\ulE[\Fp^n])$. Moreover, in the sense of Section~\ref{DivlAnmodules} below, the inductive limit $\ulE[\Fp^\infty]:=\dirlim\ulE[\Fp^n]$ is a $\Fp$-divisible local Anderson module over $R$ which satisfies $\Dr_{\hat q}(\ulM)= \ulE[\Fp^\infty]$ and $\ulM=\ulM_{\hat q}(\ulE[\Fp^\infty])$ under the functors from Theorem~\ref{ThmEqZDivGps}. Note that condition (b) of Definition~\ref{DefAndModule} implies that $(z-\zeta)^d=0$ on $\omega_{\ulE[\Fp^n]}$ for every $n$ and on $\omega_{\ulE[\Fp^\infty]}:=\invlim\omega_{\ulE[\Fp^n]}$.
\end{Example}

\section{Divisible local Anderson modules}\label{DivlAnmodules}

The name ``divisible local Anderson module'' is motivated by Example~\ref{ExTorsionOfAModule}(b). These are the function field analogs of $p$-divisible groups. They were introduced in \cite{HartlDict}, but their definition in \cite[\S\,3.1]{HartlDict} and the claimed equivalence in \cite[\S\,3.2]{HartlDict} is false. We give the correct definition below analogously to Messing~\cite[Chapter~I, Definition~2.1]{Messing}. We fix the following notation. For an \fppf-sheaf of $\BF_q[z]$-modules $G$ over a scheme $S$ we denote the kernel of $z^n\colon G\to G$ by $G[z^n]$. Clearly $(G[z^{n+m}])[z^n]=G[z^n]$ for all $n,m\in\BN$.

\begin{Definition} \label{DefZDivGp}
A \emph{$z$-divisible local Anderson module} over a scheme $S\in\Nilp_{\BF_q\dbl\zeta\dbr}$ is a sheaf of $\BF_q\dbl z\dbr$-modules $G$ on the big \fppf-site of $S$ such that
\begin{enumerate}
\item \label{DefZDivGpAxiom1}
$G$ is \emph{$z$-torsion}, that is $G = \dirlim[n] G[z^n]$,
\item \label{DefZDivGpAxiom2}
$G$ is \emph{$z$-divisible}, that is $z\colon G \to G$ is an epimorphism,
\item \label{DefZDivGpAxiom3}
For every $n$ the $\BF_q$-module $G[z^n]$ is representable by a finite locally free strict $\BF_q$-module scheme over $S$ (Definition~\ref{DefStrictF_qMod}), and
\item \label{DefZDivGpAxiom4}
locally on $S$ there exists an integer $d \in \BZ_{\geq 0}$, such that
$(z-\zeta)^d=0$ on $\omega_ G$  where $\omega_G := \invlim[n]\omega_{G[z^n]}$.
\end{enumerate}
We define the \emph{co-Lie module} of a $z$-divisible local Anderson module $G$ over $S$ as $\omega_G:=\invlim\omega_{G[z^n]}$. We will see later in Lemma~\ref{LemLieG} and Theorem~\ref{LoNilForLie} that $\omega_G$ is a finite locally free $\CO_S$-module and we define the \emph{dimension of $G$} as $\rk\omega_G$\,. It is locally constant on $S$.

A $z$-divisible local Anderson module is called \emph{(ind-)\'etale} if $\omega_G = 0$.  Since $\omega_G$ surjects onto each $\omega_{G[z^n]}$ because $\omega_{i_n}\colon\omega_{G[z^{n+1}]}\onto\omega_{G[z^n]}$ is an epimorphism, $\omega_G = 0$ if and only if all $G[z^n]$ are \'etale, see Lemma~\ref{LemmaGEtale}.

A \emph{morphism of $z$-divisible local Anderson modules over $S$} is a morphism of \fppf-sheaves of $\BF_q\dbl z\dbr$-modules.
\end{Definition}

The category of $z$-divisible Anderson modules over $S$ is $\BF_q\dbl z\dbr$-linear and an exact category in the sense of Quillen~\cite[\S2]{Quillen}. 

\begin{Remark}\label{RemQCDirLim}
We will frequently use that for a quasi-compact $S$-scheme $X$ any $S$-morphism $f\colon X\to \dirlim G[z^n]$ factors through $f\colon X\to G[z^m]$ for some $m$; see for example \cite[Lemma~5.4]{HV1}.
\end{Remark}

\begin{Remark}[\bfseries on axiom~\ref{DefZDivGpAxiom4} in Definition~\ref{DefZDivGp}]
Note the following difference to the theory of $p$-divisible groups. On a commutative group scheme multiplication by $p$ always induces multiplication with the scalar $p$ on its co-Lie module. In the case of $\BF_q\dbl z\dbr$-module schemes, axiom~\ref{DefZDivGpAxiom4} is the appropriate substitute for this fact, taking into account Example~\ref{ExTorsionOfAModule}. It allows that $z-\zeta$ is nilpotent on $\omega_{G[z^n]}$. Without axiom~\ref{DefZDivGpAxiom4} the $\CO_S$-module $\omega_G$ is not necessarily finite; see Example~\ref{ExNotZDivGp} below.
\end{Remark}

\begin{Notation}\label{NotShortExSeqX_n}
Let $G$ be a $z$-divisible local Anderson module. We denote by $i_n$ the inclusion map $G[z^n] \hookrightarrow G[z^{n+1}]$ and by $i_{n,m}\colon G[z^n]\to G[z^{m+n}]$ the composite of the inclusions $i_{n+m-1}\circ\ldots\circ i_n$.  We  denote by $j_{n,m}$ the unique homomorphism  $G[z^{m+n}] \to G[z^m]$ which is induced by multiplication with $z^n$ on $G[z^{m+n}]$ such that $i_{m,n}\circ j_{n,m}=z^n\id_{G[z^{m+n}]}$. Observe that also $j_{n,m}\circ i_{m,n}=z^n\id_{G[z^m]}$ for all $m,n\in\BN$, as can be seen by composing with the $\BF_q[z]$-equivariant monomorphism $i_{m,n}\colon G[z^m]\into G[z^{m+n}]$.
\end{Notation}

The following two propositions give an alternative characterization of divisible local Anderson modules, which is analogous to Tate's definition \cite{Tate66} of $p$-divisible groups.

\begin{Proposition} \label{PropTate}
Let $G$ be a $z$-divisible local Anderson module. 
\begin{enumerate}
\item \label{PropTate_A}
For any $0 \leq m,n$ the following sequence of group schemes over $S$ is exact
\begin{equation} \label{EqPropTate}
0\to G[z^n]\xrightarrow{i_{n,m}} G[z^{m+n}]\xrightarrow{j_{n,m}} G[z^m]\to 0.
\end{equation}
\item \label{PropTate_B}
There is a locally constant function $h\colon S\to \BN_0,s\mapsto h(s)$ such that the order of $G[z^n]$ equals $q^{nh}$. We call $h$ the \emph{height} of the $z$-divisible local Anderson module $G$.
\end{enumerate}
\end{Proposition}

\begin{proof}
\ref{PropTate_A} Since $z\colon G\to G$ is an epimorphism, also $j_{n,m}$ is. The rest of \ref{PropTate_A} is clear. Let $h:=\rk_{\CO_S}\ulM_q(G[z])$. Then $\ord G[z]=q^h$ by Theorem~\ref{ThmEqAModSch}\ref{ThmEqAModSchItem5}. Now \ref{PropTate_B} follows from \ref{PropTate_A} and the multiplicativity of the order; see Remark~\ref{RemFactsOnG}(c).
\end{proof}

\begin{Proposition}\label{PropEqvaxiom}
Let $(G_n,\,i_n\colon G_n\into G_{n+1})_{n\in\BN}$ be an inductive system of $\BF_q[z]$-module schemes which are finite locally free strict $\BF_q$-module schemes over $S$ such that
\begin{enumerate}
\item \label{PropEqvaxiom_A}
$i_n$ induces an isomorphism $i_n\colon G_n \isoto G_{n+1}[z^n]$,
\item \label{PropEqvaxiom_B}
there is a locally constant function $h\colon S\to\BN_0$ such that $\ord G_n=q^{nh}$ for all $n$,
\item \label{PropEqvaxiom_C}
locally on $S$ there exist an integer $d \in \BZ_{\geq 0}$, such that
$(z-\zeta)^d=0$ on $\omega_ G$   where $\omega_G = \invlim\omega_{G_n}$.
\end{enumerate}
Then $G = \es \dirlim G_n$ is a $z$-divisible local Anderson module.
\end{Proposition}

\begin{proof}
Since $i_n\colon G_n\into G_{n+1}$ is a monomorphism the maps $G_n(T)\into G_{n+1}(T)$ are injective for all $S$-schemes $T$ and we may identify $G_n(T)$ with a subset of $G(T)$. From \ref{PropEqvaxiom_A} it follows that $G_n=G_m[z^n]\subset G[z^n]$ for all $m\ge n$. Conversely let $x\in G[z^n](T)$ for an $S$-scheme $T$. On each quasi-compact open subscheme $U\subset T$ we can find an $m$ such that $x|_U\in G_m(U)$ by Remark~\ref{RemQCDirLim}. Now $z^n x = 0$ implies $x|_U\in G_m[z^n](U)=G_n(U)$. In total $x\in G_n(T)$. This shows that $G_n=G[z^n]$ and $G=\dirlim G[z^n]$ is $z$-torsion. 

The quotient $G_n/G_1$ is a finite locally free group scheme over $S$ by Remark~\ref{RemFactsOnG}(d). Its order is $q^{(n-1)h}$ by \ref{PropEqvaxiom_B} and the multiplicativity of the order; see Remark~\ref{RemFactsOnG}(c). The natural map $z\colon G_n/G_1\into G_n[z^{n-1}]\cong G_{n-1}$ is a monomorphism and hence a closed immersion by Remark~\ref{RemFactsOnG}(a). Thus $\CO_{G_{n-1}}\onto \CO_{G_n/G_1}$ is an epimorphism of finite locally free $\CO_S$-modules. It must be an isomorphism because $\rk_{\CO_S}\CO_{G_n/G_1}=\ord(G_n/G_1)=\ord(G_{n-1})=\rk_{\CO_S}\CO_{G_{n-1}}$ by \ref{PropEqvaxiom_B}. This proves that $z\colon G_n\to G_{n-1}$ is an epimorphism of \fppf-sheaves. Let $x\in G(T)$ for an $S$-scheme $T$. Choose a quasi-compact open covering $\{U_i\}_i$ of $T$. For each $i$ we find by Remark~\ref{RemQCDirLim} an integer $n_i$ such that $x|_{U_i}\in G_{n_i}(U_i)$. By the above, there is a $y_i\in G_{n_i+1}(U_i)\subset G(U_i)$ with $z\cdot y_i=x|_{U_i}$. This shows that $G$ is $z$-divisible. By \ref{PropEqvaxiom_C} it is a $z$-divisible local Anderson module.
\end{proof}

Note that we require the conditions \ref{DefZDivGpAxiom4} in Definition~\ref{DefZDivGp} and \ref{PropEqvaxiom_C} in  Proposition~\ref{PropEqvaxiom} due to the following example which we do \emph{not} want to consider a $z$-divisible local Anderson module.

\begin{Example} \label{ExNotZDivGp}
Let $S$ be the spectrum of a ring $R$ in which $\zeta$ is zero, and let $G_n$ be the subgroup of $\BG_{a,S}^n=\Spec R[x_1,\ldots,x_n]$ defined by the ideal $(x_1^q,\ldots,x_n^q)$. Make $G_n$ into an $\BF_q\dbl z\dbr$-module scheme by letting $z$ act through
\[
z^\ast(x_1) \es = \es 0\qquad\text{and}\qquad z^\ast(x_\nu) \es=\es x_{\nu-1}\quad\text{for }1<\nu\le n\,.
\]
Define $i_n\colon G_n\to G_{n+1}$ as the inclusion of the closed subgroup scheme defined by the ideal $(x_{n+1})$. 

As in Proposition~\ref{PropEqvaxiom} one proves that $G:=\dirlim G_n$ satisfies axioms~\ref{DefZDivGpAxiom1} to \ref{DefZDivGpAxiom3} of Definition~\ref{DefZDivGp}, but not \ref{DefZDivGpAxiom4}. Here $\omega_{G_n}=\bigoplus_{i=1}^n R\cdot dx_i\cong R^n$, and so $\omega_G$ is not a finite $R$-module. Therefore we cannot drop the conditions \ref{DefZDivGpAxiom4} in Definition~\ref{DefZDivGp} and \ref{PropEqvaxiom_C} in  Proposition~\ref{PropEqvaxiom}. 
\end{Example}

\medskip

In the remainder of this section we introduce truncated $z$-divisible local Anderson modules.

\begin{Lemma}\label{LemmaTruncatedZDAM}
\label{LemmaequivalentconditionforImageandKernel}
Let $n\in\BN$ and let $G$ be an \fppf-sheaf of $\BF_q[z]$-modules over $S$, such that $G=G[z^n]$. Then the following conditions are equivalent
\begin{enumerate}
\item \label{LemmaequivalentconditionforImageandKernel_A}
$G$ is a flat $\BF_q[z]/(z^n)$-module,
\item \label{LemmaequivalentconditionforImageandKernel_B}
$\ker(z^{n-i}) = \im (z^i)$ for $i = 0,\ldots ,n$, that is the morphism $z^i\colon G\to G[z^{n-i}]$ is an epimorphism.
\end{enumerate}
\end{Lemma}

\begin{proof} \ref{LemmaequivalentconditionforImageandKernel_A}$\Longrightarrow$\ref{LemmaequivalentconditionforImageandKernel_B}. Because of \ref{LemmaequivalentconditionforImageandKernel_A}, the multiplication with $z^i$ induces isomorphisms 
\[
\BF_q[z]/(z)\isoto z^i\BF_q[z]/z^{i+1}\BF_q[z] \qquad\text{and}\qquad G/{zG}\isoto z^iG/{z^{i+1}G}
\]
for $i\leq n-1$. This gives us $\ker (z^{n-1}) \subset \im (z)$, and the opposite inclusion $\ker (z^{n-1}) \supset \im (z)$ follows from $G=G[z^n]$. Now $\ker (z^{n-i}) \subset \ker (z^{n-1}) \subset \im (z)$ implies that $\ker (z^{n-i}) = z \ker (z^{n-i+1}) = z\cdot z^{i-1}G = z^iG$ by induction on $i$.

\medskip\noindent
\ref{LemmaequivalentconditionforImageandKernel_B}$\Longrightarrow$\ref{LemmaequivalentconditionforImageandKernel_A}. Taking $i=1$ implies $\im (z) = \ker(z^{n-1})$, and hence multiplication with $z^{n-1}$ induces an isomorphism $G/{zG} \isoto z^{n-1}G$. Since this factors through the epimorphisms $G/{zG} \to zG/{z^2G} \to \cdots \to z^{n-1}G$ we see that each of these maps is an isomorphism. Thus we have 
\begin{align} \label{eqofgr1}
\gr^\bullet\bigl(\BF_q[z]/(z^n)\bigr)\otimes _{\BF_q} \gr^0(G) \xrightarrow{\sim} \gr^\bullet(G). 
\end{align}
Note that the ideal $(z)\subset\BF_q[z]/(z^n)$ is nilpotent. Since $G/zG$ is flat over $\BF_q[z]/(z)=\BF_q$ \cite[Chapter~III, \S\,5.2, Theorem~1]{BourbakiAlgCom} implies that  $G$ is a flat  $\BF_q[z]/(z^n)$-module.
\end{proof}

\begin{Definition} \label{DefTruncatedZDAM}
Let $d,n \in \BN_{>0}$. A \emph{truncated $z$-divisible local Anderson module with order of nilpotence $d$ and level $n$} is an \fppf-sheaf of $\BF_q[z]$-modules over $S$, such that:
\begin{enumerate}
\item If $n\geq 2d$ it is an $\BF_q[z]/(z^n)$-module scheme $G$ which is finite locally free and strict as $\BF_q$-module scheme, such that $(z-\zeta)^d$ is homotopic to $0$ on $\CoL{G/S}$ and $G$ satisfies the equivalent conditions of Lemma~\ref{LemmaequivalentconditionforImageandKernel}. 
\item If $n< 2d$ it is of the form $\ker(z^n \colon G \to G)$ for some truncated $z$-divisible local Anderson module $G$ with order of nilpotence $d$ and level $2d$. 
\end{enumerate}
\end{Definition}

If $G$ is a $z$-divisible local Anderson module over $S\in\Nilp_{\BF_q\dbl\zeta\dbr}$ with $(z-\zeta)^d=0$ on $\omega_G$, we will see in Proposition~\ref{PropTruncZDAM} below that $G[z^n]$ is a truncated $z$-divisible local Anderson module with order of nilpotence $d$ and level $n$. This justifies the name.

\section{The local equivalence} \label{SectLocalEquiv}

The category of $z$-divisible local Anderson modules over $S$ and the category of local shtukas over $S$ are both $\BF_q\dbl z\dbr$-linear.
Our next aim is to extend Drinfeld's construction and the equivalence from Section~\ref{Relshtgroupscheme} to an equivalence between the category of effective local shtukas over $S$ and the category of $z$-divisible local Anderson modules over $S$.

For every effective local shtuka $\ulM=(M,F_M)$ over $S$ we observe $\ulM=\invlim(M/z^n M, F_M \mod z^n M)$ and we set
\[
\Dr_q(\ulM)\es:=\es \dirlim[n] \,\Dr_q\bigl(M/z^n M, F_M \mod z^n M\bigr)\,.
\]
The action of $\BF_q\dbl z\dbr$ on $M$ makes $\Dr_q(\ulM)$ into an \fppf-sheaf of $\BF_q\dbl z\dbr$-modules on $S$. Conversely, for every $z$-divisible local Anderson module $G = \dirlim G[z^n]$ over $S$ we set
\[
\ulM_q(G)\es =\es\bigl(M_q(G),F_{M_q(G)}\bigr)\es:=\es \invlim[n] \,\bigl(M_q(G[z^n]),F_{M_q(G[z^n])}\bigr)\,.
\]
Multiplication with $z$ on $G$ gives $M_q(G)$ the structure of an $\CO_S\dbl z\dbr$-module.

\begin{Lemma}\label{Lemlocallyfree}
Let $G = \dirlim G[z^n]$ be a $z$-divisible local Anderson module of height $r$ over $S$, see Proposition~\ref{PropTate}, then $M_q(G)$ is a locally free sheaf of $\CO_S\dbl z\dbr$-modules of rank $r$. 
\end{Lemma}

\begin{proof}
Applying $\ulM_q$ to the exact sequence $0 \to G[z^n] \xrightarrow{i_n} G[z^{n+1}] \xrightarrow {z^n} G[z^{n+1}]$ yields an exact sequence of $\CO_S\dbl z\dbr$-modules
\[
M_q(G[z^{n+1}])\es\xrightarrow{z^n}\es M_q(G[z^{n+1}])\es\xrightarrow{M_q(i_n)}\es M_q(G[z^n])\es\to\es 0\,.
\]
We deduce from~\cite[\S\,III.2.11, Proposition~14 and Corollaire~1]{BourbakiAlgCom} that $M_q(G)$ is a finitely generated $\CO_S\dbl z\dbr$-module and the canonical map $M_q(G)\to M_q(G[z^n])$ identifies $M_q(G[z^n])$ with $M_q(G)/z^n M_q(G)$. 

We claim that multiplication with $z$ on $M_q(G)$ is injective. So let $\invlim(f_n)_n\in M_q(G),\,f_n\in M_q(G[z^n])$ with $z\cdot f_n=0$ in $M_q(G[z^n])$ for all $n$. To prove the claim consider the factorization
\[
z\cdot\id_{M_q(G[z^{n+1}])}\es=\es M_q(j_{1,n})\circ M_q(i_{n,1})\colon\es M_q(G[z^{n+1}])\es\longto\es M_q(G[z^{n+1}])
\]
obtained from Notation~\ref{NotShortExSeqX_n}. Theorem~\ref{ThmEqAModSch}\ref{ThmEqAModSchItem2} implies that $M_q(j_{1,n})$ is injective, and hence $f_n=M_q(i_{n,1})(f_{n+1})$ is zero for all $n$ as desired.

Locally on $\Spec R\subset S$ the $R$-module $M_q(G[z])$ is free. By Theorem~\ref{ThmEqAModSch}\ref{ThmEqAModSchItem5} its rank is $r$. Let $m_1,\ldots,m_r$ be representatives in $M_q(G)$ of an $R$-basis of $M_q(G[z])$ and consider the presentation
\begin{equation}\label{EqLemlocallyfree}
0\;\longto\; \ker\alpha\;\longto\;\bigoplus_{i=1}^r\,R\dbl z\dbr \,m_i\xrightarrow{\;\,\alpha\;} M_q(G)\;\longto\;0\,.
\end{equation}
Note that $\alpha$ is surjective by Nakayama's Lemma \cite[Corollary~4.8]{Eisenbud} because $z$ is contained in the radical of $R\dbl z\dbr$.
The snake lemma applied to multiplication with $z$ on the sequence \eqref{EqLemlocallyfree} yields the exact sequence 
\[
0\;\longto\;\coker\bigl(z\colon\ker\alpha\to\ker\alpha\bigr)\;\longto\; \bigoplus_{i=1}^rR\,m_i\;\isoto\; M_q(G[z])\to0
\]
in which the right map is an isomorphism. This implies that multiplication with $z^n$ is surjective on $\ker\alpha$ for all $n$, and hence $\ker\alpha\,\subset\,\bigcap_nz^n\cdot(\bigoplus_{i=1}^r R\dbl z\dbr \,m_i)\,=\,0$ because $R\dbl z\dbr$ is $z$-adically separated. Therefore $M_q(G)$ is locally on $S$ a free $\CO_S\dbl z\dbr$-module of rank $r$.
\end{proof}

Recall from Theorem~\ref{ThmEqAModSch}\ref{ThmEqAModSchItem6} that the co-Lie complex $\CoL{G[z^n]/S}$ of $G[z^n]$ is canonically isomorphic to the complex of $\CO_S$-modules $0\to\sigma_{\!q}^*M_q(G[z^n])\xrightarrow{\,F_{M_q(G[z^n])}}M_q(G[z^n])\to0$. In particular, $n_{G[z^n]} \cong \ker F_{M_q(G[z^n])}$ and $\omega_{G[z^n]} \cong \coker F_{M_q(G[z^n])}$ for the $\CO_S$-modules from Definition~\ref{DefOmega}.

\begin{Lemma}\label{LemLieG}
Let $S\in \Nilp_{\BF_q\dbl\zeta\dbr}$ and let $G = \dirlim G[z^n]$ be a $z$-divisible local Anderson module over $S$. Then 
\begin{enumerate}
\item \label{LemLieG_A}
locally on $S$ there is an $N\in\BN$ such that the morphism $i_n\colon G[z^n]\into G[z^{n+1}]$ induces an isomorphism $\omega_{G[z^{n+1}]}\isoto\omega_{G[z^n]}$ for all $n\ge N$.
\item \label{LemLieG_B}
The projective system $(n_{G[z^n]})_n$ satisfies the Mittag-Leffler condition.
\item \label{LemLieG_C}
$\ulM_q(G)$ is an effective local shtuka over $S$ and $\coker(F_{M_q(G)})$ is canonically isomorphic to $\omega_G$. In particular, $\omega_G$ is a finite locally free $\CO_S$-module.
\end{enumerate}
\end{Lemma}

\begin{proof}
Working locally on $S$ we may assume that $\zeta^{N'}=0$ in $\CO_S$ and that $(z-\zeta)^d\omega_G=0$ for some integers $N'$ and $d$. Let $N\ge\max\{N',d\}$ be an integer which is a power of $p$. Then $z^N\omega_G=(z^N-\zeta^N)\omega_G=(z-\zeta)^N\omega_G=0$.

\medskip\noindent
\ref{LemLieG_A} The closed immersion $i_n\colon G[z^n]\into G[z^{n+1}]$ induces an epimorphism $\omega_{i_n}\colon\omega_{G[z^{n+1}]}\onto\omega_{G[z^n]}$ and therefore $\omega_G$ surjects onto each $\omega_{G[z^n]}$. This implies that $z^N\omega_{G[z^n]}=0$ for all $n$. Applying Lemma~\ref{LemmaLongExSeq} to the exact sequence \eqref{EqPropTate} for $m=1$, 
\[
0\longto G[z^n]\xrightarrow{\,i_{n,1}} G[z^{n+1}]\xrightarrow{\,j_{n,1}} G[z]\longto 0,
\]
and using $i_{1,n}\circ j_{n,1}=z^n\id_{G[z^{n+1}]}$ in 
\[
\xymatrix @C+1pc @R+1pc{
\omega_{G[z]} \ar[r]^{j_{n,1}^*\quad} & \omega_{G[z^{n+1}]} \ar[r]^{i_{n,1}^*} & \omega_{G[z^n]} \ar[r] & 0 \\
\omega_{G[z^{n+1}]} \ar@{->>}[u]^{i_{1,n}^*} \ar[ur]_{z^n}
}
\]
we obtain that $\ker(\omega_{G[z^{n+1}]}\onto\omega_{G[z^n]})=z^n\omega_{G[z^{n+1}]}$. Therefore $\omega_{G[z^{n+1}]}\isoto\omega_{G[z^n]}$ is an isomorphism for all $n\ge N$.

\medskip\noindent
To prove \ref{LemLieG_B} we fix an $n\ge N$. We abbreviate the $\CO_S$-modules $M_q(G[z^k])$ by $M_k$ and the map $F_{M_q(G[z^k])}$ by $F_k$. From Proposition~\ref{PropTate} and Theorem~\ref{ThmEqAModSch}\ref{ThmEqAModSchItem2} we have an exact sequence
\[
 0 \to M_k \xrightarrow{\,M_q(j_{n,k})} M_{n+k} \xrightarrow{\,M_q(i_{n,k})} M_n \to 0\,.
\]
It remains exact after applying $\sigma_{\!q}^*$ because $M_n$ is locally free.
For all $k$, we consider the commutative diagrams 
\[
\xymatrix @C-0.5pc{ 0 \ar[r]  & \ker F_{n+k} \ar[r]\ar[d]^{\rho_k} &\sigma_{\!q}^\ast M_{n+k} \ar[r]^{F'_{n+k}}\ar@{->>}[d]^{\sigma_{\!q}^*M_q(i_{n,k})}  & \im F_{n+k}  \ar[r]\ar@{->>}[d] & 0 & 0 \ar[r]  & \im F_{n+k} \ar[r]^{F''_{n+k}}\ar@{->>}[d] & M_{n+k} \ar[r]\ar@{->>}[d]^{M_q(i_{n,k})}  & \coker F_{n+k}  \ar[r]\ar[d]^\cong & 0 \\ 
0 \ar[r]  & \ker F_n \ar[r] & \sigma_{\!q}^\ast M_n \ar[r]_{F'_{n}}  & \im F_n  \ar[r] & 0 & 0 \ar[r]  & \im F_n \ar[r]_{F''_{n}} & M_n \ar[r]  & \coker F_n  \ar[r] & 0
}
\]
where we have split $F_n = F_n''\circ F'_n$ with $F_n'$ surjective and $F_n''$ injective, and where the vertical map on the right in the second diagram is an isomorphism by the identification $\coker F_n=\omega_{G[z^n]}$ from Theorem~\ref{ThmEqAModSch}\ref{ThmEqAModSchItem6} and by what we proved in \ref{LemLieG_A} above. We denote the vertical map on the left in the first diagram by $\rho_k$. The snake lemma applied to both diagrams yields the following exact sequence 
\[
 \sigma_{\!q}^\ast M_k \xrightarrow{\;F_k\,}  M_k \longto \coker \rho_k \longto 0.
\]
Therefore $\coker\rho_k \cong \coker F_k = \omega_{G[z^k]}$. In the diagram
\[
\xymatrix @R-0.5pc {
 0 \ar[r] & \im \rho_{k+1} \ar[r]\ar@{^{ (}->}[d] & \ker F_n \ar[r]\ar@{=}[d] & \coker \rho_{k+1} \ar[r]\ar[d]^\cong & 0\\
0 \ar[r] & \im \rho_{k} \ar[r] & \ker F_n \ar[r] & \coker \rho_{k} \ar[r] & 0
}
\]
the vertical map on the right is an isomorphism for $k\ge N$ by what we have proved in \ref{LemLieG_A} above. Therefore the image of $\rho_k$ stabilizes for $k\ge N$, that is $n_{G[z^n]} =\ker F_n$ satisfies the Mittag-Leffler condition. Note that also $(\im F_n)_n$ satisfies the Mittag-Leffler condition. We will use this for proving \ref{LemLieG_C}.

\medskip\noindent
\ref{LemLieG_C} We still abbreviate $M_q(G[z^n])$ by $M_n$ and $F_{M_q(G[z^n])}$ by $F_n$. The maps $F_n \colon \sigma_{\!q}^\ast M_n \to M_n$ give us two short exact sequences of projective systems  
\[
 0 \,\longto\, \ker F_n \,\longto\, \sigma_{\!q}^\ast M_n\,\longto\, \im F_n \,\longto\, 0  
\qquad\text{and}\qquad
0 \,\longto\, \im F_n \,\longto\, M_n \,\longto\, \coker F_n \,\longto\, 0\,.
\]
Taking the projective limit, using the Mittag-Leffler conditions via \cite[Proposition~II.9.1(b)]{Hartshorne}, the isomorphism $\sigma_{\!q}^\ast (M_q(G))\cong\invlim\sigma_{\!q}^\ast (M_n)$ which is due to the flatness of $M_q(G)$ over $\CO_S$, and combining both exact sequences we obtain an exact sequence 
\[
 0 \,\longto\,\invlim[n] \ker F_n \,\longto\,  \sigma_{\!q}^\ast (M_q(G))\,\xrightarrow{\,F_{M_q(G)}}\, M_q(G) \,\longto\, \invlim[n]\coker F_n\,\longto\, 0 \,.
\]
This shows that $\omega_G:=\invlim\omega_{G[z^n]}=\invlim \coker F_n=\coker F_{M_q(G)}$, which is finite locally free over $\CO_S$ by Lemma~\ref{ExistenceOfe}. Furthermore, condition \ref{DefZDivGpAxiom4} of Definition~\ref{DefZDivGp} implies that $(z -\zeta)^d$ annihilates $\coker F_{M_q(G)}$. This proves that the map $F_{M_q(G)}\colon\sigma_{\!q}^\ast(M_q(G))[\tfrac{1}{z-\zeta}]  \to M_q(G)[\tfrac{1}{z-\zeta}]$ is surjective. As both modules are locally free over $\CO_S \dbl z\dbr[\frac{1}{z-\zeta}]$ of the same rank, the map is an isomorphism. Thus $\ulM_q(G)$ is an effective local shtuka.
\end{proof}

We can now prove the following theorem. It generalizes \cite[\S\,3.4]{Anderson93}, who treated the case of formal $\BF_q\dbl z\dbr$-modules, which we state in \ref{ThmEqZDivGpsItem3}.

\begin{Theorem} \label{ThmEqZDivGps} Let $S\in \Nilp_{\BF_q\dbl\zeta\dbr}$.
\begin{enumerate}
\item  \label{ThmEqZDivGpsItem1}
The two contravariant functors $\Dr_q$ and $\ulM_q$ are mutually quasi-inverse anti-equivalences between the category of effective local shtukas over $S$ and the category of $z$-divisible local Anderson modules over $S$. 
\item  \label{ThmEqZDivGpsItem2}
Both functors are $\BF_q\dbl z\dbr$-linear, map short exact sequences to short exact sequences, and preserve (ind-) \'etale objects.
\end{enumerate}
Let $\ulM=(M,F_M)$ be an effective local shtuka over $S$ and let $G=\Dr_q(\ulM)$ be its associated $z$-divisible local Anderson module. Then
\begin{enumerate}
\setcounter{enumi}{2}
\item  \label{ThmEqZDivGpsItem3}
$G$ is a formal $\BF_q\dbl z\dbr$-module (Definition~\ref{DefFormalModule}) if and only if $F_M$ is topologically nilpotent.
\item  \label{ThmEqZDivGpsItem4}
the height (see Proposition~\ref{PropTate}) and dimension of $G$ are equal to the rank and dimension of $\ulM$.
\item  \label{ThmEqZDivGpsItem5}
the $\CO_S\dbl z\dbr$-modules $\omega_{\Dr_q(\ulM)}$ and $\coker F_M$ are canonically isomorphic.
\item  \label{ThmEqZDivGpsItem6}
if $\ulM$ is bounded by $(d,0,\ldots,0)$ for an integer $d\ge0$, then $\dim G=d$ is constant and $(z-\zeta)^d\cdot\omega_G=(0)$ globally on $S$, but the converse is false in general.
\end{enumerate}
\end{Theorem}

\begin{proof}
\ref{ThmEqZDivGpsItem1} We already saw in Lemma~\ref{LemLieG}\ref{LemLieG_C} that $\ulM_q$ sends $z$-divisible local Anderson modules to effective local shtukas. To prove the converse we use Proposition~\ref{PropEqvaxiom}. Let $\ulM=(M, F_M)$ be an effective local shtuka over $S$ and abbreviate $\ulM/z^n\ulM=:\ulM_n=(M_n,F_{M_n})$ and $G_n:=\Dr_q\bigl(\ulM_n)$. Then $G:=\Dr_q(\ulM)=\dirlim G_n$. Consider the locally constant function $h:=\rk_{\CO_S\dbl z\dbr}M$ on $S$. It satisfies $\rk_{\CO_S}(M_n)=nh$. By Theorem~\ref{ThmEqAModSch} the $G_n$ are finite locally free strict $\BF_q$-module schemes over $S$ of order $q^{nh}$, and the exact sequence of finite $\BF_q$-shtukas $\ulM_{n+1} \xrightarrow{\,z^n} \ulM_{n+1} \longto \ulM_n \longto 0$ yields an exact sequence of group schemes $0 \longto G_n \longto G_{n+1} \xrightarrow{\,z^n} G_{n+1}$. This implies that $G_n = \ker (z^n \colon G_{n+1} \to G_{n+1})=:G_{n+1}[z^n]$. By Lemma~\ref{ExistenceOfe} we know that locally on $S$ there exist positive integers $e', N$ such that $(z-\zeta)^{e'} = 0$ on $\coker F_M$ and $z^{N} = 0$ on  $\coker F_M$. Applying the snake lemma to the diagram
\[
\xymatrix @R-0.5pc @C+1pc { 0 \ar[r] &\sigma_{\!q}^\ast M \ar[r]^{F_M}\ar[d]^{ z^{n}} & M \ar[r]\ar[d]^{z^{n}}  &\coker F_M\ar[r]\ar[d]^{z^{n}}  &0 \\
0 \ar[r] &\sigma_{\!q}^\ast M \ar[r]^{F_M} & M \ar[r] &\coker F_M\ar[r] &0
}
\]
shows that $\coker F_M\to\coker F_{M_n}$ is an isomorphism for $n\ge N$. Therefore by Theorem~\ref{ThmEqAModSch}\ref{ThmEqAModSchItem6} 
\[
\omega_G\;:=\; \invlim \omega_{G_n} \;=\; \invlim \coker(F_{M_n}) \;=\; \coker F_M\,.
\]
This establishes \ref{ThmEqZDivGpsItem5} and implies $(z-\zeta)^{e'} = 0$ on $\omega_G$. Therefore $G = \dirlim G_n$ is a $z$-divisible local Anderson module by Proposition~\ref{PropEqvaxiom}. By Theorem~\ref{ThmEqAModSch} the functors $\Dr_q$ and $\ulM_q$ are quasi-inverse to each other. This proves \ref{ThmEqZDivGpsItem1}.

\medskip\noindent
\ref{ThmEqZDivGpsItem4} From our proof above, the height of $\Dr_q(\ulM)$ equals the rank of $\ulM$. The equality of dimensions follows from \ref{ThmEqZDivGpsItem5}.

\medskip\noindent
\ref{ThmEqZDivGpsItem2} The $\BF_q\dbl z\dbr$-linearity of the functors is clear by construction. From \ref{ThmEqZDivGpsItem5} it follows that both functors $\Dr_q$ and $\ulM_q$ preserve (ind-)\'etale objects. To prove the exactness of $\Dr_q$ let $0 \to \ulM'' \to \ulM \to \ulM' \to 0$ be a short exact sequence of effective local shtukas. Modulo $z^n$ it yields a short exact sequence of finite $\BF_q$-shtukas $0 \to \ulM''_n \to \ulM_n \to \ulM'_n \to 0$, where $\ulM''_n:=\ulM''/z^n\ulM''$, etc. Theorem~\ref{ThmEqAModSch} produces the exact sequence $0 \to G'[z^n] \to G[z^n] \to G''[z^n]\to 0$, where $G = \Dr_q(\ulM)$, $G' = \Dr_q(\ulM')$, $G'' = \Dr_q(\ulM'')$. This implies that $0 \to G' \to G \to G'' \to 0$ is exact, because taking direct limits in the category of sheaves is an exact functor. 

Conversely let $0 \to G' \to G \to G''\to 0$ be a short exact sequence of $z$-divisible local Anderson modules. Since multiplication with $z^n$ is an epimorphism on $G''$, the snake lemma yields the exact sequence of finite locally free strict $\BF_q$-module schemes $0 \to G'[z^n] \to G[z^n] \to G''[z^n]\to 0$. Theorem~\ref{ThmEqAModSch} implies that the sequence $0 \to \ulM''_n \to \ulM_n \to \ulM'_n \to 0$ is exact, where $\ulM = \ulM_q(G), \ \ulM' = \ulM_q(G'),\ \ulM'' = \ulM_q(G'')$. Since $\{\ulM''_n\}$ satisfies the Mittag-Leffler condition we obtain the exactness of $0 \to \ulM'' \to \ulM \to \ulM' \to 0$.

\medskip\noindent
\ref{ThmEqZDivGpsItem3} Let $G =  \Dr_q(\ulM)$. In Proposition~\ref{PropEquivAndersonLiegroups} below we will see that $G$ is a formal $\BF_q\dbl z\dbr$-module if and only if $G[z]=\ker(z\colon G\to G)$ is radicial, which by Theorem~\ref{ThmEqAModSch}\ref{ThmEqAModSchItem3} is equivalent to $F_{M_1}=(F_M\mod z)$ being nilpotent locally on $S$. The latter is the case if and only if locally on $S$ there is an integer $n$ such that $(F_M)^n\equiv 0\mod z$, that is, if and only if $F_M$ is topologically nilpotent. 

\medskip\noindent
\ref{ThmEqZDivGpsItem6} If $\ulM$ is bounded by $(d,0,\ldots,0)$ then $(z-\zeta)^d$ annihilates $\coker F_M=\omega_G$. The dimension of $G$ can be computed by pullback to the closed points $s\colon\Spec k\to S$. There $M\otimes_{\CO_S}k\cong k\dbl z\dbr^{\rk\ulM}\cong\sigma_{\!q}^*M\otimes_{\CO_S}k$ and $\zeta=0$ in $k$. The elementary divisor theorem implies $(\dim G)(s):=\dim_ks^*\omega_G=\dim_k(\coker F_M)\otimes_{\CO_S}k=\ord_z\det (F_M\otimes_{\CO_S}k)=d$ by definition of boundedness of $\ulM$ by $(d,0,\ldots,0)$. That the converse fails can be seen from the following 
\end{proof}

\begin{Example} \label{ExampleNotBounded}
Let $R=k[\epsilon]/(\epsilon^2)$ for a field $k$. Then the local shtuka $\ulM=(R\dbl z\dbr^2,F_M= \left(\begin{smallmatrix} z-\zeta & 0 \\ 0 & z-\zeta-\epsilon \end{smallmatrix}\right))$ satisfies $(z-\zeta)^2=0$ on $\coker F_M=R^2$. So its associated $z$-divisible local Anderson module $G=\Dr_q(\ulM)$ satisfies the conclusion of Theorem~\ref{ThmEqZDivGps}\ref{ThmEqZDivGpsItem6}. But $(\wedge^2 F_M)(1)=(z-\zeta)^2-\epsilon(z-\zeta)\notin(z-\zeta)^2\wedge^2M$, and hence $\ulM$ is not bounded by $(2,0)$. This also shows that \cite[Example~4.5]{HV1} is false.
\end{Example}

\begin{Corollary}\label{CorQuotientIsZDAM}
Let $S\in\Nilp_{\BF_q\dbl\zeta\dbr}$ and let $f\colon G' \to G$ be a monomorphism of $z$-divisible local Anderson modules over $S$. Then the quotient sheaf $G/G'$ is a $z$-divisible local Anderson module over $S$.
\end{Corollary}

\begin{proof}
Since the question is local on $S$ we may assume that $S=\Spec R$ is affine. For all $n$ the induced map $G'[z^n]\to G[z^n]$ is a monomorphism, hence a closed immersion by Remark~\ref{RemFactsOnG}(a). By Lemma~\ref{LemmaStrictF_qAction} it is strict $\BF_q$-linear in the sense of Faltings \cite[Definition~1]{Faltings02}, and by \cite[Proposition~2]{Faltings02} the cokernel $G''_n:=G[z^n]/G'[z^n]$ is a strict $\BF_q$-module scheme which is finite locally free by Remark~\ref{RemFactsOnG}(d). By Theorem~\ref{ThmEqAModSch} this induces the exact sequence of finite $\BF_q$-shtukas $0\to \ulM_q(G''_n)\to\ulM_q(G[z^n])\to\ulM_q(G'[z^n])\to0$. In the following diagram 
\[
\xymatrix @R-1pc @C+1pc {
& 0 \ar[d] & 0 \ar[d] & 0 \ar[d] \\
0 \ar[r] & G'[z^n] \ar[r]^{i'_{n,m}\quad} \ar[d] & G'[z^{n+m}] \ar[r]^{\es j'_{n,m}} \ar[d] & G'[z^m] \ar[r] \ar[d] & 0 \\
0 \ar[r] & G[z^n] \ar[r]^{i_{n,m}\quad} \ar[d] & G[z^{n+m}] \ar[r]^{\es j_{n,m}} \ar[d] & G[z^m] \ar[r] \ar[d] & 0 \\
0 \ar[r] & G''_n \ar[r]^{i''_{n,m}\quad} \ar[d] & G''_{n+m} \ar[r]^{\es j''_{n,m}} \ar[d] & G''_m \ar[r]  \ar[d] & 0 \\
& 0 & 0 & 0 
}
\]
the columns are exact by definition of $G''_n$, and the two upper rows are exact by Proposition~\ref{PropTate}. By the snake lemma this defines the exact sequence in the bottom row. By Theorem~\ref{ThmEqAModSch} this implies that $\ulM_q(i''_{n,1})\colon\ulM_q(G''_{n+1})\to\ulM_q(G''_n)$ is surjective for all $n$. In particular, the projective system $\ulM_q(G''_n)$ satisfies the Mittag-Leffler condition, and the morphism $\ulM_q(f)\colon \ulM:=\ulM_q(G)\to\ulM':=\ulM_q(G')$ of effective local shtukas corresponding to $f$ by Theorem~\ref{ThmEqZDivGps} is surjective by \cite[Proposition~II.9.1(b)]{Hartshorne}. The kernel $\ulM'':=\ker\ulM_q(f)=\invlim\ulM_q(G''_n)$ is a locally free $R\dbl z\dbr$-module with a morphism $F_{M''}\colon\sigma_{\!q}^*M''\to M''$ inducing an isomorphism $F_{M''}\colon\sigma_{\!q}^*M''[\tfrac{1}{z-\zeta}]\to M''[\tfrac{1}{z-\zeta}]$, because this is true for $\ulM$ and $\ulM'$. Thus $\ulM''$ is an effective local shtuka over $S$. Applying the snake lemma to the (injective) multiplication with $z^n$ on the sequence $0\to\ulM''\to\ulM\to\ulM'\to0$ shows that $\ulM''/z^n\ulM''=\ulM_q(G''_n)$. Therefore, Theorem~\ref{ThmEqZDivGps} implies that $G/G'=\Dr_q(\ulM'')=\dirlim G''_n$ is a $z$-divisible local Anderson module over $S$.
\end{proof}

\section{Frobenius, Verschiebung and deformations of local shtukas}\label{SectVersch}

\begin{Definition}\label{DefKernelOfFrob}
Let $G$ be an \fppf-sheaf of groups over an $\BF_q$-scheme $S$. For $n\in\BN_{0}$ we let $G[F_q^n]$ be the kernel of the relative $q^n$-Frobenius $F_{q^n,G}\colon G\to\sigma_{\!q^n}^*G$ of $G$ over $S$. In particular, $G[F_q^0]:=\ker(\id_G)=(0)$.
\end{Definition}

Let $S\in\Nilp_{\BF_q\dbl\zeta\dbr}$. Later we will assume that $\zeta=0$ in $\CO_S$. Let $G$ be a $z$-divisible local Anderson module over $S$ and let $\ulM=(M,F_M)=\ulM_q(G)$ be its associated local shtuka from Theorem~\ref{ThmEqZDivGps}. Then the $q$-Frobenius morphism $F_{q,G}:=\dirlim F_{q,G[z^n]}\colon G\to\sigma_{\!q}^*G$ corresponds by diagram~\eqref{EqDiagFrobenius} to the morphism 
\[
\ulM_q(F_{q,G})\,=\,F_M\colon\; M_q(\sigma_{\!q}^*G)\;=\;\sigma_{\!q}^*M_q(G)\;\longto\; M_q(G),\quad m\;\mapsto\; m\circ F_{q,G}\;=\;F_M(m)\,.
\]
In addition to the $q$-Frobenius, $G$ carries a $q$-Verschiebung which is identically zero by Theorem~\ref{Theofequivalence1}. Therefore, if $\zeta=0$ in $\CO_S$ we will introduce a ``$z^d$-Verschiebung'' in Remark~\ref{RemzdVersch} and Corollary~\ref{CorzdVersch} below, which is more useful for $z$-divisible local Anderson modules. We begin with the following

\begin{Lemma}\label{LemmaVerschLocSht}
Let $M$ be an effective local shtuka with $(z-\zeta)^d = 0$ on $\coker F_M$. Then there exists a uniquely determined homomorphism of $\CO_S\dbl z\dbr$-modules $V_M\colon M \to \sigma_{\!q}^\ast M$ with $F_M \circ V_M \, = \, (z-\zeta)^d\cdot\id_{M}$ and $V_M \circ F_M \,=\,(z-\zeta)^d\cdot\id_{\sigma_{\!q}^\ast M}$.
\end{Lemma}

\begin{proof}
Since $F_M$ is injective by Lemma~\ref{ExistenceOfe} and $(z-\zeta)^d=0$ on $\coker F_M$, the lemma follows from the following diagram
\begin{equation}\label{EqDiagV}
\xymatrix @C+1pc  { 0 \ar[r]  & \sigma_{\!q}^\ast M \ar[r]^{F_M}\ar[d]_{(z-\zeta)^d} & M \ar[r] \ar[d]^{(z-\zeta)^d}\ar@{-->}[dl]_{V_M} & \coker F_M \ar[r] \ar[d]^{(z-\zeta)^d\,=\,0} & 0\; \\
 0 \ar[r] & \sigma_{\!q}^\ast M \ar[r]_{F_M} & M \ar[r] & \coker F_M \ar[r] & 0\,.
}
\vspace{-5ex}
\end{equation}
\end{proof}

\begin{Remark}\label{RemzdVersch}
If $\zeta=0$ in $\CO_S$, the Frobenius $f:=F_M\colon\sigma_{\!q}^*\ulM[\tfrac{1}{z}]\isoto\ulM[\tfrac{1}{z}]$ satisfies $F_M\circ\sigma_{\!q}^*f=F_M\circ\sigma_{\!q}^*F_M=f\circ\sigma_{\!q}^*F_M$, and hence is a quasi-isogeny between the local shtukas $\sigma_{\!q}^*\ulM=(\sigma_{\!q}^*M,\sigma_{\!q}^*F_M)$ and $\ulM$. Likewise, if $\ulM$ is effective with $(z-\zeta)^d = 0$ on $\coker F_M$, the homomorphism $V_M$ from Lemma~\ref{LemmaVerschLocSht} is an isogeny $V_{z^d,\ulM}:=V_M\colon\ulM[\tfrac{1}{z}]\isoto\sigma_{\!q}^*\ulM[\tfrac{1}{z}]$, called the \emph{$z^d$-Verschiebung of $\ulM$}. It satisfies $F_M \circ V_{z^d,\ulM} \, = \, z^d\cdot\id_{\ulM}$ and $V_{z^d,\ulM} \circ F_M \,=\,z^d\cdot\id_{\sigma_{\!q}^\ast\ulM}$. Indeed, $\zeta=0=\zeta^q$ implies that the following diagram is commutative
\[
\xymatrix @C+3pc {
\sigma_{\!q}^\ast {M} \ar[r]^{\sigma_{\!q}^\ast V_{z^d,\ulM}}\ar[d]^{F_{M}} & \sigma_{\!q}^*(\sigma_{\!q}^*M)\ar[d]^{\sigma_{\!q}^\ast F_{M}\,=\,F_{\sigma_{\!q}^*M}}\\
{M} \ar[r]^{V_{z^d,\ulM}} & \sigma_{\!q}^\ast {M}\;,
}
\]
as $F_{\sigma_{\!q}^*M} \circ \sigma_{\!q}^\ast V_{z^d,\ulM}=\sigma_{\!q}^\ast F_{M} \circ \sigma_{\!q}^\ast V_{z^d,\ulM}=\sigma_{\!q}^*\bigl((z-\zeta)^d\cdot\id_{M}\bigr)=(z-\zeta^q)^d\cdot\id_{\sigma_{\!q}^*M} = V_{z^d,\ulM} \circ F_{M}$. 
\end{Remark}

\begin{Corollary}\label{CorzdVersch}
Assume that $\zeta=0$ in $\CO_S$. Let $G$ be a $z$-divisible local Anderson module over $S$ with $(z-\zeta)^d = 0$ on $\omega_G$. Then there is a uniquely determined morphism $V_{z^d,G} \colon \sigma_{\!q}^*G \to G$ with $F_{q,G} \circ V_{z^d,G} \,=\, z^d\cdot\id_{\sigma_{\!q}^*G}$ and $V_{z^d,G} \circ F_{q,G} \,=\, z^d\cdot\id_{G}$. It is called the \emph{$z^d$-Verschiebung of $G$}. In particular, $G[F_q^n]:=\ker(F_{q^n,G}\colon G \to \sigma_{\!q^n}^*G)$ is contained in $G[z^{nd}]$ and $\ker(V_{z^d,G}^n\colon\sigma_{\!q^n}^*G\to G)\subset\sigma_{\!q^n}^*G[z^{nd}]$ for all $n$.
\end{Corollary}

\begin{proof}
Let $\ulM=\ulM_q(G)$ be the effective local shtuka associated with $G$. Since $(z-\zeta)^d = 0$ on $\omega_G=\coker F_{M}$, the $z^d$-Verschiebung $V_{z^d,\ulM}$ of $\ulM$ from Remark~\ref{RemzdVersch} corresponds by Theorem~\ref{ThmEqZDivGps} to a morphism $V_{z^d,G}:=\Dr_q(V_{z^d,\ulM})\colon\sigma_{\!q}^*G\to G$ with $F_{q,G} \circ V_{z^d,G} \,=\, z^d\cdot\id_{\sigma_{\!q}^*G}$ and $V_{z^d,G} \circ F_{q,G} \,=\, z^d\cdot\id_{G}$, and hence $V_{z^d,G}^n \circ F_{q,G}^n \,=\, z^{nd}\cdot\id_{G}$ and $F_{q,G}^n\circ V_{z^d,G}^n \,=\, z^{nd}\cdot\id_{\sigma_{\!q^n}^*G}$. This proves the corollary.
\end{proof}

\begin{Proposition}\label{PropTruncZDAM}
Let $G$ be a $z$-divisible local Anderson module with $(z-\zeta)^d = 0$ on $\omega_G$, and let $n\in\BN$. Then $G[z^n]:=\ker(z^n \colon G \to G)$ is a truncated $z$-divisible local Anderson module with order of nilpotence $d$ and level $n$; see Definition~\ref{DefTruncatedZDAM}.
\end{Proposition}

\begin{proof}
The equivalent conditions of Lemma~\ref{LemmaequivalentconditionforImageandKernel} for the $\BF_q[z]/(z^n)$-module scheme $G[z^n]$ follow from Proposition~\ref{PropTate} by considering for all $\nu=0,\ldots,n$ the commutative diagram
\[
\xymatrix @C+2pc {
0 \ar[r] & G[z^{n-\nu}] \ar[r]^{\TS i_{n-\nu,\nu}} & G[z^n] \ar[r]^{\TS j_{n-\nu,\nu}} \ar[rd]_{\TS z^{n-\nu}} & G[z^\nu] \ar@{^{ (}->}[d]^{\TS i_{\nu,n-\nu}} \ar[r] & 0\\
& G[z^n] \ar@{->>}[u]^{\TS j_{\nu,n-\nu}} \ar[ru]_{\TS z^\nu} & & G[z^n]
}
\]
in which $i_{\nu,n-\nu}$ is a monomorphism and $j_{\nu,n-\nu}$ an epimorphism, and hence $\ker(z^{n-\nu})=\ker(j_{n-\nu,\nu})=\im(i_{n-\nu,\nu})=\im(z^\nu)$.
 By Theorem~\ref{ThmEqZDivGps}\ref{ThmEqZDivGpsItem5}, $(z-\zeta)^d = 0$ on $\coker F_{M_q(G)}$. We reduce the map $V_{M_q(G)}$ from Lemma~\ref{LemmaVerschLocSht} modulo $z^n$ to obtain a homomorphism $V_{M_q(G[z^n])}\colon M_q(G[z^n])\to\sigma_{\!q}^*M_q(G[z^n])$ with $F_{M_q(G[z^n])} \circ V_{M_q(G[z^n])} \,=\, (z-\zeta)^d\cdot\id_{M_q(G[z^n])}$ and $V_{M_q(G[z^n])}\circ F_{M_q(G[z^n])} \,=\,(z-\zeta)^d\cdot\id_{\sigma_{\!q}^\ast {M_q(G[z^n])}}$. Under the identification of the co-Lie complex $\CoL{G[z^n]/S}$ with $0\to \sigma_{\!q}^*M_q(G[z^n])\xrightarrow{\,F_{M_q(G[z^n])}} M_q(G[z^n])\to 0$ from Theorem~\ref{ThmEqAModSch}\ref{ThmEqAModSchItem6} the map $V_{M_q(G[z^n])}$ corresponds to a homotopy $h\colon t^*_{G[z^n]}\to N_{G[z^n]}$ with $dh=(z-\zeta)^d$ on $t^*_{G[z^n]}$ and $hd=(z-\zeta)^d$ on $N_{G[z^n]}$. This means that $(z-\zeta)^d$ is homotopic to zero on $\CoL{G[z^n]/S}$.
\end{proof}

\begin{Proposition}\label{PropKerTruncZDAM}
Assume that $\zeta=0$ in $\CO_S$. Let $G = G[z^l]$ be a truncated $z$-divisible local Anderson module over $S$ with order of nilpotence $d$ and level $l$. Then
\begin{enumerate}
\item \label{PropKerTruncZDAM_D}
there exists a morphism $V_{z^d,G}\colon \sigma_{\!q}^*G \to G$ with $F_{q,G} \circ V_{z^d,G} \,=\, z^d\cdot\id_{\sigma_{\!q}^*G}$ and $V_{z^d,G} \circ F_{q,G} \,=\, z^d\cdot\id_{G}$. It is not uniquely determined, unless $G$ is \'etale.
\item \label{PropKerTruncZDAM_E}
$G[F_q^i] \subset G[z^{id}]$ and \ $\ker V_{z^d,G}^i \subset \sigma_{\!q^i}^*G[z^{id}]$ for all $i$.
\end{enumerate}
Now let $n\in\BN_{>0}$ and $l=nd$. In particular, if $n = 1$ there is a truncated divisible local Anderson module $\wt G$ of level $2d$ with $G=\wt G[z^d]$ and we assume that $V_{z^d,G}=V_{z^d,\wt G}\big|_{\sigma_{\!q}^*G}$. Then 
\begin{enumerate}
\setcounter{enumi}{2}
\item \label{PropKerTruncZDAM_A}
for all $i$ with $0\leq i \leq n$ the morphism $F_{q,G}^i\colon G[F_q^n] \to \sigma_{\!q^i}^*G[F_q^{n-i}]$ is an epimorphism,
\item \label{PropKerTruncZDAM_B}
the morphisms $V_{z^d,G}^n\colon\sigma_{\!q^n}^*G\to\ker F_{q,G}^n$ and $F_{q,G}^n\colon G\to\ker V_{z^d,G}^n$ are epimorphisms,
\item \label{PropKerTruncZDAM_C}
$G[F_q^i]$ and $\ker V_{z^d,G}^i$ are finite locally free strict $\BF_q$-module schemes over $S$ for all $0\le i\le n$,
\item \label{PropKerTruncZDAM_F}
for all $0\le i\le n$ we have $\omega_{G} = \omega_{G[z^{id}]} = \omega_{G[F_q]}$ and this is a finite locally free $\CO_S$-module.
\end{enumerate}
\end{Proposition}

\begin{proof}  
\ref{PropKerTruncZDAM_D} Let $h\colon t^*_{G}\to N_{G}$ be a homotopy with $dh=(z-\zeta)^d$ on $t^*_{G}$ and $hd=(z-\zeta)^d$ on $N_{G}$. Note that $h$ is determined only up to adding a homomorphism $t^*_G\onto\coker d=\omega_G\to n_G=\ker d\into N_G$, and in particular, is not unique unless $G$ is \'etale.  Let $V\colon M_q(G)\to\sigma_{\!q}^*M_q(G)$ be the homomorphism which corresponds to $h$ under the identification of the co-Lie complex $\CoL{G/S}$ with $0\to \sigma_{\!q}^*M_q(G)\xrightarrow{\,F_{M_q(G)}} M_q(G)\to 0$ from Theorem~\ref{ThmEqAModSch}\ref{ThmEqAModSchItem6}. Then $V\circ F_{M_q(G)}=z^d\cdot\id_{\sigma_{\!q}^*M_q(G)}=\sigma_{\!q}^*(z^d\cdot\id_{M_q(G)})=\sigma_{\!q}^*(F_{M_q(G)}\circ V)=F_{\sigma_{\!q}^*M_q(G)}\circ \sigma_{\!q}^*V$ implies that $V\colon\ulM_q(G)\to\sigma_{\!q}^*\ulM_q(G)=\ulM_q(\sigma_{\!q}^*G)$ is a morphism of finite $\BF_q$-shtukas. It induces the desired morphism $V_{z^d,G}:=\Dr_q(V)\colon \sigma_{\!q}^*G \to G$ with $F_{q,G} \circ V_{z^d,G} \,=\, z^d\cdot\id_{\sigma_{\!q}^*G}$ and $V_{z^d,G} \circ F_{q,G} \,=\, z^d\cdot\id_{G}$. 

\medskip\noindent 
\ref{PropKerTruncZDAM_E} follows from $V_{z^d,G}^i \circ F_{q,G}^i \,=\, z^{id}\cdot\id_{G}$ and $F_{q,G}^i\circ V_{z^d,G}^i \,=\, z^{id}\cdot\id_{\sigma_{\!q^i}^*G}$ which are consequences of \ref{PropKerTruncZDAM_D}.

\medskip\noindent 
\ref{PropKerTruncZDAM_A} is trivial if $n = 1$ and $i=0$ or $1$. If $n \geq 2$ there is by \ref{PropKerTruncZDAM_D} a factorization $F_{q,G}^i \circ V_{z^d,G}^i = z^{id}\colon \sigma_{\!q^i}^*G \to \sigma_{\!q^i}^*G$. Since the morphism $z^{id}\colon \sigma_{\!q^i}^*G\to \sigma_{\!q^i}^*G[z^{(n-i)d}]$ is an epimorphism by Lemma~\ref{LemmaTruncatedZDAM}, and since $\sigma_{\!q^i}^*G[F_q^{n-i}] \subset \sigma_{\!q^i}^*G[z^{(n-i)d}]$ by \ref{PropKerTruncZDAM_E}, we obtain \ref{PropKerTruncZDAM_A}.

\medskip\noindent 
\ref{PropKerTruncZDAM_B} is proved by induction on $n$. For $n=1$ we use $G=\wt G[z^d]$. By Lemma~\ref{LemmaTruncatedZDAM} there is an exact sequence
\[
 0 \longrightarrow \wt G[z^{d}] \longrightarrow \wt G \xrightarrow{\;z^{d}\,} \wt G[z^{d}] \longrightarrow 0\,.
\]
Since $ G[F_q]\subset\wt G[F_q]\subset\wt G[z^{d}]$ by \ref{PropKerTruncZDAM_E}, the map $V_{z^d,\wt G} \circ F_{q,\wt G}=z^d\colon({z^{d}})^{-1}(G[F_q]) \to G[F_q]$ is an epimorphism. From $F_{q,\wt G}\colon ({z^{d}})^{-1}(G[F_q]) \to \sigma_{\!q}^*\wt G[z^{d}]=\sigma_{\!q}^*G$ we see that $V_{z^d,G}=V_{z^d,\wt G}\big|_{\sigma_{\!q}^*G} \colon \sigma_{\!q}^*G \to G[F_q]$ is an epimorphism. The statement for $F_{q,G}$ is proved in the analogous way using $\ker V_{z^d,G}\subset\ker V_{z^d,\wt G}\subset\sigma_q^*\wt G[z^d]$. Thus we have proved \ref{PropKerTruncZDAM_B} for the case $n = 1$. 

To prove it in general by induction on $n$, consider the diagram
\[
 \xymatrix @C+2pc { \sigma_{\!q^n}^*G[z^{nd}] \ar[rr]^{V_{z^d,G}^n} \ar@{->>}[d]_{z^d} & & G[F_q^n] \ar@{->>}[d]^{F_{q,G}}\\
\sigma_{\!q^n}^*G[z^{(n-1)d}] \ar@{->>}[r]_{\sigma_{\!q}^*V_{z^d,G}^{n-1}\qquad} & (\sigma_{\!q}^*G[z^{(n-1)d}])[F_q^{n-1}] \ar@{=}[r] & \sigma_{\!q}^*G[F_q^{n-1}]\,.
}
\]
In the bottom row, $\sigma_{\!q}^*V_{z^d,G}^{n-1}$ is an epimorphism by the induction hypothesis, and the equality comes from \ref{PropKerTruncZDAM_E}. The vertical map on the left is an epimorphism by Lemma~\ref{LemmaTruncatedZDAM}, and therefore $F_{q,G} \circ V_{z^d,G}^n$ is an epimorphism. Thus if we can show that $\ker(F_{q,G})=G[F_q]$ is contained in the image of $V_{z^d,G}^n$, it will follow that $V_{z^d,G}^n$ is an epimorphism. But by the case $n = 1$ settled above 
\[
 G[F_q] = V_{z^d,G}(\sigma_{\!q}^*G[z^d]) = V_{z^d,G} \circ z^{(n-1)d}(\sigma_{\!q}^*G[z^{nd}]) = V_{z^d,G}^n\circ F_{q,G}^{n-1}(\sigma_{\!q}^*G[z^{nd}]) \subset V_{z^d,G}^n(\sigma_{\!q^n}^*G[z^{nd}]).
\]
This proves that $V_{z^d,G}^n$ is an epimorphism. The statement for $F_{q,G}^n$ is proved in the same way. 

\medskip\noindent
\ref{PropKerTruncZDAM_C} The morphisms $F_{q,G}^i\colon G\to\sigma_{\!q^i}^*G$ and $V_{z^d,G}^i\colon\sigma_{\!q^i}^*G\to G$ between group schemes of finite presentation over $S$ are themselves of finite presentation by \cite[IV$_1$, Proposition~1.6.2(v)]{EGA}. Therefore $G[F_q^i]:=\ker F_{q,G}^i$ and $\ker V_{z^d,G}^i$ are of finite presentation over $S$ by \cite[IV$_1$, Proposition~1.6.2(iii)]{EGA}. As closed subschemes of $G$, respectively $\sigma_{\!q^i}^*G$, they are also finite over $S$. Since in \ref{PropKerTruncZDAM_B} we proved that $V_{z^d,G}^i\colon \sigma_{\!q^i}^*G[z^{id}]\onto (G[z^{id}])[F_q^i]=G[F_q^i]$ and $F_{q,G}^i\colon G[z^{id}]\onto \ker V_{z^d,G[z^{id}]}^i=\ker V_{z^d,G}^i$ are epimorphisms, they are faithfully flat by Remark~\ref{RemFactsOnG}(b). Therefore $G[F_q^i]$ and $\ker V_{z^d,G}^i$ are flat over $S$ by \cite[IV$_3$, Corollaire~11.3.11]{EGA}, and hence finite locally free. Over any affine open $U\subset S$ the $\BF_q$-equivariant morphisms $F_{q,G}^i$ and $V_{z^d,G}^i$ lift by Lemma~\ref{LemmaStrictF_qAction} to morphisms in $\DGr(\BF_q)_U$. Thus they are $\BF_q$-strict morphisms in the sense of Faltings \cite[Definition~1]{Faltings02}. By \cite[Proposition~2]{Faltings02} their kernels $G[F_q^i]\times_SU$ and $\ker (V_{z^d,G}^i)\times_SU$ are strict $\BF_q$-module schemes over $U$. So the $\BF_q$-strictness of $G[F_q^i]$ and $\ker V_{z^d,G}^i$ over all of $S$ follows from Lemma~\ref{LemmaStrictIsLocal}.

\medskip\noindent
\ref{PropKerTruncZDAM_F} For any group scheme $G=\Spec R[X_1,\ldots,X_r]/I$ of finite type over $\Spec R$, we compute $G[F_q]=\Spec R[X_1,\ldots,X_r]/(I,X_1^q,\ldots,X_r^q)$. By the conormal sequence \cite[Proposition~II.8.12]{Hartshorne} for the closed immersion $G[F_q]\subset G$ this implies $\omega_G=\omega_{G[F_q]}$. The inclusion $G[F_q] \subset G[z^d]$ from \ref{PropKerTruncZDAM_E} therefore implies $G[F_q]= (G[z^{id}])[F_q]$, and hence $\omega_{G} = \omega_{G[z^{id}]} = \omega_{G[F_q]}$ for all $i$. Moreover, since $G[F_q]$ is a finite locally free strict $\BF_q$-module scheme over $S$ by \ref{PropKerTruncZDAM_C}, we can compute $\omega_{G[F_q]}$ as $\coker F_{M_q(G[F_q])}$ where $(M_q(G[F_q]),F_{M_q(G[F_q])})$ is the associated finite $\BF_q$-shtuka from Theorem~\ref{ThmEqAModSch}. In particular, $F_{M_q(G[F_q])}=\ulM_q(F_{q,G[F_q]})=0$ and this implies that $\coker F_{M_q(G[F_q])}=M_q(G[F_q])$ is a finite locally free $\CO_S$-module. 
\end{proof}

\bigskip

In the remainder of this section we will show that to lift a $z$-divisible local Anderson-module is equivalent to lifting its ``Hodge filtration''. Let $S\in\Nilp_{\BF_q\dbl\zeta\dbr}$ and let $G$ be a $z$-divisible local Anderson-module over $S$ satisfying $(z-\zeta)^d\cdot\omega_G=0$. Let $(M,F_M)$ be its effective local shtuka. Then $(z-\zeta)^d\cdot\coker F_M=0$ and we consider the map $V_M$ from Lemma~\ref{LemmaVerschLocSht}. The injective morphism $F_M$ induces by diagram~\eqref{EqDiagV} an exact sequence of $\CO_S\dbl z\dbr$-modules
\[
0\;\longto\; \coker V_M\xrightarrow{F_M} M/(z-\zeta)^d M \;\longto\; \coker F_M\;\longto\; 0\,.
\]
In particular $\coker V_M$ is a locally free $\CO_S$-module of finite rank. Conversely, $V_M$ induces the exact sequence of $\CO_S\dbl z\dbr$-modules
\begin{equation}\label{EqHodgeFiltr}
0\;\longto\; \coker F_M\xrightarrow{V_M} \sigma_{\!q}^\ast M/(z-\zeta)^d \sigma_{\!q}^\ast M \;\longto\; \coker V_M\;\longto\; 0\,.
\end{equation}

\begin{Definition}[{compare \cite[\S\,5.7]{HartlJuschka}}]
We call $H(G):=\Koh^1_{\rm dR}\bigl(G,\CO_S[z]/(z-\zeta)^d\bigr):=\sigma_{\!q}^\ast M/(z-\zeta)^d \sigma_{\!q}^\ast M$ the \emph{de Rham cohomology of $G$ with coefficients in $\CO_S[z]/(z-\zeta)^d$}. It is a locally free $\CO_S[z]/(z-\zeta)^d$-module of rank equal to $\rk\ulM=\height G$.
The $\CO_S[z]$-submodule $V_M(\coker F_M)\subset H(G)$ is called the \emph{Hodge filtration} of the $z$-divisible local Anderson-module $G$.
\end{Definition}

Now let $i\colon S'\into S$ be a closed subscheme defined by an ideal $I$ with $I^q=0$. Then the morphisms $\Frob_{q,S}$ and $\Frob_{q,S'}$ factor through $i$
\[
\Frob_{q,S}\es = \es i\circ j\colon  \es S\to S'\to S\qquad\text{and}\qquad\Frob_{q,S'}\es = \es j\circ i\colon  \es S'\to S\to S'\,.
\]
where $j\colon S\to S'$ is the identity on the underlying topological space $|S'|=|S|$ and on the structure sheaf this factorization is given by
\begin{eqnarray*}
\TS\CO_S \es \xrightarrow{\es i^\ast\;} & \CO_{S'} & \TS\xrightarrow{\es j^\ast\;} \es \CO_S\\
\TS b\quad \mapsto\;\: & b \mod I& \TS\;\:\mapsto\quad b^q\,.
\end{eqnarray*}
Let $G'$ be a divisible local Anderson-module over $S'$ with $(z-\zeta)^d\cdot\omega_{G'}=0$, and denote by $(M',F_{M'})$ its local shtuka. We set $H(G')_S :=j^\ast M'/(z-\zeta)^d \,j^\ast M'$. This is a locally free module over $\CO_S[z]/(z-\zeta)^d$ and satisfies $i^\ast H(G')_S=H(G')$.

\begin{Theorem}\label{ThmDefo}
The functor $G\,\mapsto \,\bigr(\,i^*G,\, V_M(\coker F_M)\subset H(G)\bigl)$ defines an equivalence between 
\begin{enumerate}
\itemsep0ex
\topsep0ex
\item 
the category of $z$-divisible local Anderson-modules $G$ over $S$ with $(z-\zeta)^d\cdot\omega_G=0$, and
\item 
the category of pairs $\bigr(\,G',\, Fil\subset H(G')_S\,\bigl)$ where $G'$ is a $z$-divisible local Anderson-module over $S'$ and $Fil\subset H(G')_S$ is an $\CO_S\dbl z\dbr$-submodule whose quotient is a flat $\CO_S$-module, and which specializes to the $\CO_S'\dbl z\dbr$-submodule $V_{M'}(\coker F_{M'})\subset H(G')$ under $i$.
\end{enumerate}
\end{Theorem}

\begin{proof}
We describe the quasi-inverse functor. Let $\bigr(\,G',\, Fil\subset H(G')_S\,\bigl)$ be given and let $(M',F_{M'})$ be the local shtuka of $G'$. We define $V_M\colon M\into j^\ast M'$ as the kernel of the morphism $j^\ast M'\onto H(G')_S/Fil$. Since $Fil\subset H(G')_S$ specializes to $V_M(\coker F_{M'})\subset H(G')$ we obtain $i^*(H(G')_S/Fil)=H(G')/V_{M'}(\coker F_{M'})=\coker V_{M'}$. This implies $i^*M\cong M'$ and $\sigma_{\!q}^\ast M=j^\ast i^\ast M\cong j^\ast M'$. Moreover $\coker V_M$ is annihilated by $(z-\zeta)^d$. Thus there is an injective morphism of $\CO_S\dbl z\dbr$-modules $F_M\colon \sigma_{\!q}^\ast M\to M$ with $F_MV_M=(z-\zeta)^d\id_M$ and $V_MF_M=(z-\zeta)^d\id_{\sigma_{\!q}^\ast M}$. From sequence~\eqref{EqHodgeFiltr} we see that the cokernel of $F_M$ is a locally free $\CO_S$-module. Clearly the $z$-divisible local Anderson-module $G$ over $S$ associated with the local shtuka $(M,F_M)$ specializes to $G'$ and has $Fil\subset H(G')_S=H(G)$ as its Hodge filtration.
\end{proof}

\begin{Remark}
We only treated the case where $S'\subset S$ is defined by an ideal $I$ with $I^q=0$. The general case for $d=1$ is treated by Genestier and Lafforgue~\cite[Proposition~6.3]{GL} using $\zeta$-divided powers in the style of Grothendieck and Berthelot.
\end{Remark}

\section{Divisible local Anderson modules and formal Lie groups} \label{SectFormalLieGps}

In this section we clarify the relation between $z$-divisible local Anderson modules and formal $\BF_q\dbl z\dbr$-modules; see Definition~\ref{DefFormalModule}. We follow the approach of Messing~\cite{Messing} who treated the analogous situation of $p$-divisible groups and formal Lie groups. 

\begin{Definition}
Let $G$ be an \fppf-sheaf of abelian groups over $S\in\Nilp_{\BF_q\dbl\zeta\dbr}$. We say that $G$ is 
\begin{itemize}
\item 
\emph{$F$-torsion} if $G = \dirlim[n] G[F_q^n]$, see Definition~\ref{DefKernelOfFrob},
\item 
\emph{$F$-divisible} if $F_{q,G}\colon G \to \sigma_{\!q}^*G$ is an epimorphism.
\end{itemize}
\end{Definition}

Recall that Messing \cite[Chapter~II, Theorem~2.1.7]{Messing} proved that a sheaf of groups $G$ on $S$ is a formal Lie group \cite[Chapter~II, Definitions~1.1.4 and 1.1.5]{Messing}, if and only if $G$ is $F$-torsion, $F$-divisible, and the $G[F_q^n]$ are finite locally free $S$-group schemes.

\begin{Theorem}
\label{THpzcase}
When $\zeta = 0$ in $\CO_S$ and $G$ is a $z$-divisible local Anderson module over $S$, then $\dirlim[n] G[F_q^n]$ is a formal $\BF_q\dbl z\dbr$-module. It is equal to $\olG := \es \dirlim[k] \Inf^k(G)$, where for any $S$-scheme $T$, Messing \cite[Chapter~II, (1.1)]{Messing} defines
\begin{eqnarray}\label{EqInf}
(\Inf^kG)(T) & := & \Bigl\{\,x\in G(T)\colon \text{there is an \fppf-covering $\{\Spec R_i\to T\}_i$}\\[0mm]
& & \quad\text{and for every $i$ an ideal $I_i\subset R_i$ with $I_i^{k+1}=(0)$}\nonumber\\[0mm]
& & \quad\text{such that the pull-back $x\in G(\Spec R_i/I_i)$ is zero}\,\Bigr\}\,.\nonumber
\end{eqnarray}
\end{Theorem}

\begin{proof} 
By \cite[Chapter~II, Theorem~2.1.7]{Messing} it suffices to show that $\dirlim G[F_q^n]$ is $F$-torsion, $F$-divisible and that the $G[F_q^n]$ are finite locally free. By construction $\dirlim G[F_q^n]$ is $F$-torsion. By Definition~\ref{DefZDivGp}\ref{DefZDivGpAxiom4} there is locally on $S$ an integer $d$ with $(z-\zeta)^d\cdot\omega_G=(0)$, and then $G[F_q^n]\subset G[z^{nd}]$ by Corollary~\ref{CorzdVersch} and $G[z^{nd}]$ is a truncated $z$-divisible local Anderson module with order of nilpotence $d$ and level $nd$ by Proposition~\ref{PropTruncZDAM}. Therefore Proposition~\ref{PropKerTruncZDAM} shows that $G[F_q^n]$ is finite locally free, and that $F_{q,G}\colon G[F_q^n]\to \sigma_{\!q}^*G[F_q^{n-1}]$ is an epimorphism. Consequently, $F_{q,G} \colon \dirlim G[F_q^n] \to \sigma_{\!q}^*(\dirlim G[F_q^n]) = \dirlim \sigma_{\!q}^*G[F_q^{n-1}]$ is an epimorphism and so $\dirlim G[F_q^n]$ is $F$-divisible, and hence a formal Lie group. The action of $\BF_q\dbl z\dbr$ makes it into a formal $\BF_q\dbl z\dbr$-module.

To prove the last statement of the theorem, observe that for any $S$-scheme $T$ the homomorphism $F_{q^n,G} \colon G(T) \to (\sigma_{\!q^n}^*G)(T)$ is simply the map sending $x$ to $x \circ \Frob_{q^n,T}$ as can be seen from the following diagram
\[
\xymatrix @C+3pc {
T \ar[rr]^{\Frob_{q^n,T}} \ar[d]_{x} \ar[dr]^{F_{q^n,G}(x)} & & T \ar[d]^{x}\\
G \ar[r]^{F_{q^n,G}\qquad} \ar@/_1.3pc/[rr]_{\qquad\qquad\quad\Frob_{q^n,G}}\ar[d] & \sigma_{\!q^n}^*G \ar[r] \ar[d] & G \ar[d] \\
S \ar@{=}[r] & S \ar[r]_{\Frob_{q^n,S}} & S\,.\hspace{-1ex}
}
\]
Therefore, the monomorphism $G[F_q^n]\into G$ defines an inclusion $G[F_q^n]\subset\Inf^{q^n-1}G$, the ideals $I_i$ in \eqref{EqInf} being the augmentation ideal in $\CO_{G[F_q^n]}$ defining the zero section. We claim that this inclusion is an equality. So let $x \in (\Inf^{q^n - 1}G)(T)$ and let $R_i$ and $I_i$ be as in \eqref{EqInf}. Then $I_i^{q^n}=(0)$ implies that $\Frob_{q^n,R_i}$ factors through $R_i\longonto R_i/I_i \xrightarrow{\es j\;} R_i$. So $F_{q^n,G}(x)|_{\Spec R_i}=x\circ\Frob_{q^n,R_i}=j^*(x|_{\Spec R_i/I_i})=0$, that is $x\in G[F_q^n]$. Thus we have $G[F_q^n]=\Inf^{q^n-1}G$ and $\dirlim G[F_q^n]= \dirlim\Inf^k(G) \subset G$ which completes the proof. 
\end{proof}

Our next aim is to extend the theorem to all $S\in\Nilp_{\BF_q\dbl\zeta\dbr}$. For that purpose we start with the following

\begin{Lemma}
\label{CorExtmapzero}
Let $S$ be a scheme with $\zeta^{N+1} = 0$ in $\CO_S$, and let $G = G[z^{nd}]$ be a truncated $z$-divisible local Anderson module over $S$ with order of nilpotence $d$ and level $nd$ with $n \geq N+1$. Then for any affine open subset $U$ of $S$ and any quasi-coherent sheaf $\CF$ of $\CO_U$-modules the natural homomorphism for the co-Lie complexes $\Ext^1_{\CO_U}(\CoL{G[z^{(n-N-1)d}]/U}, \CF) \longrightarrow \Ext^1_{\CO_U}(\CoL{G[z^{nd}]/U}, \CF)$ is zero.
\end{Lemma}

\begin{proof} 
We proceed by induction on $N$ and begin with $N = 0$. If $n = 1$, then $G[z^{(n-N-1)d}] = (0)$ and there is nothing to prove. If $n \geq 2$ we use \cite[Chapter~II, Corollary~3.3.9]{Messing} for the sequence 
\[
 0 \longrightarrow G[z^{(n-1)d}] \longrightarrow G[z^{nd}]  \xrightarrow{\,z^{(n-1)d}} G[z^d] \longrightarrow 0\,.
\]
So we have to show that
\begin{enumerate}
\item \label{MessingCrit_A}
$\omega_{G[z^{nd}]}\onto\omega_{G[z^{(n-1)d}]}$ is an isomorphism,
\item \label{MessingCrit_B}
$\omega_{G[z^{nd}]}$ and $\omega_{G[z^d]}$ are locally free $\CO_S$-modules, and
\item \label{MessingCrit_C}
$\rk\omega_{s^*G[z^{(n-1)d}]}\le\rk\omega_{s^*G[z^d]}$ for all points $s\in S$.
\end{enumerate}
All three statements follow from Proposition~\ref{PropKerTruncZDAM}\ref{PropKerTruncZDAM_F}. This concludes the proof when $N=0$.

For general $N$ we take the exact sequence 
\[
 0 \to \zeta \CF \to \CF \to \CF/{\zeta \CF} \to 0.
\]
and consider the commutative diagram with exact rows
\[
\xymatrix @C-1pc @R-0.5pc {\Ext^1_{\CO_U}(\CoL{G[z^{(n-N-1)d}]/U}, \zeta \CF) \ar[r]\ar[d] &\Ext^1_{\CO_U}(\CoL{G[z^{(n-N-1)d}]/U}, \CF)\ar[r] \ar[d]& \Ext^1_{\CO_U}(\CoL{G[z^{(n-N-1)d}]/U}, \CF/\zeta \CF)\ar[d]   \\
\Ext^1_{\CO_U}(\CoL{G[z^{(n-N)d}]/U}, \zeta \CF)\ar[r]\ar[d] & \Ext^1_{\CO_U}( \CoL{G[z^{(n-N)d}]/U}, \CF) \ar[r]\ar[d] & \Ext^1_{\CO_U}(\CoL{G[z^{(n-N)d}]/U}, \CF/\zeta \CF)  \\
\Ext^1_{\CO_U}(\CoL{G[z^{nd}]/U}, \zeta \CF)\ar[r] & \Ext^1_{\CO_U}( \CoL{G[z^{nd}]/U}, \CF) \,.
} 
\]
Since $\zeta\cdot(\CF/\zeta\CF)=(0)$, the right vertical arrow can be computed by base change to the zero locus $\Var(\zeta)\subset S$ of $\zeta$. So it is the zero map by what we have proved above, and hence the image of $\Ext^1_{\CO_U}(\CoL{G[z^{(n-N-1)d}]/U},  \CF)$ in $\Ext^1_{\CO_U}( \CoL{G[z^{(n-N)d}]/U}, \CF)$ lies inside the image of $\Ext^1_{\CO_U}(\CoL{G[z^{(n-N)d}]/U}, \zeta \CF)$. Since $\zeta^N\cdot(\zeta\CF)=(0)$, the lower left vertical arrow can similarly be computed by base change to the zero locus $\Var(\zeta^N)\subset S$, and hence it is the zero map by our induction hypothesis. This proves the lemma.
\end{proof}

\begin{Theorem}\label{DLAMFsmooth}
If $S\in\Nilp_{\BF_q\dbl\zeta\dbr}$ and $G$ is a $z$-divisible local Anderson module over $S$, then $G$ is formally smooth.
\end{Theorem}

\begin{proof}
Let $X'$ be an affine scheme over $S$ and let $X$ be a closed subscheme defined by an ideal of square zero. Let $f \colon X \to G$ be an $S$-morphism. We must show that $f$ can be lifted to an $S$-morphism $f' \colon X' \to G$.
\[
\xymatrix { 
X \ \ar[d]_f \ar@{^{(}->}[r]  & X'\ar@{-->}[dl]^{f'} \\ 
G
} 
\]
As $X$ is quasi-compact we have $G(X) = \dirlim[n] G[z^n](X) = \dirlim[n] G[z^{nd}](X)$, and hence $f \colon X \to G[z^{nd}]$ for some $n$ by Remark~\ref{RemQCDirLim}. We cover $X$ by  a finite number of affine opens $U_i, \ i = 1,\ldots , m$ such that the image of $U_i$ in $S$ is contained in an affine open $V_i$. Since  $\zeta$ is nilpotent on each $V_i$ there is an integer $N$ such that $\zeta^{N+1}$ is zero on $\bigcup V_i$. Replacing $S$ by $S' = \bigcup V_i$ and $G$ by $G_{S'}$ we are led to the case where $\zeta^{N+1}=0$ in $\CO_S$. But now Lemma~\ref{CorExtmapzero} and \cite[Chapter~II, Proposition~3.3.1]{Messing} show that $f$ can be lifted to an $f'\colon X'\to G[z^{(n+N+1)d}]$ and the theorem is proved. 
\end{proof}

\begin{Lemma}\label{Liftlemma}
Let $G$ be a $z$-divisible local Anderson module over $S$ with  $(z-\zeta)^d = 0$ on $\omega_G$ for some $d \in \BN$. Assume we are given an $S$-scheme $X'$ and a subscheme $X$ defined by a sheaf of ideals $I$ such that $I^{k+1} = (0)$ and $\zeta^{N}\cdot {{I/I^2}} = 0$ for some integer $N$. Let $N'$ be the smallest integer which is a power of $p$ and greater or equal to $N$ and $d$. If an $S$-morphism $f' \colon X' \to G$ satisfies $f = f'\vert_X\colon X \to G[z^n]$, then $f'$ factors through $f'\colon X' \to G[z^{n+kN'}]\subset G$.
\end{Lemma}

\begin{proof} The problem is local on $X'$ and hence we can assume that $X'$ is affine  and thus quasi-compact. But then $f' \in G(X') = \dirlim G[z^m](X')$ and hence we can assume that $f'\colon X' \to G[z^{n'}]$ for some $n'$ by Remark~\ref{RemQCDirLim}. We now use induction on $k$ and the sequence of closed subschemes $\Var(I^l)\subset X'$ for $l=1,\ldots,k+1$. Thus we can assume that $I^2 = 0$ and $k=1$. 

Since $f\in G[z^n](X)$ we have $z^nf = 0$, and so $z^nf' \in G[z^{n'}](X')$ has the property that its restriction to $G[z^{n'}](X)$ is zero. Since $I^2 = 0$, the group of sections of $G[z^{n'}]$ over $X'$ whose restriction to $X$ is zero, is by \cite[III, Th\'eor\`eme~0.1.8(a)]{SGA3} isomorphic to the group $\Hom_{\CO_X}(\omega_{G[z^{n'}]}\otimes_{\CO_S}\CO_X, I)$ under an isomorphism which sends the zero morphism $X'\to G[z^{n'}]$ to the zero element, and the morphism $z^nf'$ to an element which we denote by $h \in  \Hom_{\CO_X}(\omega_G\otimes_{\CO_S}\CO_X, I)$. Since $\zeta^N$ kills $I$ and $N'\ge N$ we obtain $\zeta^{N'}\cdot h=0$. On $\omega_G$ the assumption $(z-\zeta)^d=0$ implies $z^{N'}=\zeta^{N'}$, and so the section $z^{N'}(z^nf')$ is sent to $z^{N'}\cdot h=\zeta^{N'}\cdot h=0$. This implies $z^{n+N'}f'=0$, that is, $f' \in G[z^{n+N'}](X')$.
\end{proof}

\begin{Corollary}
\label{Corforrep}
Let $\zeta^N=0$ in $\CO_S$ and let $G$ and $d$ be as in Lemma~\ref{Liftlemma}. Let $N'$ be the smallest integer which is power of $p$ and greater or equal to $N$ and $d$. Then the $k$-th infinitesimal neighborhood of $G[z^n]$ in $G$ is the same as that of $G[z^n]$ in $G[z^{n+kN'}]$. In particular, $\Inf^k(G) = \Inf^k(G[z^{kN'}])$ and this is therefore representable.
\end{Corollary}

\begin{proof} By definition \cite[Chapter~II, Definition~(1.01)]{Messing}, an $S$-morphism $f\colon T' \to G$ belongs to the $k$-th infinitesimal neighborhood of $G[z^n]$ in $G$, if and only if there is an \fppf-covering $\{\Spec R_i \to T'\}_i$ and ideals $I_i\subset R_i$ with $I_i^{k+1}=(0)$ such that $f|_{\Spec R_i/I_i}\in G[z^n](\Spec R_i/I_i)$. But then $f \in G[z^{n+kN'}](T')$ by Lemma~\ref{Liftlemma}. The last statement is the special case with $n=0$.
\end{proof}

\begin{Theorem}\label{LoNilForLie}
Let $G$ be a $z$-divisible local Anderson module over $S\in\Nilp_{\BF_q\dbl\zeta\dbr}$. Then $\olG = \dirlim\Inf^k(G)$ is a formal $\BF_q\dbl z\dbr$-module.
\end{Theorem}

\begin{proof} 
As $\olG$ clearly is an $\BF_q\dbl z\dbr$-submodule of $G$, we must show that it is a formal Lie variety; see \cite[Chapter~II, Definition~1.1.4]{Messing}. By construction it is ind-infinitesimal. Since the question is local on $S$ we may assume that there are integers $N$ and $d$ as in Corollary~\ref{Corforrep}. Then the sheaf $\Inf^k(G)$ is representable for all $k$. By Theorem~\ref{DLAMFsmooth} we know that $G$ is formally smooth and by definition \eqref{EqInf} of $\Inf^k(G)$ this implies that $\olG$ is formally smooth. Let $N'$ be the smallest integer which is power of $p$ and greater or equal to $N$ and $d$. Then $G[z^{kN'}]$ satisfies the lifting condition 2) of \cite[Chapter~II, Proposition~3.1.1]{Messing} by Theorem~\ref{DLAMFsmooth} and Lemma~\ref{Liftlemma}. Therefore, by \cite[loc.\ cit.]{Messing} $G[z^{kN'}]$ satisfies condition 2) and 3) of \cite[Chapter~II, Definition~1.1.4]{Messing}, and hence is a formal $\BF_q\dbl z\dbr$-module.
\end{proof}

\begin{Remark} 
We already know from Theorem~\ref{ThmEqZDivGps} and Lemma~\ref{ExistenceOfe} that $\omega_G$ is locally free of finite rank. This now follows again from the theorem, because $\omega_G=\omega_{\olG}$.
\end{Remark}

\bigskip

Next we pursue the question when a $z$-divisible local Anderson module is a formal $\BF_q\dbl z\dbr$-module and vice versa.

\begin{Lemma}\label{LemmaPowerOfIdealIsZero}
Let $B$ be a ring in which $\zeta$ is nilpotent, and let $I$ be a nilpotent ideal of $B$. Define a sequence of ideals  $I_1 := \zeta I + I^2,\ldots, I_{n+1} := \zeta I_n + {(I_n)}^2$. Then for $n$ sufficiently large $I_n = (0)$.
\end{Lemma}

\begin{proof}
 Let $J = \zeta B + I$. Then it is easy to check that $I_n \subset J^{n+1}$. Since $\zeta$ and $I$ both are nilpotent, so is the ideal $J$. This implies $I_n = 0$ for $n$ sufficiently large. 
\end{proof}

\begin{Lemma}\label{Lemmaztorsion}
If $S\in\Nilp_{\BF_q\dbl\zeta\dbr}$ and $G$ is a formal $\BF_q\dbl z\dbr$-module over $S$ such that locally on $S$ there is an integer $d$ with $(z-\zeta)^d = 0$ on $\omega_G$, then $G$ is $z$-torsion.
\end{Lemma}

\begin{proof}
We must show $G = \dirlim G[z^n]$ and since this is a statement about sheaves, it suffices to check it locally on $S$. Thus we can assume $S = \Spec R$ with $\zeta\in R$ nilpotent and $G$ is given by a power series ring $R\dbl X_1,\ldots ,X_d\dbr$; see \cite[p.~26]{Messing}. If $T$ is any affine $S$-scheme, say $T = \Spec B$, then an element of $G(T)$ will be an $N$-tuple $(b_1,\ldots ,b_d)$ with each $b_i$ nilpotent. Let $I$ be the ideal generated by $\{ b_1,\ldots ,b_d\}$. Let $N'$ be a power of $p$ with $N'\geq d$. Then multiplication with $z^{N'}$ on $G$ is given by power series $(z^{N'})^*(X_i)\in R\dbl X_1,\ldots ,X_d\dbr$ with linear term $\zeta^{N'}X_i$ and without constant term, because $\omega_G=(X_1,\ldots,X_d)/(X_1\ldots,X_d)^2$. Therefore each component of $z^{N'} \cdot (b_1,\ldots ,b_d)$ belongs to  $\zeta^{N'} I + I^2\subset \zeta I+I^2=:I_1$. Then each component of $z^{nN'} \cdot (b_1,\ldots ,b_d)$ belongs to the ideal $I_n$ from Lemma~\ref{LemmaPowerOfIdealIsZero}, and hence the lemma shows that $(b_1,\ldots,b_d)$ is $z$-torsion.
\end{proof}

The next result is analogous to Messing's characterization \cite[Chapter~II, Proposition~4.4]{Messing} for a $p$-divisible group to be a formal Lie group, and also its proof follows similarly using Theorem~\ref{LoNilForLie}.

\begin{Proposition} \label{PropEquivAndersonLiegroups}
Let $S\in\Nilp_{\BF_q\dbl\zeta\dbr}$ and let $G$ be a $z$-divisible local Anderson module over $S$. Then the following conditions are equivalent:
\begin{enumerate}
\item \label{PropEquivAndersonLiegroups_A}
$G = \olG$.
\item \label{PropEquivAndersonLiegroups_B}
$G$ is a formal $\BF_q\dbl z\dbr$-module.
\item \label{PropEquivAndersonLiegroups_C}
$G[z^n]$ is radicial for all $n$.
\item \label{PropEquivAndersonLiegroups_D}
$G[z]$ is radicial. \qed
\end{enumerate}
\end{Proposition}

\begin{Corollary}\label{CorEqvcat1}
For $S\in\Nilp_{\BF_q\dbl\zeta\dbr}$, there is an equivalence of categories between that of $z$-divisible local Anderson modules over $S$ with $G[z]$ radicial, and the category of $z$-divisible formal $\BF_q\dbl z\dbr$-modules $G$ with $G[z]$ representable by a finite locally free group scheme, such that locally on $S$ there is an integer $d$ for which $(z-\zeta)^d = 0$ on $\omega_G$.
\end{Corollary}

\begin{proof}
By Lemma~\ref{Lemmaztorsion} and Proposition~\ref{PropEquivAndersonLiegroups} both categories are identified with the same full sub-category of \fppf-sheaves of $\BF_q[z]$-modules on $S$, once we observe that $G[z^n]:=\ker(z^n\colon G\to G)$ is a strict $\BF_q$-module as the kernel of an $\BF_q$-linear homomorphism of formal Lie groups which are $\BF_q$-modules.
\end{proof}

\begin{Corollary}
Let $S\in\Nilp_{\BF_q\dbl\zeta\dbr}$ be the spectrum of an Artinian local ring. Then a $z$-divisible formal $\BF_q\dbl z\dbr$-module, such that locally on $S$ there is an integer $d$ for which $(z-\zeta)^d = 0$ on $\omega_G$, is a $z$-divisible local Anderson module with $G[z]$ radicial and conversely.
\end{Corollary}

\begin{proof}
This follows from Corollary~\ref{CorEqvcat1}, because the $G[z^n]$ are automatically representable by finite locally free group schemes by \cite[Chapter~II, Proposition~4.3]{Messing}.
\end{proof}

The next result is analogous to Messing's characterization \cite[Chapter~II, Proposition~4.7]{Messing} for a $p$-divisible group to be ind-\'etale, and also its proof follows verbatim.

\begin{Proposition}
Let $S\in\Nilp_{\BF_q\dbl\zeta\dbr}$ and let $G$ be a $z$-divisible local Anderson module over $S$. In order that $\olG=0$ it is necessary and sufficient that $G$ is (ind-)\'etale. \qed
\end{Proposition}

We have the  following lemma for $z$-divisible local Anderson modules over $S$ similarly to and with the same proof as \cite[Chapter~II, Proposition~4.11]{Messing}.

\begin{Lemma}\label{LemmaolGExact}
Let $S\in\Nilp_{\BF_q\dbl\zeta\dbr}$ and let $0 \to G_1 \to G_2 \to G_3 \to 0$ be an exact sequence of $z$-divisible local Anderson modules over $S$. Then $0 \to \olG_1 \to \olG_2 \to \olG_3 \to 0$ is also exact. \qed
\end{Lemma}

Finally there is a criterion when $\olG$ is itself a $z$-divisible local Anderson module in analogy to Messing's criterion \cite[Chapter~II, Proposition~4.9]{Messing}.

\begin{Proposition}\label{Propextension}
Let $S\in\Nilp_{\BF_q\dbl\zeta\dbr}$ and let $G$ be a $z$-divisible local Anderson module over $S$. Then the following conditions are equivalent.
\begin{enumerate}
\item \label{Propextension_A}
$\olG$ is a $z$-divisible local Anderson module.
\item \label{Propextension_B}
$G$ is an extension of an (ind-)\'etale $z$-divisible local Anderson module $G''$ by an ind-infinitesimal $z$-divisible local Anderson module $G'$. 
\item \label{Propextension_C}
$G$ is an extension  of an (ind-)\'etale $z$-divisible local Anderson module $G''$ by a $z$-divisible formal $\BF_q\dbl z\dbr$-module $G'$.
\item \label{Propextension_D}
For all $n$, $G[z^n]$ is an extension of a finite \'etale group by a finite locally-free radicial group.
\item \label{Propextension_E}
$G[z]$ is an extension of a finite \'etale group by a finite locally-free radicial group.
\item \label{Propextension_F}
the map $S\to\BZ$, $s \mapsto\ord(G[z]_s)_\et=:$ separable rank $(G[z]_s)$ is a locally constant function on $S$.
\end{enumerate}
\end{Proposition}

\begin{proof}
The proof proceeds in the same way as \cite[Chapter~II, Proposition~4.9]{Messing} using Corollary~\ref{CorQuotientIsZDAM} and Lemma~\ref{LemmaolGExact} in \ref{Propextension_A}$\Longleftrightarrow$\ref{Propextension_B}, Corollary~\ref{CorEqvcat1} and $\omega_G=\omega_{G'}$ in \ref{Propextension_B}$\Longleftrightarrow$\ref{Propextension_C}, and Lemma~\ref{LemmaEtaleIsStrict} in \ref{Propextension_D}$\Longrightarrow$\ref{Propextension_C}.
\end{proof}

\begin{Corollary} \label{CorCanonDecompZDiv}
If $S$ is the spectrum of a field $L$ every $z$-divisible local Anderson-module $G=\dirlim G[z^n]$ over $S$ is canonically an extension of an (ind-)\'etale divisible local Anderson-module $G^\et$ by a $z$-divisible formal $\BF_q\dbl z\dbr$-module $\olG$
\[
0\es\longto\es \olG\es\longto\es G\es\longto\es G^\et\es\longto\es 0\,.
\]
$G^\et$ is the largest (ind-)\'etale quotient of $G$. With notation as in Proposition~\ref{PropCanonDecompAMod} we have $\olG=\dirlim G[z^n]^0$ and $G^\et=\dirlim G[z^n]^\et$. If $L$ is perfect the extension splits canonically.

This decomposition is compatible with the decomposition of the local shtuka $\ulM_q(G)$ from Proposition~\ref{PropCanonDecompZCryst} under the functors $\ulM_q$ and $\Dr_q$ from Theorem~\ref{ThmEqZDivGps}.
\end{Corollary}

\begin{proof}
Proposition~\ref{Propextension}, whose condition \ref{Propextension_F} is trivially satisfied, provides the extension and the equalities $\olG=\dirlim G[z^n]^0$ and $G^\et=\dirlim G[z^n]^\et$. From this the characterization of $G^\et$ and the canonical splitting for perfect $L$ follows; see Proposition~\ref{PropCanonDecompAMod}. Finally the compatibility with the decomposition of the local shtuka $\ulM_q(G)$ from Proposition~\ref{PropCanonDecompZCryst} follows from the characterization of $G^\et$ being (ind-)\'etale, respectively $\olG$ being a formal $\BF_q\dbl z\dbr$-module in terms of their associated local shtukas proved in Theorem~\ref{ThmEqZDivGps}.
\end{proof}

\begin{appendix}
\section{Review of the cotangent complex}\label{AppCotCom}

In this appendix we carry out the elementary exercise to compare the definitions of the cotangent complex given by Illusie~\cite[\S\,VII.3.1]{Illusie72}, Lichtenbaum and Schlessinger~\cite[\S\,2.1]{LS}, and Messing~\cite[Chapter~II, \S\,3.2]{Messing} for a finite locally free group scheme $G=\Spec A$ over $S=\Spec R$. Recall that $G$ is a relative complete intersection by \cite[Proposition~III.4.15]{SGA3}. This means that locally on $S$ we can take $A = R[X_1,\ldots , X_n]/I$ where the ideal $I$ is generated by a regular sequence $(f_1,\ldots,f_n)$ of length $n$; compare \cite[IV$_4$, Proposition~19.3.7]{EGA}.

\subsection*{The cotangent complex in the sense of Lichtenbaum and Schlessinger}
\label{CotComLSApplication}

We follow the notation of Lichtenbaum and Schlessinger \cite[\S\,2.1]{LS} and take the free $R[\ulX]$-module $F=R[\ulX]\cdot g_1\oplus\ldots\oplus R[\ulX]\cdot g_n$. We set $U:=\ker(j\colon F\onto I,\,g_\nu\mapsto f_\nu)$ and let $U_0$ be the image of the $R[\ulX]$-linear map $F\otimes_{R[\ulX]} F\to F,\,x\otimes y\mapsto j(x)y-j(y)x$.

\begin{Lemma}\label{Lemmaforfinitepresentation}
There is an exact sequence of $R[\ulX]$-modules 
\begin{equation}\label{finitelypresented}
\xymatrix @R=0pc {
\bigoplus_{1\le \mu<\nu\le n}\; R[\ulX]\cdot h_{\mu\nu} \ar[r]^{\quad i} & \bigoplus_{\nu=1}^n \; R[\ulX]\cdot g_\nu \ar[r]^{\qquad\quad j} & I \ar[r] & 0\,.\\
h_{\mu\nu}\ar@{|->}[r] & f_\nu g_\mu-f_\mu g_\nu\,,\quad g_\nu \ar@{|->}[r] & f_\nu
}
\end{equation}
In particular the ideal $I$ is finitely presented and $U = U_0$.
\end{Lemma}

\begin{proof}
By definition the map $j$ in \eqref{finitelypresented} is surjective. To prove exactness in the middle let $\sum_{\nu=1}^{n}a_\nu g_\nu \in \ker j$, that is $\sum_{\nu=1}^{n}a_\nu f_\nu = 0$ in $R[\ulX]$.  This implies $a_nf_n = 0$ in $R[\ulX]/(f_1,\ldots,f_{n-1})$. Since $f_n \in I$ is a non-zero-divisor in  $R[\ulX]/(f_1,\ldots,f_{n-1})$ we have $a_n = 0$ in $R[\ulX]/(f_1,\ldots,f_{n-1})$. Thus there exist $b_{n\mu}\in R[\ulX]$ for $1 \leq \mu \leq n-1$ such that ${a}_n = \sum_{\mu=1}^{n-1}b_{n\mu}f_\mu$. It follows that 
\[
\sum_{\nu=1}^{n}a_\nu g_\nu \;\equiv\; \sum_{\nu=1}^{n}a_\nu g_\nu - i\Bigl(\sum_{\mu=1}^{n-1}b_{n\mu}h_{n\mu}\Bigr) \;\equiv\; (a_1 + b_{n1}f_n)g_1 + \ldots + (a_{n-1} +b_{n,n-1}f_n) g_{n-1}  \mod \im(i)\,.
\]
Continuing in this way we get
\[
\sum_{\nu=1}^{n}a_\nu g_\nu \;\equiv\; (a_1 +  b_{n1}f_n +  b_{n-1,1}f_{n-1} + \ldots + b_{2,1}f_2)g_1 \;\mod\im(i),
\]
and hence $(a_1 +  b_{n1}f_n +  b_{n-1,1}f_{n-1} + \ldots + b_{2,1}f_2)f_1 = 0$ in $R[\ulX]$. Since $f_1$ is a non-zero-divisor in $R[\ulX]$ we conclude $\sum_{m=1}^{n}a_\nu g_\nu\in \im(i)$. This proves that $I$ is finitely presented over $R[\ulX]$. Moreover, $U=\ker(j)=\im(i)=U_0$.
\end{proof}

Consequently the cotangent complex of Lichtenbaum and Schlessinger \cite[Definition~2.1.3]{LS} is the complex of ${\mathcal{O}_G}$-modules concentrated in degrees $-1$ and $0$ given by
\[
\CoCLS{G/S}\colon \enspace 0 \longrightarrow  I/{I^2}   \longrightarrow  \Omega_ {R[\ulX]/R} \otimes_{R[\ulX]} A \longto 0\,.
\]
By \cite[Corollaire~III.3.2.7]{Illusie71} this complex is quasi-isomorphic to the cotangent complex $\CoCI{G/S}$ defined by Illusie~\cite[II.1.2.3]{Illusie71}, which we considered in Section~\ref{CotCom} before Definition~\ref{DefOmega}.

\subsection*{The cotangent complex in the sense of Messing} 
\label{CotComMessing}
We next recall the definition of the cotangent complex of $G/S$ given by Messing~\cite[Chapter~II, \S\,3.2]{Messing}. Since $G$ is a group scheme, $A$ is a bi-algebra with comultiplication $\Delta \colon A \to A \otimes_R A$ and counit $\epsilon_A\colon A \to R$. Then $\check{A} = \Hom_{R\text{-Mod}}(A,R)$ carries an $R$-algebra structure via the dual morphisms $\check{\Delta}$ of $\Delta$ and $\check{\epsilon}_A$ of $\epsilon_A$.

\begin{Definition}
We let $U(G):=\Spec(\Sym^\bullet_R A)$. It represents the contravariant functor from $S$-schemes to rings, sending an $S$-scheme $T$ to the ring $\Gamma(T,\ \check{A} \otimes_R\mathcal{O}_T)$; see \cite[II, 1.7.9]{EGA}.

We let $U(G)\mal$ be the contravariant functor from $S$-schemes to abelian groups whose points with values in an $S$-scheme $T$ are the invertible elements in the ring $U(G)(T)=\Gamma(T,\ \check{A} \otimes_R \mathcal{O}_T)$.
\end{Definition}

So $U(G)\mal$ is defined by the fiber product diagram 
\[
\xymatrix{    U(G)\mal =  S\times _{U(G)}(U(G)\times U(G)) \ar[d] \ar[r]  & U(G)\times U(G)\ar[d]^{\TS\check{\Delta}} \\
                           S\ar[r]^{\TS\check{\epsilon}_A} &  U(G)}
\]
Since $U(G)$ is affine and of finite presentation over $S$, its unit section $\check{\epsilon}_A$ is a closed immersion of finite presentation by \cite[IV$_1$, Proposition~1.6.2]{EGA}, and the same is true for $U(G)\mal \hookrightarrow  U(G)\times U(G)$. Therefore  $U(G)\mal$ is an affine group scheme of finite presentation over $S$. By \cite[IV$_4$, Proposition~19.3.7]{EGA} the unit section of the smooth $S$-scheme $U(G)$ is a regular immersion. Therefore the immersion $U(G)\mal \hookrightarrow  U(G)\times U(G)$ is also regular by \cite[IV$_4$, Proposition~19.1.5]{EGA}.

$U(G)\mal$ is smooth over $S$ because $U(G)$ is and the inclusion $U(G)\mal\to U(G)$ is a smooth monomorphism (and hence an open immersion by \cite[IV$_4$, Th\'eor\`eme~17.9.1]{EGA}). Indeed, smoothness can be tested by the infinitesimal lifting criterion \cite[\S\,2.2, Proposition~6]{BLR} as follows. Let $I\subset B$ be an ideal in a ring $B$ with $I^2=0$ and let $b\in \check{A}\otimes_R B=U(G)(B)$ be a point with $(b\mod I)\in U(G)\mal(B/I)$. Then any lift $b'\in\check{A}\otimes_R B$ of $(b\mod I)^{-1}\in U(G)\mal(B/I)\subset \check{A}\otimes_R B/I$ satisfies $bb'-1\in \check{A}\otimes_R I$ and $0=(bb'-1)^2=1-bb'(2-bb')$. It follows that $b\in U(G)\mal(B)$.

\medskip

Messing considers the natural monomorphism $i\colon G \hookrightarrow  U(G)\mal$ which is defined by viewing a $T$-valued point of $G$ as a homomorphism of $\mathcal{O}_T$-algebras $A\otimes_R \mathcal{O}_T\to \mathcal{O}_T$ and hence as an element of $\Gamma(T,\ \check{A} \otimes_R \mathcal{O}_T)$. The fact that such a homomorphism when viewed as an element of $U(G)(T)$ is invertible, follows from the commutativity of the following diagram
\begin{equation*}  
\xymatrix @C+1.5pc  {        {G(T)\times G(T)} \ar[r]^{\TS i\times i \ \ \ \ \ \ \ \ \  \ \ \ \ \ \ } \ar[d]^{\TS\Spec\Delta}  & \enspace \Gamma(T,\ \check{A} \otimes_{\mathcal{O}_S} \mathcal{O}_T)\times  \Gamma(T,\ \check{A} \otimes_{\mathcal{O}_S} \mathcal{O}_T) \ar[d]^{\TS\check{\Delta}} \\
                          G(T)\ar[r]^{\TS i}&  \Gamma(T,\ \check{A} \otimes_{\mathcal{O}_S} \mathcal{O}_T)}
\end{equation*}
This diagram is commutative because both the left and right vertical arrow come from $\Delta$. For every $f \in G(T)$ there exists a $g \in G(T)$ such that $(\Spec\Delta)(f,g) = 1$. Since  $i(f)\cdot i(g) = \check{\Delta}\circ(i \times i)(f,g) = i\circ(\Spec\Delta)(f,g) = i(1)$, it is enough to prove that $i(1)= 1$. Now $1 \in G(T)$ is equivalent to  $\epsilon_A\colon A \otimes_{\mathcal{O}_S} \mathcal{O}_T \to \mathcal{O}_T $ which in turn is equivalent to  ($\check{\epsilon}_A \colon \mathcal{O}_T \to \check{A} \otimes_{\mathcal{O}_S} \mathcal{O}_T) \equiv 1\in \Gamma(T,\ \check{A} \otimes_{\mathcal{O}_S} \mathcal{O}_T)$. This shows $i(1)= 1$. Also it shows that the morphism  $i\colon G \hookrightarrow  U(G)\mal$ is a homomorphism of group schemes. Since $G$ is a relative complete intersection and finite over $S$, it follows that the monomorphism $i\colon G \hookrightarrow  U(G)\mal$ is a regular closed immersion; see~\cite[Chapter~II, Lemmas~3.2.5 and 3.2.6]{Messing}. 
Let $J$ be the ideal defining $G$ in $U(G)\mal$. Then Messing \cite[Chapter~II, Definition~3.2.8]{Messing} defines the cotangent complex of $G$ over $S$ as the complex of ${\mathcal{O}_G}$-modules concentrated in degrees $-1$ and $0$ 
\[
\CoCM{G/S}\colon \es 0\longto J/{J^2} \longto i^*(\Omega_ {U(G)\mal\!/S})\longto 0\,.
\]

\begin{Proposition}\label{PropLS-M}
The cotangent complexes $\CoCLS{G/S}$ and $\CoCM{G/S}$ are homotopically equivalent.
\end{Proposition}

\begin{proof}
The scheme $U(G)\times U(G)$ is locally on $S$ of the form $U(G)\times U(G) = \Spec R[\ulX]$ for a polynomial algebra and we can form the cotangent complex $\CoCLS{G/S}$ using $R[\ulX]$. We set $U(G)\mal = \Spec\olR$. Let $I$ be the ideal defining  $G$ in  $U(G)\times U(G)$ and let $K$ be the ideal defining $U(G)\mal$ in $U(G)\times U(G)$. Then $J=I/K$. By \cite[IV$_4$, Proposition~19.1.5]{EGA} 
the composition $G \xrightarrow{\es i\;} U(G)\mal \longto U(G)\times U(G)$ is a regular closed immersion and the canonical sequence 
\[
0 \longto i^\ast (K/K^2)  \longto I/{I^2}  \xrightarrow{\es g^{(-1)}\;} J/{J^2}  \longto 0
\]
is also exact on the left. We have denoted the third map by $g^{(-1)}$. Since $U(G)\mal$ is smooth over $S$, \cite[IV$_4$, Proposition~17.2.5]{EGA} yields an exact sequence of finite locally free $\olR$-modules 
\[
 0 \longto K/K^2 \longto \Omega^1_{R[\ulX]/R}\otimes_{R[\ulX]} \olR \longto \Omega^1_{\olR/R} \longto 0
\]
which we tensor with $A$ to get an exact sequence of $A$-modules
\[
 0 \longto i^\ast (K/K^2) \longto  \Omega^1_{R[\ulX]/R}\otimes_{R[\ulX]} A \xrightarrow{\es g^{(0)}\;} \Omega^1_{\olR/R}\otimes_{\olR} A \longto 0. 
\]
We have denoted the third map by $g^{(0)}$. Since $\Omega^1_{\olR/R}\otimes_{\olR}A$ is a finite locally free $A$-module we may choose a section $f^{(0)}$ of $g^{(0)}$. In the diagram with commuting solid arrows
\[ 
\xymatrix @C+2pc  { 
& 0\ar[d] & \quad 0\ar@<1ex>[d] \\
& i^\ast (K/K^2)\ar[d]^{\alpha^{(-1)}} \ar@{=}[r]& i^\ast (K/K^2)\ar@<1ex>[d]^{\alpha^{(0)}} \\
 0 \ar[r] & I/{I^2} \ar[r]_{d^{(-1)}\qquad} \ar[d]^{g^{(-1)}} \ar@<1ex>@{..>}[u]^{s^{(-1)}}  & \Omega^1_ {R[\ulX]/R} \otimes_{R[\ulX]} A \ar[r] \ar@{..>}[u]^{s^{(0)}} \ar@<1ex>[d]^{g^{(0)}}\ar@/_1.5pc/@{..>}[l]_{h^{(-1)}} & 0 \\
0 \ar[r] & J/{J^2} \ar[d] \ar[r]_{\tilde d^{(-1)}\qquad} \ar@<1ex>@{..>}[u]^{f^{(-1)}} & \Omega^1_ {\olR/R}\otimes_{\olR}A \ar@<1ex>[d] \ar[r] \ar@{..>}[u]^{f^{(0)}} & 0 \\
& 0 & \quad 0 
}
\]
we define the section $s^{(0)}$ of $\alpha^{(0)}$ by $\id-f^{(0)}g^{(0)}=\alpha^{(0)}s^{(0)}$. Then $s^{(-1)}:=s^{(0)}d^{(-1)}$ satisfies $s^{(-1)}\alpha^{(-1)}=s^{(0)}d^{(-1)}\alpha^{(-1)}=s^{(0)}\alpha^{(0)}=\id_{i^\ast(K/K^2)}$. We define the section $f^{(-1)}$ of $g^{(-1)}$ by $\id-\alpha^{(-1)}s^{(-1)}=f^{(-1)}g^{(-1)}$. Then $d^{(-1)}f^{(-1)}g^{(-1)}=d^{(-1)}(\id-\alpha^{(-1)}s^{(-1)})=d^{(-1)}-\alpha^{(0)}s^{(0)}d^{(-1)}=f^{(0)}g^{(0)}d^{(-1)}=f^{(0)}\tilde d^{(-1)}g^{(-1)}$ and hence $d^{(-1)}f^{(-1)}=f^{(0)}\tilde d^{(-1)}$. This means that we obtain homomorphisms of complexes
\[ 
\xymatrix{ \CoCLS{G/S}\colon \ar[d]^g & 0 \ar[r] & I/{I^2} \ar[r]\ar[d]^{g^{(-1)}}    & \Omega^1_ {R[\ulX]/R} \otimes_{R[\ulX]} A \ar[r]\ar[d]^{g^{(0)}} & 0 \\
\CoCM{G/S}\colon \ar@<1ex>[u]^f & 0 \ar[r] & J/{J^2} \ar[r]\ar@<1ex>[u]^{f^{(-1)}} & \Omega^1_ {\olR/R}\otimes_{\olR}A \ar[r]\ar@<1ex>[u]^{f^{(0)}} & 0
}
\]
with $gf=\id$. We define the homotopy $h^{(-1)}:=\alpha^{(-1)}s^{(0)}$. Then
\[
\begin{array}{rcccccl}
\id-f^{(0)}g^{(0)}&=&\alpha^{(0)}s^{(0)}&=&d^{(-1)}\alpha^{(-1)}s^{(0)}&=&d^{(-1)}h^{(-1)} \qquad \text{and} \\[1mm]
\id-f^{(-1)}g^{(-1)}&=&\alpha^{(-1)}s^{(-1)}&=&\alpha^{(-1)}s^{(0)}d^{(-1)}&=&h^{(-1)}d^{(-1)}\,.
\end{array}
\]
This proves that $f$ and $g$ form a homotopy equivalence between $\CoCLS{G/S}$ and $\CoCM{G/S}$.
\end{proof}
\end{appendix}


\vfill

\begin{minipage}[t]{0.5\linewidth}
\noindent
Urs Hartl\\
Universit\"at M\"unster\\
Mathematisches Institut \\
Einsteinstr.~62\\
D -- 48149 M\"unster
\\ Germany
\\[1mm]
\href{http://www.math.uni-muenster.de/u/urs.hartl/index.html.en}{www.math.uni-muenster.de/u/urs.hartl/}
\end{minipage}
\begin{minipage}[t]{0.45\linewidth}
\noindent
Rajneesh Kumar Singh\\
Ramakrishna Vivekananda University\\
PO Belur Math, Dist Howrah \\
711202, West Bengal\\
India
\\[1mm]
\end{minipage}

\end{document}